\documentclass[a4paper,10pt]{amsart}

\usepackage{amsmath,amsthm,amssymb,enumerate}
\usepackage{mathtools}
\usepackage{stmaryrd}
\usepackage{epsfig}
\usepackage{amssymb}
\usepackage{amsmath}
\usepackage{amssymb}
\usepackage{amsmath,amsthm}
\usepackage{eqparbox}
\usepackage[latin1]{inputenc}
\usepackage{bm}
\usepackage{bbm}
\usepackage{esint}
\usepackage{tikz-cd}
\usetikzlibrary{decorations.pathreplacing}
\usepackage{amsfonts}
\usepackage{amsxtra}
\usepackage{euscript,mathrsfs}
\usepackage{color}
\usepackage{verbatim}
\usepackage[left=3.25cm,right=3.25cm,top=3cm,bottom=2.25cm]{geometry}
\usepackage[colorlinks=true, linktocpage=true, linkcolor=red!70!black, citecolor=green!50!black, urlcolor=blue!50!black]{hyperref}
\allowdisplaybreaks

\usepackage{tikz-cd}
\usetikzlibrary{decorations.pathreplacing}

\usepackage{enumitem}
\setenumerate{label={\rm (\alph{*})}}

\usepackage{amsfonts}
\usepackage{amsxtra}

\numberwithin{equation}{section}

\usepackage{xcolor}
\colorlet{ColorPink}{red!30}
\usepackage{graphicx}

\newcommand{\R}{\mathbb R}
\newcommand{\N}{\mathbb N}

\newcommand{\Q}{Q}

\renewcommand\ae{{a.\@e.\@}}
\newcommand\D{{\rm D}}

\newcommand\qq{\qquad}

\newcommand{\dx}{\,\mathrm{d}x}

\DeclareMathOperator{\BV}{BV}

\newcommand{\dif}{\mathrm{d}}

\renewcommand{\dif}{\operatorname{d}\!}
\newcommand{\lebe}{\operatorname{L}}
\newcommand{\sobo}{\operatorname{W}}

\newcommand{\locc}{\operatorname{loc}}

\newcommand{\hold}{\operatorname{C}}

\newcommand{\bv}{\operatorname{BV}}

\newcommand{\ball}{\operatorname{B}}

\newcommand{\mres}{\mathbin{\vrule height 1.6ex depth 0pt width
0.13ex\vrule height 0.13ex depth 0pt width 1.3ex}}

\newcommand{\dashint}{\fint}
\newcommand{\spt}{\operatorname{spt}}

\allowdisplaybreaks

\renewcommand\qq\qquad

\newcommand{\Cstabv}{S_j}
\newcommand{\CstabG}{A_j}

\theoremstyle{plain}

\newtheorem*{theorem*}{Theorem}
\newtheorem{theorem}{Theorem}[section]
\newtheorem{lem}[theorem]{Lemma}
\newtheorem{prop}[theorem]{Proposition}
\newtheorem{corollary}[theorem]{Corollary}

\newtheorem{definition}[theorem]{Definition}

\theoremstyle{definition}

\newtheorem{rem}[theorem]{Remark}
\newtheorem{example}[theorem]{Example}

\makeatletter
\def\namedlabel#1#2{\begingroup
    #2%
    \def\@currentlabel{#2}%
    \phantomsection\label{#1}\endgroup
}
\makeatother

\begin{document}


\title[Non-uniformly elliptic variational problems on BV]{Non-uniformly elliptic \\ variational problems on BV}
\author[L. Beck]{Lisa Beck}
\address{L. Beck: Institut f\"{u}r Mathematik, Universit\"{a}t Augsburg, Universit\"{a}tsstr. 12a, 86159 Augsburg, Germany} 
\email{lisa.beck@math.uni-augsburg.de}

\author[F. Gmeineder]{Franz Gmeineder}
\address{F. Gmeineder: Fachbereich Mathematik und Statistik, Universit\"{a}t Konstanz, Universit\"{a}tsstra\ss e 10, 78464 Konstanz, Germany} 
\email{franz.gmeineder@uni-konstanz.de}

\author[M. Sch\"{a}ffner]{Mathias Sch\"{a}ffner}
\address{M. Sch\"{a}ffner: Institut f\"{u}r Mathematik, Theodor Lieser Strasse 5,44227 Halle (Saale), Germany}
\email{mathias.schaeffner@mathematik.uni-halle.de}


\date{\today}


\maketitle

\begin{abstract}
We establish $\sobo^{1,1}$-regularity and higher gradient integrability for relaxed minimizers of convex integral functionals on $\bv$. Unlike classical examples such as the minimal surface integrand, we only require linear growth from below but not necessarily from above. This typically comes with a non-uniformly degenerate elliptic behaviour, for which our results extend the presently available bounds from the superlinear growth  case in a sharp way. 
\end{abstract}
\setcounter{tocdepth}{1}
\tableofcontents

\section{Introduction}
Functionals with linear growth are one of the central topics in the Calculus of Variations, comprising the classical (non-parametric) minimal surface problem as a special case. Compared with more classical settings -- e.g., ~elastic energies in reflexive Sobolev spaces -- they come with a crucial lack of compactness and potential concentration effects. This  eventually leads to relaxed formulations on the space~$\bv$ of functions of bounded variation \cite{AMBFUSPAL00,GIAMODSOU79,GIUSTI77}, and the underlying regularity theory needs to take care of the generic degenerate elliptic behaviour of the integrands. Fundamentally different from the Sobolev case, a primary objective here is the absence of the singular parts and so the $\sobo^{1,1}$-regularity of minima; see Section~\ref{sec:contextmainresults} for more detail. 

While conditions for Sobolev regularity in the realm of linear growth functionals are fairly well understood \cite{BECBULGME20,BECEITGME,BECSCH13,BECSCH15,BILDHAUER03,GMEKRI19,GMEINEDER20}, much less is known for \emph{non-uniformly elliptic} problems on $\bv$. Deferring the precise meaning of this notion to our discussion below, such functionals represent the endpoint case of non-uniformly elliptic superlinear growth problems: With origins in the works of Marcellini \cite{MARCELLINI89,MARCELLINI91}, the latter -- also known as $(p,q)$-growth functionals -- now face an abundance of criteria leading to improved regularity assertions; see \cite{BELSCH20,BELSCH24,CarozzaKristensenPasarelli,ESPLEOMIN99,ESPLEOMIN04} for a non-exhaustive list, and De Filippis \& Mingione \cite{DEFMIN25,MINGIONE06Dark} for overviews. Yet, the borderline case $p=1$ is widely open so far and comes with more fundamental obstructions than in superlinear growth scenarios. In this paper, we aim to give the first results in this limiting case which bridge between the two growth regimes in a sharp way. Referring the reader to Section~\ref{sec:contextmainresults} below for the precise statements, we give our detailed set-up first. 

\subsection{Setting and LSM-extensions}\label{sec:introsetting}
In all of what follows, let $\Omega\subset\R^{n}$ be an open and bounded set with Lipschitz boundary. As the central objects of the present paper, we consider convex variational integrands $F\in\hold(\R^{N\times n})$ which satisfy the growth bound
\begin{align}\label{eq:1qgrowth}
\gamma|z| \leq F(z) \leq \Gamma(1+|z|^{q})\qquad\text{for all}\;z\in\R^{N\times n}, 
\end{align}
where $0<\gamma\leq \Gamma <\infty$ and  $1\leq q<\infty$. For future reference, we note that the choice $q=1$ in \eqref{eq:1qgrowth} corresponds to \emph{classical linear growth} and is, for instance, fulfilled by the area integrand $F(z)\coloneqq \sqrt{1+|z|^{2}}$.

Henceforth, let $u_{0}\in\sobo^{1,q}(\Omega;\R^{N})$ be a Dirichlet datum. In view of the variational problem 
\begin{align}\label{eq:minimize}
\text{to minimize}\;\;\;\mathscr{F}[u;\Omega] \coloneqq  \int_{\Omega}F(\nabla u)\dif x \qquad\text{over}\;\sobo_{u_{0}}^{1,q}(\Omega;\R^{N})\coloneqq u_{0}+\sobo_{0}^{1,q}(\Omega;\R^{N}), 
\end{align}
we deduce from the upper growth bound of \eqref{eq:1qgrowth} that $\mathscr{F}[-;\Omega]$ is well-defined on the Dirichlet class $\sobo_{u_{0}}^{1,q}(\Omega;\R^{N})$. However, by the lower growth bound from \eqref{eq:1qgrowth}, minimizing sequences will, in general, only be bounded in $\sobo^{1,1}(\Omega;\R^{N})$. The latter space is non-reflexive, and concentration effects might prevent minimizing sequences from being relatively weakly compact in $\sobo^{1,1}(\Omega;\R^{N})$ indeed. As a classical procedure, compactness can be achieved by passing to the larger space $\bv(\Omega;\R^{N})$ of \emph{functions of bounded variation}. In consequence, since $\mathscr{F}[-;\Omega]$ is a priori only well-defined on $\sobo^{1,q}(\Omega;\R^{N})$, this necessitates a suitable extension to $\bv(\Omega;\R^{N})$. 

Keeping in mind the direct method of the Calculus of Variations, such extensions ought to be lower semicontinuous with respect to the convergence yielding compactness; in our situation, this is the weak*-convergence on $\bv(\Omega;\R^{N})$ (see Section~\ref{sec:prelims} for the precise definition). This is accomplished by the  Lebesgue--Serrin--Marcellini (LSM) extension
\begin{align}\label{eq:LSM}
\overline{\mathscr{F}}_{u_{0}}^{*}[u;\Omega] \coloneqq \inf\bigg\{\liminf_{j\to\infty}\int_{\Omega}F(\nabla u_{j})\dif x\colon\;\begin{matrix} (u_{j}) \text{ in } u_{0}+\sobo_{0}^{1,q}(\Omega;\R^{N}) \\ u_{j}\stackrel{*}{\rightharpoonup} u\;\text{in}\;\bv(\Omega;\R^{N})\end{matrix}\bigg\}
\end{align}
for $u\in\bv(\Omega;\R^{N})$. Originally introduced in \cite{MARCELLINI86} in the semiconvex context, LSM extensions have been considered recently by De Filippis et al. \cite{DEFKOCKRI24,DEFMIN25b} in the study of convex functionals of $(p,q)$-growth with $p>1$. However, in the $\bv$-situation as discussed here, they usually cannot be represented by integrals or measures by easy means if $q>1$, and come with both benefits and drawbacks; see Section \ref{sec:contextmainresults}\ref{item:(a)} below  for a discussion. For now, we note that sequences as required in \eqref{eq:LSM} always exist in our setting (see Lemma~\ref{lem:nonempty}), whereby $\overline{\mathscr{F}}_{u_{0}}^{*}[-;\Omega]$ is well-defined on $\bv(\Omega;\R^{N})$. Crucially, the Dirichlet datum is directly included in the definition of the relaxed functional. This feature might be anticipated from classical relaxations in the linear growth context, where  it  corresponds to the usual solid-boundary-value approach. In particular, it causes the emergence of boundary penalisation terms; see also the integral representation~\eqref{eq:relaxintegral} in the case $q=1$ below.   

In the following, we call $u\in\bv(\Omega;\R^{N})$ a \emph{relaxed} or \emph{$\bv$-minimizer} of~$\mathscr{F}$ subject to the Dirichlet datum $u_{0}$ provided that 
\begin{align}\label{eq:minimality}
\overline{\mathscr{F}}_{u_{0}}^{*}[u;\Omega] \leq \overline{\mathscr{F}}_{u_{0}}^{*}[v;\Omega]\qquad\text{for all}\;v\in\bv(\Omega;\R^{N}). 
\end{align}
We note that in \eqref{eq:minimality}, the competitors are allowed to be arbitrary in $\bv(\Omega;\R^{N})$. As a key point, Theorem~\ref{thm:consistency} below establishes that~\eqref{eq:LSM} is in fact the correct extension for the minimization problem~\eqref{eq:minimize}. Namely, subject to the growth condition~\eqref{eq:1qgrowth} and Dirichlet data  $u_{0}\in\sobo^{1,q}(\Omega;\R^{N})$, the original functional $\mathscr{F}[-;\Omega]$ coincides with the relaxed functional $\overline{\mathscr{F}}_{u_{0}}^{*}[-;\Omega]$ on the Dirichlet class $\sobo_{u_{0}}^{1,q}(\Omega;\R^{N})$, relaxed minimizers of~$\mathscr{F}$ always exist, and we have the no-gap-result 
\begin{align}\label{eq:nogap!}
\min_{\bv(\Omega;\R^{N})}\overline{\mathscr{F}}_{u_{0}}^{*}[-;\Omega] = \inf_{\sobo_{u_{0}}^{1,q}(\Omega;\R^{N})}\mathscr{F}[-;\Omega].
\end{align}
Interestingly, minimality of some  $u\in\bv(\Omega;\R^{N})$ for the relaxed functional directly implies strong structural features on the densities $\frac{\dif\D^{s}u}{\dif|\D^{s}u|}$ of the singular parts, see Proposition~\ref{prop:finiteness}:  The finiteness of the relaxed functional excludes that these densities are contained in  parts of the rank-one-cone of directions where $F$ has superlinear growth. 

In general, by potential concentration effects of minimizing sequences, the absolutely continuous and singular parts of gradients of relaxed minimizers are intertwined in a delicate manner. It is thus natural to inquire as to which conditions on the integrand~$F$ imply the \emph{complete absence} of the singular parts $\D^{s}u$ and, if possible, higher gradient integrability.
\subsection{Main results, context  and strategy of proof}\label{sec:contextmainresults}  Regularity assertions of this kind necessarily rely on some form of ellipticity of $F$. In this regard, a flexible scale is that of \emph{$(\mu,q)$-ellipticity}, letting us grasp the typical degenerate elliptic behaviour of $\hold^{2}$-integrands verifying \eqref{eq:muell1q}. With origins in \cite{BILFUC01,BILDHAUERFUCHS2002MU,BILDHAUER03}, we say that $F\in\hold^{2}(\R^{N\times n})$ is \emph{$(\mu,q)$-elliptic} with $1\leq\mu<\infty$ and $1\leq q<\infty$ if there exist constants $0<\lambda\leq\Lambda<\infty$ such that  
\begin{align}\label{eq:muell1q}
\lambda (1+|z|^{2})^{-\frac{\mu}{2}}|\xi|^{2} \leq \langle \nabla^{2}F(z)\xi,\xi\rangle \leq \Lambda (1+|z|^{2})^{\frac{q-2}{2}}|\xi|^2\qquad\text{for all}\;z,\xi\in\R^{N\times n}.
\end{align}
In the classical linear growth context, \eqref{eq:muell1q} is typically satisfied for $1< \mu <\infty$ and $q=1$, and implies that the corresponding ellipticity ratio  
\begin{align}\label{eq:ellratio}
\mathscr{R}_{F}(z)\coloneqq \frac{\text{highest eigenvalue of $\nabla^{2}F(z)$}}{\text{lowest eigenvalue of $\nabla^{2}F(z)$}}
\end{align}
blows up as $|z|\to\infty$. Thus, in the terminology of De Filippis \& Mingione \cite{DEFMIN25}, even classical linear growth $\hold^{2}$-integrands are \emph{non-uniformly elliptic}. In this sense, integrands obeying \eqref{eq:muell1q} can be regarded as \emph{very} non-uniformly elliptic. It is then natural to examine how far the ellipticity ratio can be deteriorated while still maintaining the $\sobo_{\locc}^{1,1}$-regularity of relaxed minimizers \emph{and} connecting this borderline case with the available results for classical linear and $(p,q)$-growth in a sharp way. This is answered by
\begin{theorem}[Universal higher gradient integrability]\label{thm:main}
Let $\Omega\subset\R^{n}$ be open and bounded with Lipschitz boundary. Moreover, let $F\in\hold^{2}(\R^{N\times n})$ be a variational integrand with \eqref{eq:1qgrowth} and \eqref{eq:muell1q}, where $1\leq\mu,q<\infty$ satisfy
\begin{align}\label{eq:maincondition}
q + \mu < 2 + \frac{2}{n-1} \;\;\;\text{and}\;\;\;\begin{cases}
1\leq \mu < 1+\frac{2}{n}&\;\text{if}\;n\geq 3, \\
1\leq \mu\leq 2&\;\text{if}\;n=2. 
\end{cases}
\end{align}
Then, for any $u_{0}\in\sobo^{1,q}(\Omega;\R^{N})$, \emph{every  relaxed minimizer $u\in\bv(\Omega;\R^{N})$ of $\mathscr{F}$ subject to the Dirichlet datum $u_{0}$ belongs to $\sobo^{1,1}(\Omega;\R^{N})$}. 

More precisely, there exist a constant {$c=c(\gamma,\Gamma,\lambda,\Lambda,n,q,\mu)\in[1,\infty)$} and an exponent  $\mathtt{d}=\mathtt{d}(n,\mu,q)\in[1,\infty)$ such that the following estimates hold for every such $u\in\bv(\Omega;\R^{N})$ and all  balls $\ball_{R}(x_{0})\Subset\Omega$: 
\begin{enumerate}
    \item\label{item:main2} If $1\leq\mu<1+\frac{2}{n}$  {and $n\geq 3$}, then  \begin{align}\label{eq:finalboundA1}
\begin{split}
\bigg(\dashint_{\ball_{R/2}(x_0)}|\nabla u|^{\frac{(2-\mu)n}{n-2}}\dif x\bigg)^\frac{n-2}{(2-\mu)n}   \leq c\,\bigg(1+\bigg(\dashint_{{\ball_{R}(x_0)}}|\nabla u|\dif x \bigg)^{\mathtt{d}}\bigg).
\end{split}
\end{align}
    \item\label{item:main2b} If $1\leq \mu<2$ and $n=2$, then 
\begin{align}\label{eq:finalboundA2}
\begin{split}
{\;\;-}{\!\!\!\|\,|\nabla u|^{2-\mu}\|}_{\exp\lebe^{1}(\ball_{R/2}(x_{0}))}^{\frac{1}{2-\mu}}
 &   \leq c\,\bigg(1+\bigg(\dashint_{{\ball_{R}(x_0)}}|\nabla u|\dif x \bigg)^{\mathtt{d}}\bigg),
\end{split}
\end{align}
where $\!\!\!{\;\;-}{\!\!\!\|\cdot\|}_{\exp\lebe^{1}(\ball_{R/2}(x_{0}))}$ is the scaled Orlicz norm with respect to the Young function $t\mapsto\exp(|t|)-1$; see Section~\ref{sec:notation} for its precise definition. 
\item\label{item:main3} If $n=2$ and $\mu=2$, then we have for every $1\leq t <\infty$ that
\begin{align}\label{eq:finalboundA1log}
\begin{split}
\bigg(\dashint_{\ball_{R/2}(x_0)}|\nabla u|^{t}\dif x\bigg)^\frac{1}{t} &   \leq \exp\bigg(c\,t\,\bigg(1+\bigg(\dashint_{\ball_{R}(x_0)}|\nabla u|\dif x \bigg)^{\mathtt{d}}\bigg)\bigg).
\end{split}
\end{align}
\end{enumerate}
\end{theorem}
Note that, if one writes the left-hand side of \eqref{eq:finalboundA1} as 
\begin{align*}
{\;\;-}{\!\!\!\|\,|\nabla u|{^{2-\mu}}\|}_{\lebe^{\frac{n}{n-2}}(\ball_{R/2}(x_{0}))}^{\frac{1}{2-\mu}} = \bigg(\dashint_{\ball_{R/2}(x_0)}|\nabla u|^{\frac{(2-\mu)n}{n-2}}\dif x\bigg)^\frac{n-2}{(2-\mu)n}, 
\end{align*}
then \eqref{eq:finalboundA2} can be understood as the sharp substitute of \eqref{eq:finalboundA1} for $n=2$. In the two-dimensional case and subject to the stronger condition \eqref{eq:2dstrongercond} below, the previous result can be sharpened to yield the local $\hold^{1,\alpha}$-regularity of relaxed minimizers: 
\begin{theorem}[Universal  $\hold^{1,\alpha}$-regularity in $n=2$ dimensions]\label{thm:main2}
Let $\Omega\subset\R^{2}$ be open and bounded with Lipschitz boundary, and let $F\in\hold^{2}(\R^{N\times 2})$  be a variational integrand satisfying \eqref{eq:1qgrowth} and \eqref{eq:muell1q}, where 
\begin{align}\label{eq:2dstrongercond}
\max\{2,q\}+3\mu < 6. 
\end{align}
Moreover, let $u_{0}\in\sobo^{1,q}(\Omega;\R^{N})$. Then, for any $0<\alpha<1$, every relaxed minimizer $u\in\bv(\Omega;\R^{N})$ of $\mathscr{F}$ subject to the Dirichlet datum $u_{0}$ belongs to $(\sobo^{1,1}\cap\hold_{\locc}^{1,\alpha})(\Omega;\R^{N})$. 
\end{theorem}

We now briefly comment on the proof of Theorems~\ref{thm:main} and~\ref{thm:main2}. To this end, it is instructive to 
firstly recall the available results in the purely linear growth case, meaning that $q=1$ in the upper bound in~\eqref{eq:1qgrowth}. Owing to convexity of $F$, which gives us access to Reshetnyak's lower semicontinuity theorem \cite{RESHETNYAK68}, one then has the classical integral representation 
\begin{align}\label{eq:relaxintegral}
\begin{split}
\overline{\mathscr{F}}_{u_{0}}^{*}[u;\Omega] & = \int_{\Omega}F(\nabla u)\dif x   + \int_{\Omega}F^{\infty}\Big(\frac{\dif\D^{s}u}{\dif|\D^{s} u|}\Big)\dif|\D^{s}u|  \\ & \;\;\;\;\;\;\;\;\;\;\;\; \;\;\;\;\;\;\;\; + \int_{\partial\Omega}F^{\infty}(\mathrm{tr}_{\partial\Omega}(u_{0}-u)\otimes\nu_{\partial\Omega})\dif\mathscr{H}^{n-1}
\end{split}
\end{align}
for $u\in\bv(\Omega;\R^{N})$. 
In \eqref{eq:relaxintegral},  $\D u = \D^{a}u + \D^{s}u = \nabla u\mathscr{L}^{n}+\frac{\dif\D^{s}u}{\dif|\D^{s}u|}|\D^{s}u|$ is the Lebesgue-Radon-Nikod\'{y}m decomposition of the measure derivative~$\D u$, $\mathrm{tr}_{\partial\Omega}\colon\bv(\Omega;\R^{N})\to\lebe^{1}(\partial\Omega;\R^{N})$ is the boundary trace operator on $\bv(\Omega;\R^{N})$, and $\nu_{\partial\Omega}\colon\partial\Omega\to\mathbb{S}^{n-1}$ is the outer unit normal to $\partial\Omega$; see Section~\ref{sec:funcspaces} for more detail. 

In this setting and moreover assuming \eqref{eq:muell1q} with $q=1$, the first $\sobo^{1,1}$-regularity results are due to Bildhauer \cite{BILDHAUER2002MUBOUNDED,BILDHAUER03}. More precisely, if $1<\mu<1+\frac{2}{n}$, then one relaxed minimizer belongs to $\sobo^{1,1}(\Omega;\R^{N})$ and, if $1<\mu\leq 3$, one locally bounded relaxed minimizer belongs to $\sobo^{1,1}(\Omega;\R^{N})$. The first author and Schmidt \cite{BECSCH13} extended these results to \emph{any} relaxed minimizer and not only a specific one obtained by pure viscosity approximations. To underline this point, assume for simplicity that $u_{0}\in\sobo^{1,2}(\Omega;\R^{N})$; then one may consider viscosity stabilizations 
\begin{align}\label{eq:viscstabs}
\mathscr{F}_{j}[v;\Omega] \coloneqq  \int_{\Omega}F(\nabla v)\dif x + \frac{1}{j}\int_{\Omega}|\nabla v|^{2}\dif x, \qquad v\in\sobo_{u_{0}}^{1,2}(\Omega;\R^{N}).  
\end{align}
Denoting by $v_{j}$ the unique minimizer of \eqref{eq:viscstabs}, the strategy of \cite{BILDHAUER2002MUBOUNDED,BILDHAUER03} is centered around establishing uniform higher integrability estimates on the functions $v_{j}$. The sequence $(v_{j})$, in turn, is shown to converge to \emph{some} minimizer of \eqref{eq:relaxintegral}. However, since $F^{\infty}$ is positively $1$-homogeneous (whereby $F^{\infty}$ is not strictly convex) and acts on a different part of $\D u$ than the one governed by the strictly convex function $F$, the relaxed functional~\eqref{eq:relaxintegral} fails to be strictly convex. As a consequence, minimizers of \eqref{eq:relaxintegral} might be non-unique; see \cite{BECSCH13,SANTI72}. Thus, even though the minimizer found by the approach following \eqref{eq:viscstabs} is $\sobo^{1,1}$-regular, this does not rule out the existence of other, more irregular minimizers. This issue can be circumvented by an approximation strategy based on the Ekeland variational principle \cite{EKELAND74}, see \cite{BECSCH13} for its first use in the context of linear growth functionals and \cite{BECEITGME,EITLERLEWINTAN25,FUSPRISTE25,GMEINEDER16,GMEKRI19,GMEINEDER20,SCHMIDTHabil,WOZNIAK23} for related implementations. Based on this discussion, we now highlight some key points and novelties in the overall proof.\\ 
\vspace{-0.25cm}

\namedlabel{item:(a)}{(a)} \textbf{LSM-relaxations.} To the best of our knowledge, the present paper seems to be the first Sobolev regularity contribution in the vectorial BV-context which completely avoids integral representations. All previous approaches in the purely linear growth situation rely at some point on the representation \eqref{eq:relaxintegral} and  refined continuity properties for e.g. the area-strict topology. From a lower semicontinuity viewpoint, \eqref{eq:LSM} is indeed the canonical extension of the functional $\mathscr{F}[-;\Omega]$ to $\bv(\Omega;\R^{N})$, comes with no Lavrentiev gap (see Theorem~\ref{thm:consistency}) but is also technically favourable: The existence of recovery sequences is relatively easy to establish (see Lemma~\ref{lem:recovery}), letting us avoid the detour over the integral representation \eqref{eq:relaxintegral} and the continuity part of Reshetnyak's theorem; see Remark~\ref{rem:recovery} for more detail. However, in the exponent regime considered here, the requisite integral representations are currently not available anyway. This issue is essentially due to the appearance of the Dirichlet data in \eqref{eq:LSM}, and is briefly addressed for the reader's convenience in the Appendix, Section~\ref{sec:appendix}. As one of the key points, however, the present paper furnishes the metaprinciple that \emph{integral representations are not required} for Sobolev regularity assertions, \emph{even for classical linear growth problems $q=1$} as a special case. Working exclusively with \eqref{eq:LSM}, however, comes to the effect that minimality itself is hard to be localised in general and if so, such local minimality results or localised estimates as \eqref{eq:finalboundA1} or \eqref{eq:finalboundA2} are a \emph{consequence} of regularity statements; see,  e.g., Remark~\ref{rem:convex} and Corollary~\ref{cor:locmin}. \\   
\vspace{-0.25cm}

(b) \textbf{Universal gradient estimates: Ekeland approximations versus gaps.} In our situation, which includes classical linear growth as a special case, the ghost of non-uniqueness still persists (apart from very special situations, see, e.g. \cite{BOUSQUETLLEDOS25}). Following the discussion after \eqref{eq:viscstabs}, this necessitates suitable Ekeland-type viscosity approximations $u_{j}$ based on the perturbation space $\sobo^{-1,1}$, see Section~\ref{sec:negsob} for  the latter. This approximation step is crucial in order to arrive at \emph{universal} regularity results, that is, for \emph{all} relaxed minimizers. However, despite being finely adjusted to the underlying scenarios in the purely linear growth context, previous approaches as developed in \cite{BECEITGME,BECSCH13,FUSPRISTE25,GMEKRI19,GMEINEDER20,SCHMIDTHabil} do not suffice here. In essence, this is due to the \emph{Lipschitz-type} (but not Lipschitz) estimate from Lemma~\ref{lem:liptype}. Without further refinement, this estimate entails that the auxiliary approximating problems might become \emph{decoupled} from the original ones during the Ekeland approximation procedure. Indeed, on a more technical level, a straightforward adaptation of the methods established in \cite{BECSCH13,BECEITGME,GMEINEDER20} might lead to intermediate Lavrentiev gaps during the approximation process. Because of $q>1$, however, a direct $q$-growth uniformization of approximating integrands such as, e.g., in \cite{ESPLEOMIN99,CarozzaKristensenPasarelli}    might lead to $\mathscr{F}_{j}[u_{j};\Omega]\not\to\overline{\mathscr{F}}_{u_{0}}^{*}[u;\Omega]$, see Remark~\ref{rem:strategy}. In particular, the connection to the original functional might get lost. Since we are bound to employ an Ekeland-type argument due to the possible phenomenon of non-uniqueness of relaxed minimizers, this means that non-uniqueness and the deteriorated growth behaviour are two \emph{coupled effects} which have a considerably worse impact than in the purely linear growth case. 

Thus, the construction of a suitable approximating sequence whose members share on the one hand a useful local $\sobo^{-1,1}$-almost minimality amenable to subsequent regularity estimates and on the other hand satisfy $\mathscr{F}_{j}[u_{j};\Omega]\to\overline{\mathscr{F}}_{u_{0}}^{*}[u;\Omega]$, requires a fine-tuned and conceptually novel set-up. Its precise implementation, which relies on a quantitative handling of the potential $\lebe^{q}$-gradient blow-ups of $\sobo^{1,q}$-almost minimizers for each $j\in\mathbb{N}$, is given in Section~\ref{sec:Ekelandstart}; see also  Remarks \ref{rem:strategy}--\ref{rem:recovery}.  \\
\vspace{-0.25cm}

(c) \textbf{Universal gradient estimates and exponent ranges.} 
As a consequence of this somewhat different procedure, the subsequent perturbed Euler--Lagrange inequality takes a slightly weaker form than in previous contributions, partially leading to the non-admissibility of natural test maps; see Section~\ref{sec:Cacc}. While comparatively weak, the resulting estimates are still robust enough to allow an optimisation of cut-offs  in the spirit of \cite{BELSCH20,BELSCH24,CIASCH24,DEFKOCKRI24}; in this regard, Section~\ref{sec:goodcut} provides the requisite background results. In combination with Reshetnyak's lower semicontinuity theorem, we prove in Sections~\ref{sec:highint} and~\ref{sec:higher} that -- despite the additional difficulties due to the $q$-growth from above in \eqref{eq:1qgrowth} -- the $\sobo^{1,1}$-regularity assertion from Theorem~\ref{thm:main} holds for \emph{any} relaxed minimizer. Since we do not rely on integral or measure representations of $\overline{\mathscr{F}}_{u_{0}}^{*}[-;\Omega]$, the precise form of the estimates \eqref{eq:finalboundA1}--\eqref{eq:finalboundA1log} requires another argument. As a \emph{consequence} of an intermediate universal regularity result (Proposition~\ref{prop:semimain}) and thereby improving the very weak convergence properties of the approximating sequence, the latter is ultimately shown to have weak gradients converging $\mathscr{L}^{n}$-a.e.. This requires a restructuring of the proof, but comes to the effect that the estimates can be localised to appear in their natural forms \eqref{eq:finalboundA1}--\eqref{eq:finalboundA1log}. 

Theorem~\ref{thm:main2}, in turn, strongly hinges on Theorem~\ref{thm:main}\ref{item:main2b} and advances an argument due to Bildhauer \& Fuchs  \cite{BILFUC03} towards $p=1$ and its applicability in the Ekeland approximation scheme. The latter, being necessary to obtain universal estimates, forces us to perform the underlying limit passages differently. In particular, they must be accomplished in an order to get \emph{useful} access to the available a priori estimates. This requires a variant of the Frehse--Seregin lemma \cite{FRESER98}, see Lemma~\ref{lem:FRESER98}. We thereby arrive at  gradient continuity, from where an adaptation of an argument due Ancona and Brezis \cite{ANCONA09} implies Theorem~\ref{thm:main2}. \\
\vspace{-0.25cm}

In view of the available  results for the related functionals of $(p,q)$-growth, the exponent ranges of Theorems~\ref{thm:main} and~\ref{thm:main2} seem to be close to optimal. In this sense, our results provide a sharp bridging between the case $p=1$ and $p>1$; see also Remark~\ref{rem:compamain} below. Moreover, integrands whose growth can be located at the interface of both scenarios are equally included in our setting as special cases. For instance, this concerns  non-uniformly elliptic integrands of critical Orlicz growth (e.g., $\mathrm{L}\log\mathrm{L}$-growth from below), a class that has lately attracted attention, see \cite{DEFILIPPIS23} for an overview and \cite{ELEMARMASPER22,DEFMIN23,DEFPIC24,DEFILIPPISDEFILIPPISPICCININI24} for some recent results in this direction. 
Since the condition \eqref{eq:muell1q} is typically satisfied in this situation for $\mu=1$, Theorems~\ref{thm:main} and~\ref{thm:main2} immediately apply in the underlying ranges of $q$.

\color{black}
\begin{rem}[Comparison with $(p,q)$- and $(\mu,s,q)$-growth conditions]\label{rem:compamain}
We briefly compare the precise assumptions on the exponents $\mu$ and $q$ with related previous  results in the superlinear context, meaning that $F(z)/|z|\to \infty$ as $|z|\to \infty$. In \cite{BILFUC01,FUCMIN00}, various regularity results are established under the assumption
%
\begin{equation}\label{ineq:muqold}
\mu+q<2+\frac{2}{n},
\end{equation}
which, in general, is more restrictive than \eqref{eq:maincondition}. Condition \eqref{ineq:muqold} has a nice interpretation with earlier regularity results under $(p,q)$-growth conditions, in which the lower bound in \eqref{eq:1qgrowth} is replaced by $\gamma|z|^p$ for some $p>1$ and \eqref{eq:muell1q} holds with $\mu=2-p$: Classical results in this context (see, e.g.,  \cite{CarozzaKristensenPasarelli,MARCELLINI89,MARCELLINI91}), ensure gradient regularity under the condition  
\begin{equation}\label{ineq:pqold}
\frac{q}{p}<1+\frac{2}{n},
\end{equation}
which perfectly matches \eqref{ineq:muqold} in the limit $p\searrow 1$. More recently, the condition \eqref{ineq:pqold} was improved in \cite{BELSCH24,SCHAEFFNER21,SCHAEFFNER24} to
\begin{equation}\label{ineq:pqnew}
\frac{q}{p}<1+\frac{2}{n-1},
\end{equation}
which is consistent with condition  \eqref{eq:maincondition} of the present paper in the limit $p\searrow1$. In fact, with methods of the present paper it seems possible to extend the higher integrability results of \cite{SCHAEFFNER21} to the subquadratic case and $n=2$ dimensions. Finally, we remark that \eqref{ineq:pqnew} reads $q<3p$ if $n=2$, which is also consistent with condition \eqref{eq:2dstrongercond} with the choice $\mu=2-p$; note that \eqref{eq:2dstrongercond} trivially holds if $1\leq p\leq q\leq 2$ and $\mu=2-p$. 
\end{rem}

Lastly, if $q=1$ in \eqref{eq:1qgrowth} and \eqref{eq:muell1q}, purely linear growth integrands share some similarities with $(p,q)$-growth integrands on the level of \emph{second} derivatives. Still, the landmark for $\sobo^{1,1}$-regularity in the realm of the unconstrained Dirichlet problem on $\bv$ is $1<\mu<1+\frac{2}{n}$ (see also  \cite{BILDHAUER03,GMEINEDER20}), which is slightly extended by  Theorem~\ref{thm:main}~\ref{item:main3}, also covering the endpoint case $\mu=1+\frac{2}{n}$ in $n=2$ dimensions. Further improvements are only available for a related Neumann problem \cite{BECBULGME20} or under additional $\lebe^{\infty}$-hypotheses \cite{BECEITGME,BECSCH13,BILDHAUER2002MUBOUNDED}. In the latter case, the dimension-independent range $1<\mu\leq 3$ implies the  $\sobo^{1,1}$-regularity of relaxed minimizers. While we shall pursue such $\lebe^{\infty}$-constrained results as companions of Theorem~\ref{thm:main} in future work, improvements for $q=1$ in the unconstrained case seem unlikely. In particular, it is not clear to us whether the methods of this paper can be adapted to yield $\sobo^{1,1}$-regularity with $q=1$ and $1<\mu<1+\frac{2}{n-1}$ in the vectorial case. Indeed, in view of the currently available strategies, the results of the present paper indicate that -- even though purely linear growth problems have a strong resemblance to $(p,q)$-growth problems on the level of second derivatives -- both classes are different on a more fundamental level; see Remark~\ref{rem:unconditional}. 

\subsection{Structure of the paper} We now briefly comment on the organization of the paper. In Section~\ref{sec:prelims}, we fix notation and collect the key background facts and definitions for the proofs of Theorems~\ref{thm:main} and~\ref{thm:main2}. After recording an analytic lemma on good cut-off functions in Section~\ref{sec:goodcut}, Section~\ref{sec:relax} gathers various fundamental properties of the relaxed functionals. In particular, this comprises the existence of relaxed minimizers and a no-gap-result, identifying~\eqref{eq:LSM} as the \emph{correct} extension. Along with the more involved construction of the underlying Ekeland-type vanishing viscosity sequence as a central point, Section~\ref{sec:proofmain} is devoted to the proof of Theorem~\ref{thm:main} and a related dimension reduction result for the singular set. Section~\ref{sec:proofmain2} serves to establish Theorem~\ref{thm:main2}. The Appendix, Section~\ref{sec:appendix}, gives a quick discussion of integral representations and provides the proofs of several auxiliary results that enter the main part. 
 
\section{Preliminaries}\label{sec:prelims}
\subsection{General notation}\label{sec:notation}
Throughout the entire paper and unless stated otherwise, $\Omega\subset\R^{n}$ denotes an open and bounded set with Lipschitz boundary. For $x_{0}\in\R^{n}$ and $r>0$, we put $\ball_{r}(x_{0}) \coloneqq \{x\in\R^{n}\colon\;|x-x_{0}|<r\}$ and, accordingly,  $\mathbb{S}^{n-1}\coloneqq \partial\!\ball_{1}(0)$. The $n$-dimensional Lebesgue and $(n-1)$-dimensional Hausdorff measures are denoted by $\mathscr{L}^{n}$ and $\mathscr{H}^{n-1}$. For two matrices $z,z'\in\R^{N\times n}$, we write $\langle z,z'\rangle$ for their Hilbert--Schmidt inner product. 

For a finite dimensional inner product space $V$, we denote by $\mathrm{RM}(\Omega;V)$ the $V$-valued Radon measures on $\Omega$, whereas we write $\mathrm{RM}_{\mathrm{fin}}(\Omega;V)$ for the finite, $V$-valued Radon measures on $\Omega$ endowed with the total variation norm $|\mu|(\Omega)$. For a non-negative Radon measure, in formulae $\mu\in\mathrm{RM}(\Omega)$, a Borel set $U\subset\Omega$ with $\mathscr{L}^{n}(U)>0$ and a $\mu$-integrable map $u\colon \Omega\to V$, we moreover set 
\begin{align*}
\dashint_{U}u\dif\mu \coloneqq \frac{1}{\mathscr{L}^{n}(U)}\int_{U}u\dif\mu\;\;\;\text{and}\;\;\;(u)_{\ball_{r}(x_{0})}\coloneqq \dashint_{\ball_{r}(x_{0})}u\dif x. 
\end{align*}
To obtain inequalities involving Orlicz norms which obey the correct scaling, it is convenient to denote for a Young function $\Phi\colon \R_{\geq 0}\to\R_{\geq 0}$, a ball $\ball_{r}(x_{0})$ and $u\in\lebe_{\locc}^{\Phi}(\R^{n};V)$
\begin{align*}
{\;\;-}{\!\!\!\|u\|}_{\lebe^{\Phi}(\ball_{r}(x_{0}))} \coloneqq \inf\Big\{\lambda>0\colon\;\dashint_{\ball_{r}(x_{0})}\Phi\Big(\frac{|u|}{\lambda}\Big)\dif x \leq 1 \Big\}.
\end{align*}
Given a measurable map $u\colon\Omega\to\R^{N}$ and $x\in\Omega$, we denote for $\ell\in\{1,\ldots ,n\}$ and $0<h<\mathrm{dist}(x,\partial\Omega)$ the associated difference quotient in $\ell$-th direction by
\begin{align*}
\Delta_{\ell,h}^{\pm}u(x) \coloneqq \frac{1}{h}(u(x\pm h e_{\ell})-u(x)),
\end{align*}
where $e_{\ell}$ is the $\ell$-th standard unit vector, and sometimes write $\Delta_{\ell,h}\coloneqq \Delta_{\ell,h}^{+}$ for brevity. Finally, we use $c,C>0$ to denote generic constants which may change from one line to the other, and we only specify them if their precise values are required in the sequel. 
\subsection{Function spaces}\label{sec:funcspaces} We now collect some background definitions and results for various function spaces to enter the main part. With few exceptions, the material is discussed in greater in detail in \cite{AMBFUSPAL00,EVANSGARIEPY}. 
\subsubsection{Functions of bounded variation.} We say that  $u\in\lebe^{1}(\Omega;\R^{N})$ is of \emph{bounded variation}, denoted by $u\in\bv(\Omega;\R^{N})$, if its total variation
\begin{align*}
|\D u|(\Omega)\coloneqq \sup\Big\{\int_{\Omega}\langle u,\mathrm{div}(\varphi)\rangle\dif x\colon\;\varphi\in\hold_{c}^{\infty}(\Omega;\R^{N\times n}),\;\| \varphi \|_{\lebe^\infty(\Omega)} \leq 1 \Big\}
\end{align*}
is finite; here, $\mathrm{div}(\varphi)$ is the row-wise divergence. The space $\bv_{\locc}(\Omega;\R^{N})$ then is defined in the obvious way. As already mentioned in the introduction, the Lebesgue--Radon--Nikod\'{y}m decomposition of $\D u$ for $u\in\bv(\Omega;\R^{N})$ reads as
\begin{align}\label{eq:gradientdecomp}
\D u = \D^{a} u + \D^{s} u = \nabla u\,\mathscr{L}^{n} + \frac{\dif \D^s u}{\dif|\D^{s} u|}|\D^{s} u|,\qquad\text{where}\;\D^{a} u\ll\mathscr{L}^{n}\;\text{and}\;\D^{s} u\perp\mathscr{L}^{n}. 
\end{align}
Here, $\nabla u\in\lebe^{1}(\Omega;\R^{N\times n})$ is the approximate gradient of $u$. For future reference, we recall that Alberti's celebrated rank-one theorem \cite{ALBERTI93} asserts that 
\begin{align}\label{eq:alberti}
\mathrm{rk}\Big(\frac{\dif \D^{s} u}{\dif|\D^{s} u|}\Big) = 1\qquad |\D^{s} u|\text{-a.e. in $\Omega$}. 
\end{align}
Given $u,u_{1},\ldots \in\bv(\Omega;\R^{N})$, we say that the sequence $(u_{j})$ converges in the \emph{weak*-sense} to $u$ and write $u_{j}\stackrel{*}{\rightharpoonup} u$ provided that $u_{j}\to u$ strongly in $\lebe^{1}(\Omega;\R^{N})$ and $\D u_{j}\stackrel{*}{\rightharpoonup} \D u$ in the sense of weak*-convergence in $\mathrm{RM}_{\mathrm{fin}}(\Omega;\R^{N\times n})$. 

If a sequence $(u_{j})$ in $\bv(\Omega;\R^{N})$ is bounded with respect to the $\bv$-norm $\|v\|_{\bv(\Omega)}\coloneqq \|v\|_{\lebe^{1}(\Omega)}+|\D v|(\Omega)$, then there exists $u\in\bv(\Omega;\R^{N})$ and a (non-relabelled) subsequence such that $u_{j}\stackrel{*}{\rightharpoonup} u$ in $\bv(\Omega;\R^{N})$. In the following, we shall refer to this fact as \emph{weak*-compactness theorem} on $\bv(\Omega;\R^{N})$. The next lemma is certainly clear to the experts, but we have not found a precise reference and so we include it for the reader's convenience: 

\begin{lem}\label{lem:nonempty} 
 Let $u_{0},u\in\bv(\Omega;\R^{N})$. Then there exists a sequence $(u_{j})$ in $u_{0}+\hold_{c}^{\infty}(\Omega;\R^{N})$ such that $u_{j}\stackrel{*}{\rightharpoonup}u$ in the weak*-sense on $\bv(\Omega;\R^{N})$. 
\end{lem}
\begin{proof}
Put $v\coloneqq u-u_{0}$ and define, for sufficiently large $j\in\mathbb{N}$, $v_{j}\coloneqq \rho_{\varepsilon_{j}}*(\eta_{j}v)$. Here, $\rho_{\varepsilon_{j}}$ is the $0<\varepsilon_{j}<\frac{1}{j}$-rescaled variant of a standard mollifier, and $\eta_{j}\in\hold_{c}^{1}(\Omega;[0,1])$ satisfies $\eta_{j}(x)=0$ whenever $0<\mathrm{dist}(x,\partial\Omega)<\frac{1}{j}$ as well as $\eta_{j}(x)=1$ whenever $x\in\Omega$ is such that $\frac{2}{j}<\mathrm{dist}(x,\partial\Omega)$. Clearly, $v_{j}\to v$ strongly in $\lebe^{1}(\Omega;\R^{N})$, and we have for all $\varphi\in\hold_{c}^{\infty}(\Omega;\R^{N\times n})$ that 
\begin{align*}
\int_{\Omega}\langle\varphi,\nabla v_{j}\rangle\dif x & = -\int_{\Omega}\langle\mathrm{div}(\varphi),v_{j}\rangle\dif x \to - \int_{\Omega}\langle\mathrm{div}(\varphi),v\rangle\dif x = \int_{\Omega}\langle\varphi,\D v\rangle\qquad \text{as } j\to\infty. 
\end{align*}
The sequence $(u_{j})$ with $u_{j}\coloneqq u_{0}+v_{j}$ then has the desired properties. 
\end{proof}
We note that the previous lemma is in line with the fact that the boundary trace operator $\mathrm{tr}_{\partial\Omega}\colon\bv(\Omega;\R^{N})\to\lebe^{1}(\partial\Omega;\R^{N})$ is continuous with respect to norm convergence, but not with respect to weak*-convergence on $\bv(\Omega;\R^{N})$.
\subsubsection{Negative Sobolev spaces}\label{sec:negsob}
Following \cite{BECSCH13,GMEINEDER20}, it is convenient to perform Ekeland perturbations in the negative Sobolev space $\sobo^{-1,1}(\Omega;\R^{N})$ in the main regularity proof. We recall that $T\in\mathscr{D}'(\Omega;\R^{N})$ belongs to $\sobo^{-1,1}(\Omega;\R^{N})$ if there exist maps $T_{0},\ldots ,T_{n}\in\lebe^{1}(\Omega;\R^{N})$ with 
\begin{align}\label{eq:W-1,1A}
T = T_{0} + \sum_{k=1}^{n}\partial_{k}T_{k} \qquad\text{as an identity on}\;\mathscr{D}'(\Omega;\R^{N}). 
\end{align}
The $\sobo^{-1,1}$-norm of $T$ is given  by 
\begin{align*}
\|T\|_{\sobo^{-1,1}(\Omega)} \coloneqq \inf\bigg\{\sum_{k=0}^{n}\|T_{k}\|_{\lebe^{1}(\Omega)}\colon\;T_{0},\ldots,T_{n}\in\lebe^{1}(\Omega;\R^{N})\;\text{satisfy}\;\eqref{eq:W-1,1A}\bigg\}, 
\end{align*}
and makes $\sobo^{-1,1}(\Omega;\R^{N})$ into a Banach space. For future reference, we record that 
\begin{align}\label{eq:W-11cancel}
\|\partial_{j} u\|_{\sobo^{-1,1}(\Omega)}\leq \|u\|_{\lebe^{1}(\Omega)}\qquad\text{for all}\;u\in\lebe^{1}(\Omega;\R^{N})\;\text{and all}\;j\in\{1,\ldots ,n\}. 
\end{align}
Moreover, if $u\in\lebe^{1}(\Omega;\R^{N})$ is compactly supported in $\Omega$, then we have the estimate 
\begin{align}\label{eq:diffquotW-11}
\|\Delta_{\ell,h}^{\pm}u\|_{\sobo^{-1,1}(\Omega)} \leq \|u\|_{\lebe^{1}(\Omega)}\qquad\text{for all}\;\ell\in\{1,\ldots ,n\}\;\text{and all}\;0<h<\mathrm{dist}(\spt(u),\partial\Omega) 
\end{align}
for the difference quotients; see \cite[Sec.~3.2.3]{BECEITGME} and \cite[Lem.~2.5]{GMEINEDER20} for more detail.

\subsubsection{A Campanato-type lemma} Next, we recall a logarithmic Campanato-type embedding which can be found in the classical paper of Frehse \cite{FREHSE75}, see also Kovats \cite[Main Theorem]{KOVATS99}. 

\begin{lem}\label{lem:campanato}
Suppose that, for some $x_0\in\R^n$, $R>0$ and a function $u\in \lebe^2(\ball_{2R}(x_{0}))$ there exists constants $K>0$ and $p>1$ such that, for all $y\in \ball_{R}(x_{0})$ and all $0<r<R$, there holds
\begin{align*}
    \bigg(\dashint_{\ball_r(y)}|u(x)-(u)_{y,r}|^2\dif x\bigg)^\frac{1}{2}\leq K\log(2R/r)^{-p}.
\end{align*}
Then, $u$ is continuous in $\ball_{R}(x_{0})$ and there exists $K'=K'(p,K)>0$ with
\begin{align*}
|u(x)-u(x')|\leq K'\log(2R /|x-x'|)^{1-p}\qquad\text{for all}\;x,x'\in \ball_{R}(x_{0}). 
\end{align*}
\end{lem}
\subsubsection{Weighted Lebesgue spaces} In order to establish the admissibility of certain test maps in the main part, see Section~\ref{sec:Cacc}, we finally collect the following lemma from \cite[Lem. 3.4]{BECEITGME}; for related results, see \cite[Lem~II.1.18]{GAJKROZAC74} and \cite[Ex. 12.12(f)]{LukesMaly}.
\begin{lem}\label{lem:weighted}
Let $m\in\mathbb{N}$,  $\theta\in\lebe^{1}(\Omega)$ with $\theta\geq 1$ $\mathscr{L}^{n}$-a.e. in $\Omega$ and put $\mu\coloneqq \theta\mathscr{L}^{n}\mres\Omega$. Denoting the corresponding weighted $\lebe^{2}$-space by $\lebe_{\mu}^{2}(\Omega;\R^{m})$, suppose that $(u_{j})$ in $\lebe_{\mu}^{2}(\Omega;\R^{m})$ converges
\begin{enumerate}
\item weakly to some $u\in\lebe_{\mu}^{2}(\Omega;\R^{m})$, and 
\item pointwisely $\mathscr{L}^{n}$-a.e. to some measurable $v\colon\Omega\to\R^{m}$. 
\end{enumerate}
Then $u=v$. 
\end{lem}
\subsubsection{Exponential Orlicz classes}\label{sec:expOrliczclasses}
In the main part, see Section~\ref{sec:highint}, we shall need Orlicz spaces adapted to the function $z\mapsto \exp(|z|^{2-\mu})-1$ with $1\leq \mu <2$. This function is only convex for sufficiently large values of the arguments, which is why we require a modification and comparison estimates. Put $\theta_{\mu}(t)\coloneqq \exp(t^{2-\mu})-1$. We compute 
\begin{align*}
\theta''_{\mu}(t) = (2-\mu)\Big((1-\mu)t^{-\mu}+(2-\mu)t^{2-2\mu}\Big)\exp(t^{2-\mu})
\end{align*}
and, recalling that $1\leq \mu<2$, note that 
\begin{align*}
t>t_{0}\coloneqq \Big(\frac{\mu-1}{2-\mu}\Big)^{\frac{1}{2-\mu}}\qquad \Longrightarrow \qquad \theta''_{\mu}(t)>0,  
\end{align*}
whereby $\theta_{\mu}$ is convex on $(t_{0},\infty)$. We then consider, for a value $t_{0}<t_{1}<\infty$ to be fixed later on, the continuous function $\Phi_{\mu}\colon \R_{\geq 0}\to\R_{\geq 0}$ defined by 
\begin{align}\label{eq:escobar}
\Phi_{\mu}(t)\coloneqq \begin{cases}
    \frac{1}{t_{1}}(\exp(t_{1}^{2-\mu})-1)t&\;\text{if}\;t\leq t_{1},\\ 
    \theta_{\mu}(t)&\;\text{if}\;t\geq t_{1}. 
\end{cases}
\end{align}
Moreover, we note for the left and right derivative of $\Phi_\mu$ at $t_1$ that 
\begin{equation*}
\frac{1}{t_{1}}(\exp(t_{1}^{2-\mu})-1)\leq (2-\mu)t_{1}^{1-\mu}\exp(t_{1}^{2-\mu}) \qquad \stackrel{s\coloneqq t_{1}^{2-\mu}}{\Longleftrightarrow} \qquad \exp(s)-1\leq (2-\mu)s\exp(s), 
\end{equation*}
and from here, it is clear that there exists $t_{1}=t_{1}(\mu)>t_{0}$ such that $\Phi_{\mu}$ is a convex function. In the following, we fix such a choice of $t_{1}$ and we state a comparison estimate which is proved in the  Appendix, Section~\ref{sec:AppendixExponential}.

\begin{lem}\label{lem:exponential}
Let $1\leq \mu<2$ and let $t_{1}>0$ be as above. Then there exists a constant $c=c(t_{1},\mu)>1$ such that for every  $v\in \lebe^{1}(\ball_{r}(x_{0}))$ with $x_{0}\in\R^{n}$ and $r>0$ there holds
\begin{align}\label{eq:tenpenny}
\frac{1}{c}\!\!\!{\;\;-}{\!\!\!\|v\|}_{\lebe^{\Phi_{\mu}}(\ball_{r}(x_{0}))} \leq \!\!\!{\;\;-}{\!\!\!\||v|^{2-\mu}\|}_{\exp\lebe^{1}(\ball_{r}(x_{0}))}^{\frac{1}{2-\mu}} \leq c\, \!\!\!{\;\;-}{\!\!\!\|v\|}_{\lebe^{\Phi_{\mu}}(\ball_{r}(x_{0}))}. 
\end{align}
\end{lem}
\subsection{Functionals of measures}\label{sec:funmeas}
Even though our main result, Theorem~\ref{thm:main}, will be established by not referring to integral representations, it is yet useful to record some background facts on convex functionals of measures; see, e.g.,  \cite{DEMTEM24,GOFSER64} for more detail. We begin with: 

\begin{definition}[Recession function]\label{def:recfct}
Let $f\colon \R^{m}\to (-\infty,\infty]$ be a convex function. The \emph{recession function} $f^{\infty}\colon\R^{m}\to(-\infty,\infty]$ of $f$ then is defined by 
\begin{align}
f^{\infty}(z)\coloneqq \lim_{t\searrow 0}tf\Big(\frac{z}{t}\Big) \qquad \text{for all } z\in\R^{m}. 
\end{align}
\end{definition}
The recession function captures the behaviour of the function at infinity. Moreover, $f^{\infty}(z)$ exists in $(-\infty,\infty]$ for every $ z\in\R^{m}$ by the convexity of $f$; this is a consequence of the fact that convex functions can always be minorised by affine-linear maps and that difference quotients of convex functions are increasing. However, since $f$ does not need to be of linear growth, it is indeed possible that $f^{\infty}(z)=+\infty$ for some $z\in\R^{m}$. 
We now proceed to Reshetnyak's lower semicontinuity theorem \cite{RESHETNYAK68}, which we give in the form of \cite[Thm. 2.4]{BECSCH13}.
\begin{prop}[Reshetnyak's lower semicontinuity theorem]\label{prop:resh}
Let $\Omega\subset\R^{n}$ be open and bounded. Moreover, suppose that $\mu,\mu_{1},\mu_{2},\ldots \in\mathrm{RM}_{\mathrm{fin}}(\Omega;\R^{m})$ take values in a closed, convex cone $K\subset\R^{m}$ and are such that $\mu_{j}\stackrel{*}{\rightharpoonup}\mu$ in $\mathrm{RM}_{\mathrm{fin}}(\Omega;\R^{m})$. Then, if $f\colon K\to[0,\infty]$ is lower semicontinuous, convex and $1$-homogeneous, then there holds 
\begin{align}
\int_{\Omega}f\Big(\frac{\dif\mu}{\dif|\mu|}\Big)\dif|\mu| \leq \liminf_{j\to\infty}\int_{\Omega}f\Big(\frac{\dif\mu_{j}}{\dif|\mu_{j}|}\Big)\dif|\mu_{j}|.
\end{align}
\end{prop}
For future reference, we single out the following remark: 
\begin{rem}\label{rem:continuitypartResh}
Proposition~\ref{prop:resh} usually comes with a variant concerning continuity (see, e.g. \cite[Thm. 2.4]{BECSCH13}): If, in the situation of Proposition~\ref{prop:resh}, $f\colon K\to[0,\infty)$ is continuous, $1$-homogeneous and we moreover have that $|\mu_{j}|(\Omega)\to|\mu|(\Omega)$, then there holds 
\begin{align}\label{eq:reshcontinuity}
\int_{\Omega}f\Big(\frac{\dif\mu}{\dif|\mu|}\Big)\dif|\mu| = \lim_{j\to\infty}\int_{\Omega}f\Big(\frac{\dif\mu_{j}}{\dif|\mu_{j}|}\Big)\dif|\mu_{j}|.
\end{align}
In our applications, however, $f\colon K\to[0,\infty]$ is typically neither continuous (but only lower semicontinuous) and extended real-valued. Under these assumptions, the continuity assertion \eqref{eq:reshcontinuity} does not hold true, see Example~\ref{ex:reshnon} below. 
\end{rem}
\begin{rem}\label{rem:linearperspective}
Let $F\colon\R^{N\times n}\to\R$ be convex and put $K\coloneqq \R_{\geq 0}\times\R^{N\times n}$. 
We introduce the linear perspective integrand $F^{\#}\colon K\to\R\cup\{\infty\}$ by 
\begin{align}\label{eq:linearperspective}
F^{\#}(t,z)\coloneqq \begin{cases}
    tF(\frac{z}{t})&\;\text{if}\;t>0, \\ 
    F^{\infty}(z)&\;\text{if}\;t=0, 
\end{cases}\qquad z\in\R^{N\times n},
\end{align}
whereby $F^{\#}$ is a continuous, convex  function which is homogeneous of degree $1$.  For a Radon measure $\mu\in\mathrm{RM}_{\mathrm{fin}}(\Omega;\R^{N\times n})$, we put $\nu\coloneqq (\mathscr{L}^{n},\mu)$ and observe that 
\begin{align*}
\int_{\Omega}F^{\#}\Big(\frac{\dif\nu}{\dif|\nu|}\Big)\dif|\nu| = \int_{\Omega}F\Big(\frac{\dif\mu}{\dif\mathscr{L}^{n}}\Big)\dif x + \int_{\Omega}F^{\infty}\Big(\frac{\dif\mu}{\dif|\mu^{s}|}\Big)\dif|\mu^{s}|, 
\end{align*}
where $\mu=\mu^{a}+\mu^{s}$ is the  Radon-Nikod\'{y}m decomposition of $\mu$ such that $\mu^{a}\ll\mathscr{L}^{n}$ and $\mu^{s}\bot\mathscr{L}^{n}$.
\end{rem}
\begin{example}\label{ex:reshnon} Let $1<q<\infty$. We  define a convex function $f\colon\R\to \R$ by $f(z)\coloneqq z$ for $z\geq 0$ and $f(z)\coloneqq |z|^{q}$ for $z<0$. Following Remark~\ref{rem:linearperspective}, we record that $f^{\#}\colon\R_{\geq 0}\times\R\to\R\cup\{\infty\}$ is lower semicontinuous and satisfies $f^{\#}\equiv \infty$ on $\{0\}\times \R_{<0}$, whereas $f^{\#}(0,z)=z$ for $z\in \R_{\geq 0}$. For $j\geq 2$, define  $\mu_{j}\coloneqq -\frac{1}{j}\delta_{-1/j}$. Then $\mu_{j}\stackrel{*}{\rightharpoonup}0$ in $\mathrm{RM}_{\mathrm{fin}}((-1,1);\R)$ together with $|\mu_{j}|((-1,1))\to 0$ as $j\to\infty$. Adopting the notation of Remark~\ref{rem:linearperspective} with $n=N=1$, we have $\nu_{j}\stackrel{*}{\rightharpoonup} \nu \coloneqq (\mathscr{L}^{1},0)$ and $|\nu_{j}|((-1,1))\to|\nu|((-1,1))$. However, we compute 
\begin{align*}
\int_{(-1,1)}f^{\#}\Big(\frac{\dif\nu_{j}}{\dif|\nu_{j}|}\Big)\dif|\nu_{j}| = \frac{1}{j}\int_{(-1,1)}f^{\infty}(-1)\dif\delta_{-1/j} = +\infty \stackrel{j\to\infty}{\not\longrightarrow} 0 = \int_{(-1,1)}f^{\#}\Big(\frac{\dif\nu}{\dif|\nu|}\Big)\dif|\nu|. 
\end{align*}
In particular, Reshetnyak's continuity theorem is not available for lower semicontinuous, extended real-valued convex integrands; note that $f$ however matches the chief assumption \eqref{eq:1qgrowth}.
\end{example}
\subsection{Ekeland's variational principle} Next, we recall Ekeland's variational principle \cite{EKELAND74} (see also \cite[Thm. 5.6, Rem. 5.5]{GIUSTI03}) which shall prove crucial for our approach. 
\begin{lem}[Ekeland]\label{lem:Ekeland} Let $(X,d_{X})$ be a complete metric space, and let $\mathcal{F}\colon X\to\R\cup\{\infty\}$ be a functional which is lower semicontinuous with respect to $d_{X}$, bounded from below and with $\inf_X \mathcal{F} < \infty$. If $x\in X$ and $\varepsilon>0$ are such that $\mathcal{F}[x]<\infty$ and 
\begin{align*}
\mathcal{F}[x] \leq \inf_{X}\mathcal{F} + \varepsilon^{2}, 
\end{align*}
then there exists $y\in X$ such that 
\begin{align*}
d_{X}(x,y)<\varepsilon\;\;\;\text{and}\;\;\;\mathcal{F}[y] \leq \mathcal{F}[z] + {\varepsilon}d_{X}(y,z)\qquad\text{for all}\;z\in X. 
\end{align*}
\end{lem}

\subsection{Miscellaneous estimates} In this section, we collect some elementary estimates which shall enter the main part. We begin with the following Lipschitz-type estimate:
\begin{lem}[{\cite[Lem. 5.2]{GIUSTI03}}]\label{lem:liptype}
Let $1\leq q<\infty$, and let $F\in\hold(\R^{N\times n})$ be a  \emph{convex} function such that there exists a constant $c>0$ with  
\begin{align*}
0\leq F(z)\leq c(1+|z|^{q})\qquad\text{for all}\;z\in\R^{N\times n}. 
\end{align*}
Then there exists a constant $L>0$ such that
\begin{align}\label{eq:liptype}
|F(z)-F(z')| \leq L\big(1+|z|+|z'|\big)^{q-1}|z-z'|\qquad\text{for all}\;z,z'\in\R^{N\times n}. 
\end{align}
\end{lem}
Next, we record an elementary but fundamental lemma: 
\begin{lem}[{\cite[Lem. 6.1]{GIUSTI03}}]\label{lem:holefilling}
Let $0<\rho<\sigma<\infty$ and let $Z\colon [\rho,\sigma]\to\R_{\geq 0}$ be a bounded, non-negative function. Suppose that, for all $\rho\leq s<t\leq\sigma$, there holds 
\begin{align}\label{eq:iter1}
Z(s)\leq \theta Z(t) + (t-s)^{-\alpha}A_{1}+ (t-s)^{-\beta}A_{2} + B, 
\end{align}
where $A_{1},A_{2},B\geq 0$, $\alpha,\beta>0$ and $0\leq\theta <1$ are constants. Then there exists $c=c(\alpha,\beta,\theta)\in [1,\infty)$ such that 
\begin{align}\label{eq:iter2}
Z(s)\leq c\,((t-s)^{-\alpha}A_{1}+(t-s)^{-\beta}A_{2}+B). 
\end{align}
\end{lem}
Lastly, in view of the Lipschitz bounds in the two-dimensional case (see Section~\ref{sec:proofmain2}), we require the following variant of an estimate due to Frehse and Seregin \cite[Lem. 2.4]{FRESER98}.
\begin{lem}[of Frehse--Seregin-type]\label{lem:FRESER98}
Let $n=2$, $\alpha\in[1,2)$, $L>0$ and $\theta\geq 0$. For $x_0\in\R^2$ and $r>0$, set $\mathcal{A}_r(x_0) \coloneqq \ball_{2r}(x_0)\setminus \overline{\ball}_r(x_0)$. Suppose that, for some $x_0\in\R^2$ and $0<R<\infty$, a pair of functions $H\in \lebe^2(\ball_{2R}(x_0))$ and $h\in \sobo^{1,2}(\ball_{2R}(x_0)))$ satisfies the inequality
\begin{equation}\label{eq1:FreSe99variant}
\int_{\ball_r(x_0)}H^2\,{\mathrm d}x\leq \frac{L}{r} \biggl(\int_{\mathcal{A}_r(x_0)}H^2\,{\mathrm d}x\biggr)^\frac12\int_{\mathcal{A}_r(x_0)}|h|^\alpha |H|\,{\mathrm d}x + \theta\qquad\text{for all } 0<r<R. 
\end{equation}
Then, for any $1\leq p<\infty$, there exists a constant 
\begin{align*}
c=c(p,R^{-1}\|h\|_{\lebe^{2}(\ball_{2R}(x_0))}+\|\nabla h\|_{\lebe^{2}(\ball_{2R}(x_0))},L,\alpha)>0
\end{align*}
such that
\begin{align*}
\int_{\ball_r(x_0)}H^2\,\mathrm dx\leq \frac{c}{\big(\log_2(\frac{2R}r)\big)^p}\int_{\ball_{2R}(x_0)}H^2\,{\mathrm d}x + \theta\qquad\text{for all } 0<r<R.
\end{align*}
\end{lem}
Since the statement of the lemma is not quite that of  \cite[Lem. 2.4]{FRESER98}, a self-contained proof is offered in the Appendix, Section~\ref{sec:proofofFS}.  
\section{Good cut-offs}\label{sec:goodcut}
Towards the main regularity proofs in Section~\ref{sec:proofmain} and~\ref{sec:proofmain2} below, we now isolate a lemma on good cut-offs in the spirit of \cite{BELSCH20,BELSCH24}. 
\begin{lem}\label{L:optim}
Fix $p\in[2,\infty)$ and suppose that $Q\geq p$ is such that
\begin{equation*}
	\begin{cases}
		Q\leq 2\frac{n-1}{n-3} &\mbox{if $n\geq4$,}
		\\
		Q<+\infty &\mbox{if $n=3$,}\\
		Q=\infty&\mbox{if $n=2$.}
	\end{cases}
\end{equation*}
Then there exists a constant $C=C(n,Q)$ such that for any ball $\ball_R=\ball_R(x_0)$, any $v\in (\sobo^{1,2}\cap \lebe^p)(\ball_R)$ and all radii $0< \rho<\sigma<R$, there exists a cut-off function $\eta\in \sobo_0^{1,\infty}(\ball_R)$ satisfying
\begin{equation}\label{L:optim:eta}
0\leq \eta\leq 1,\qquad \eta=1\qquad\mbox{in $\ball_{\rho}$},\qquad\|\nabla \eta\|_{\lebe^\infty(\ball_R)}\leq \frac{2}{\sigma-\rho}
\end{equation}
such that the following holds for all $\alpha\geq 0$: If $n\geq3$, we have 
\begin{align}\label{L:optim:claim}
\|v|\nabla \eta|^\alpha\|_{\lebe^Q(\ball_R)}\leq C(\sigma-\rho)^{\frac1Q-\alpha}\biggl(\frac{\|\sigma\nabla v\|_{\lebe^2(\ball_\sigma\setminus \ball_\rho)}}{((\sigma-\rho)\rho^{\gamma_1})^\frac12}+\frac{\| v\|_{\lebe^p(\ball_\sigma\setminus \ball_\rho)}}{((\sigma-\rho)\rho^{\gamma_2})^\frac1p}\biggr)
\end{align}
where 
\begin{align}\label{eq:gammaadjust}
\gamma_1 \coloneqq (n-1)\Big(1-\frac{2}Q\Big)\;\;\;\text{and}\;\;\;\gamma_2 \coloneqq (n-1)\Big(1-\frac{p}Q\Big).
\end{align}
For $n=2$, we have 
\begin{align}\label{L:optim:claimn2}
\|v|\nabla \eta|^\alpha\|_{\lebe^\infty(\ball_R)}\leq C(\sigma-\rho)^{-\alpha}\biggl(\frac{\| v\|_{\lebe^p(\ball_\sigma\setminus \ball_\rho)}}{((\sigma-\rho)\rho)^\frac1p}+\frac{\| v\|_{\lebe^p(\ball_\sigma\setminus \ball_\rho)}^\frac{p}{p+2}\|\sigma\nabla v\|_{\lebe^2(\ball_\sigma\setminus \ball_\rho)}^\frac{2}{p+2}}{((\sigma-\rho)^2\rho^{2})^\frac1{p+2}}\biggr) 
\end{align}
and 
\begin{align}\label{L:optim:claimn2:exp}
\begin{split}
& \|\exp( |v|) |\nabla \eta|^\alpha\|_{\lebe^\infty(\ball_R)} \\
& \;\;\;\;\;\;\;\leq C(\sigma-\rho)^{-\alpha}\biggl(\frac{\| \exp(|v|)\|_{\lebe^1(\ball_\sigma\setminus \ball_\rho)}}{(\sigma-\rho)\rho}+\frac{\|\exp(|v|)\|_{\lebe^1(\ball_\sigma\setminus \ball_\rho)}\|\sigma\nabla v\|_{\lebe^2(\ball_\sigma\setminus \ball_\rho)}^2}{(\sigma-\rho)^2\rho^{2}}\biggr).
\end{split}
\end{align}
\end{lem}

\begin{proof}
We follow the strategy of \cite[Lemma 5.1]{CIASCH24} combined with improvements in dimension $n=2$ inspired by \cite{SCHAEFFNER24}. Let $v\in (\sobo^{1,2}\cap \lebe^p)(\ball_R)$. Define, for $r\in[0,R]$, the map $v_r:\mathbb S^{n-1}\to\R$ by $v_r(z)\coloneqq v(rz)$ for $z\in \mathbb S^{n-1}$. Then there exists a set $\mathscr{N}\subset [0,R]$ such that $\mathscr{L}^{1}(\mathscr{N})=0$ and $v_r\in \sobo^{1,2}(\mathbb S^{n-1})$ for every $r\in[0,R]\setminus \mathscr{N}$. We set
\begin{equation}\label{def:U}
U_1\coloneqq \biggl\{r\in[\rho,\sigma]\setminus \mathscr{N}\colon\int_{\mathbb S^{n-1}}\frac{|\nabla_{\tau} v_r(z)|^2}{r^{2}}\,\mathrm{d}\mathscr H^{n-1}(z)\leq \frac{4}{(\sigma-\rho) r^{n-1}}\int_{\ball_\sigma\setminus \ball_{\rho}}|\nabla v|^2\dx\biggr\}, 
\end{equation}
where $\nabla_{\tau}$ denotes the weak tangential gradient along $\mathbb{S}^{n-1}$. 
In consequence, Fubini's  theorem, the elementary inequality $|\nabla_\tau v_r(z)|\leq |r\nabla v(rz)|$ and the definition of $U_1$ in the form
\begin{align*}
\int_{\ball_\sigma \setminus \ball_{\rho}}|\nabla v|^2\dx=&\int_{\rho}^\sigma r^{n-1}\int_{\mathbb S^{n-1}}|\nabla v(r z)|^2\,\mathrm{d}\mathscr H^{n-1}(z)\,\mathrm{d}r\\
\geq&\int_{(\rho,\sigma)\setminus U_1}r^{n-1}\int_{\mathbb S^{n-1}}|r^{-1}\nabla_{\tau} v_r(z)|^2\,\mathrm{d}\mathscr H^{n-1}(z)\,\mathrm{d}r\\
>&\frac{4(\sigma-\rho -\mathscr{L}^{1}(U_1))}{\sigma-\rho}\int_{\ball_\sigma  \setminus \ball_{\rho}}|\nabla v|^2\dx
\end{align*}
imply that $\mathscr{L}^{1}(U_1)\geq \frac34 (\sigma-\rho)$. By analogous means, we find that 
\begin{equation}\label{def:U2}
U_2\coloneqq \biggl\{r\in[\rho,\sigma ]\setminus \mathscr{N} \colon \int_{\mathbb S^{n-1}}|v_r(z)|^p\,\mathrm{d}\mathscr H^{n-1}(z)\leq \frac{4}{(\sigma-\rho) r^{n-1}}\int_{\ball_\sigma\setminus \ball_{\rho}}|v|^p\dx\biggr\}
\end{equation}
satisfies $\mathscr{L}^{1}(U_2)\geq \frac34 (\sigma-\rho)$. Hence, putting $U\coloneqq U_1\cap U_2$, we arrive at 
\begin{equation}\label{bound:lowerU}
\mathscr{L}^{1}(U)\geq \frac{\sigma-\rho}2. 
\end{equation}
Next, we define a Lipschitz cut-off function $\eta\in \sobo^{1,\infty}(\ball_R;[0,1])$ by
$$
\eta(x)\coloneqq \widetilde \eta(|x|),\qquad\mbox{where}\qquad \widetilde \eta(r)\coloneqq \begin{cases}1&\text{if $r\in(0,\rho)$,}\\\displaystyle \frac{1}{\mathscr{L}^{1}(U)}\int_{r}^\sigma{\mathbbm{1}_{U}}(s)\, \mathrm{d} s&\text{if $r\in (\rho,\sigma)$,}\\0&\text{if $r\in(\sigma,R)$.}\end{cases}
$$
By definition, we have that $0\leq \eta\leq1$, $\eta=1$ in $\ball_{\rho}$ and $\eta\in \sobo_0^{1,\infty}(\ball_\sigma)$, and by the Lebesgue differentiation theorem there holds, for $x=rz$ with $r\in[0,R]$ and $z\in \mathbb S^{n-1}$,  
\begin{equation}\label{lemma:optim:esnableeta}
|\nabla \eta(rz)|=\begin{cases}0&\mbox{for $\mathscr{L}^{1}$-a.e.\ $r\notin U$,}\\ \displaystyle\frac1{\mathscr{L}^{1}(U)}&\mbox{for $\mathscr{L}^{1}$-a.e.\ $r\in U$.}
\end{cases}
\end{equation}
Hence, recalling \eqref{bound:lowerU}, the map $\eta$ satisfies all the properties claimed in \eqref{L:optim:eta}.  For $r\in[0,R]\setminus \mathscr{N}$, we set
\begin{equation}
E(v_r)\coloneqq \|\nabla_{\tau} v_r\|_{\lebe^2(\mathbb S^{n-1})}+\|v_r\|_{\lebe^p(\mathbb S^{n-1})}
\end{equation}
and observe that the choice of $U$ ensures that 
\begin{align}\label{est:rEvr}
\begin{split}
\sup_{r\in U} \Big\{r^\frac{n-1}QE(v_r)\Big\} & \leq \sup_{r\in U} \bigg\{ r^\frac{n-1}Q\biggl(\biggl(\frac{4r^2}{(\sigma-\rho)r^{n-1}}\biggr)^\frac12\|\nabla v\|_{\lebe^2(\ball_\sigma\setminus \ball_\rho)}\Big. \\ & \Big.\;\;\;\;\;\;\;\;\;\;\;\;\;\;\;\;\;\;\;\;\;\;\;\;\;\;\;\;\;\;\;\;\;\;\;\;\;\;\;\; + \biggl(\frac{4}{(\sigma-\rho)r^{n-1}}\biggr)^\frac1p\| v\|_{\lebe^p(\ball_\sigma\setminus \ball_\rho)}\biggr) \bigg\}\\
& \leq \biggl(\frac{4}{(\sigma-\rho)\rho^{(n-1)(1-\frac{2}Q)}}\biggr)^\frac12\|\sigma\nabla v\|_{\lebe^2(\ball_\sigma\setminus \ball_\rho)} \\ & \;\;\;\;\;\;\;\;\;\;\;\;\;\;\;\;\;\;\;\;\;\;\;\;\;\;\;\;\;\;\;\;\;\;\;\;\;\;\;\; +\biggl(\frac{4}{(\sigma-\rho)\rho^{(n-1)(1-\frac{p}Q)}}\biggr)^\frac1p\| v\|_{\lebe^p(\ball_\sigma\setminus \ball_\rho)}.
\end{split}
\end{align}
Let us first consider the case $n\geq3$. We note that the exponent $Q\in[1,\infty)$ is such that $\sobo^{1,2}(\mathbb S^{n-1})\hookrightarrow \lebe^Q(\mathbb S^{n-1})$, and there exists $C=C(n,Q)>0$ such that 
\begin{equation*}
\|v_r\|_{\lebe^Q(\mathbb S^{n-1})}\leq C E(v_r)\qquad\text{for all}\;r\in U.
\end{equation*} 
Hence, we deduce from \eqref{lemma:optim:esnableeta} that
\begin{align*}
\|v|\nabla \eta|^\alpha\|_{\lebe^Q(\ball_R)}\leq \mathscr{L}^{1}(U)^{\frac1Q - \alpha} \sup_{r \in U}\Big\{r^\frac{n-1}Q\|v_r\|_{\lebe^Q(\mathbb S^{n-1})}\Big\},
\end{align*}
and the claimed estimate \eqref{L:optim:claim} follows by employing $\tfrac{ \sigma-\rho}{2} \leq \mathscr{L}^{1}(U)\leq \sigma-\rho$ from \eqref{bound:lowerU} together with \eqref{est:rEvr}. 

It remains to consider the case $n=2$. Here, we use the following consequence of Lemma~\ref{L:1dinterpolation} below: There exists $C=C(p)>0$ such that
\begin{equation*}
 \|v_r\|_{\lebe^\infty(\mathbb S^{1})}\leq C \Big(\|v_r\|_{\lebe^p(\mathbb S^1)}+\| v_r\|_{\lebe^p(\mathbb S^1)}^\frac{p}{p+2}\|\nabla v_r\|_{\lebe^2(\mathbb S^1)}^\frac{2}{p+2}\Big)\qquad\text{for all}\;r\in U.
\end{equation*} 
The above estimate in combination with the properties of $\eta$ in the form
\begin{align*}
\|v|\nabla \eta|^\alpha\|_{\lebe^\infty(B_R)}\leq&\Big(\frac{2}{(\sigma-\rho)}\Big)^\alpha\sup_{r\in U}\|v_r\|_{\lebe^\infty(\mathbb S^1)}
\end{align*}
and the choice of $U$ imply  \eqref{L:optim:claimn2}.

Lastly, the claimed inequality \eqref{L:optim:claimn2:exp} follows by a similar argument with minor modifications. Firstly, in this case, we recall that $n=2$ and  define  $U_2$ by
\begin{equation}\label{def:U2exp}
U_2\coloneqq \biggl\{r\in[\rho,\sigma ]\setminus \mathscr{N}\,:\,\int_{\mathbb S^{1}}\exp(|v_r(z)|)\,\mathrm{d}\mathscr H^{1}(z)\leq \frac{4}{(\sigma-\rho) r}\int_{\ball_\sigma\setminus \ball_{\rho}}\exp(|v|)\dx\biggr\}.
\end{equation}
Clearly, we still have \eqref{bound:lowerU}. Moreover, using \eqref{est:1dinterpolation:exp} from below, we find a constant $C>0$ such that
\begin{equation*}
\|\exp( |v_r|)\|_{\lebe^\infty(\mathbb S^{1})}\leq C \Big(\|\exp(|v_r|)\|_{\lebe^1(\mathbb S^1)}+\|\exp(|v_r|)\|_{\lebe^1(\mathbb S^1)}\|\nabla v_r\|_{\lebe^2(\mathbb S^1)}^2\Big) 
\end{equation*} 
holds for all $r\in U$. 
The claimed inequality \eqref{L:optim:claimn2:exp} then follows by inserting the properties of $\eta$, $U_1$ defined in \eqref{def:U} and $U_2$ defined in \eqref{def:U2exp}. This completes the proof. 
\end{proof}
\begin{lem}\label{L:1dinterpolation} Let $I\subset \R$ be a bounded interval, and let $1\leq p<\infty$. Then we have 
\begin{equation}\label{est:1dinterpolation}
\|u- (u)_I\|_{\lebe^\infty(I)}\leq \Big(\frac{p+2}2\Big)^\frac{2}{p+2}\|u-(u)_I\|_{\lebe^p(I)}^\frac{p}{p+2}\|u'\|_{\lebe^2(I)}^\frac{2}{p+2}
\end{equation}
for all $u\in\sobo^{1,2}(I)$. Moreover, we have 
\begin{equation}\label{est:1dinterpolation:exp}
 \|\exp(|u|)\|_{\lebe^\infty(I)}\leq \| \exp(|u|)\|_{\lebe^1(I)}\|u'\|_{\lebe^2(I)}^2+2\mathscr{L}^{1}(I)^{-1}\|\exp(|u|)\|_{\lebe^1(I)}. 
\end{equation}
\end{lem}

\begin{proof}
By density, it suffices to show \eqref{est:1dinterpolation} for $u\in \hold^1(I)$ and, without loss of generality, we may assume that $(u)_I=0$. By the integral mean value theorem, there exists $y\in I$ such that $u(y)=(u)_I=0$. By the fundamental theorem of calculus, we have for every $x\in I$ that 
\begin{equation*}
|u(x)|^\frac{p+2}2 \leq \frac{p+2}2\biggl|\int_y^x|u(s)|^\frac{p}2 |u'(s)|\,\dif s\biggr|\leq \frac{p+2}2 \|u\|_{\lebe^p(I)}^\frac{p}2\|u'\|_{\lebe^2(I)}.
\end{equation*}
We obtain \eqref{est:1dinterpolation} by taking the supremum over all $x \in I$ and raising the resulting inequality to the power $\tfrac{2}{p+2}$. The argument for \eqref{est:1dinterpolation:exp} is similar. We here choose $y \in I$ such that $\exp(|u(y)|) = (\exp(|u(\cdot)|))_I$ and we then obtain 
\begin{equation*}
\exp(|u(x)|)-\exp(|u(y)|)\leq \biggl| \int_y^x \exp(|u(s)|) |u'(s)|\,\dif s\leq \biggr| \|\exp(|u|)\|_{\lebe^2(I)}\|u'\|_{\lebe^2(I)},
\end{equation*} 
which implies
\begin{equation}\label{est:1dinterpolation:exp:1}
 \|\exp(|u|)\|_{\lebe^\infty(I)}\leq \| \exp(|u|)\|_{\lebe^2(I)}\|u'\|_{\lebe^2(I)}+\fint_I \exp(|u|)\dx.
\end{equation}
Inserting the elementary inequality 
$$
\| \exp(|u|)\|_{\lebe^2(I)}\leq \|\exp(|u|)\|_{\lebe^1(I)}^\frac{1}{2}\|\exp(|u|)\|_{\lebe^\infty(I)}^\frac{1}{2}
$$
into \eqref{est:1dinterpolation:exp:1}, we obtain with Young's inequality that
\begin{equation*}
 \|\exp(|u|)\|_{\lebe^\infty(I)}\leq \frac{1}{2}\|\exp(|u|)\|_{\lebe^\infty(I)}+ \frac{1}{2}\| \exp(|u|)\|_{\lebe^1(I)}\|u'\|_{\lebe^2(I)}^2+\fint_I \exp(|u|)\dx.
\end{equation*}
This completes the proof.
\end{proof}

\section{On the relaxed functional}\label{sec:relax}
In this section, we collect some elementary properties of the Lebesgue--Serrin--Marcellini extension \eqref{eq:LSM}, where we recall that $\Omega\subset\R^{n}$ is open and bounded  with Lipschitz boundary. Moreover, we assume that $1\leq q<\infty$, $u_{0}\in\sobo^{1,q}(\Omega;\R^{N})$, and that the convex function  $F\in\hold(\R^{N\times n})$ satisfies  \eqref{eq:1qgrowth}. We begin with the following routine, yet crucial observation: 
\begin{lem}[Recovery sequence]\label{lem:recovery}
Let $u\in\bv(\Omega;\R^{N})$ be such that $\overline{\mathscr{F}}_{u_{0}}^{*}[u;\Omega]<\infty$. Then there exists a \emph{recovery sequence}  $(u_{j})$ in $ u_{0}+\hold_{c}^{\infty}(\Omega;\R^{N})$ for $u$, meaning that $u_{j}\stackrel{*}{\rightharpoonup} u$ in the weak*-sense on $\bv(\Omega;\R^{N})$ and 
\begin{align}\label{eq:recoverymain}
\lim_{j\to\infty} \int_{\Omega}F(\nabla u_{j})\dif x = \overline{\mathscr{F}}_{u_{0}}^{*}[u;\Omega]. 
\end{align}
\end{lem}
\begin{proof}
For each $k\in\mathbb{N}$, there exists a sequence $(w_{j}^{k})$ in $u_{0}+\sobo_{0}^{1,q}(\Omega;\R^{N})$ such that $w_{j}^{k}\stackrel{*}{\rightharpoonup} u$ as $j\to\infty$ and 
\begin{align}\label{eq:recov1}
\liminf_{j\to\infty}\int_{\Omega}F(\nabla w_{j}^{k})\dif x \leq \overline{\mathscr{F}}_{u_{0}}^{*}[u;\Omega] +\frac{1}{k}. 
\end{align}
We find a subsequence $(w_{j_{i}}^{k})$ of $(w_{j}^{k})$ such that 
\begin{align}\label{eq:recov2a}
\lim_{i\to\infty}\int_{\Omega}F(\nabla w_{j_{i}}^{k})\dif x = \liminf_{j\to\infty}\int_{\Omega}F(\nabla w_{j}^{k})\dif x \stackrel{\eqref{eq:recov1}}{\leq}  \overline{\mathscr{F}}_{u_{0}}^{*}[u;\Omega] +\frac{1}{k}.
\end{align}
Next, for each $k\in\mathbb{N}$, we choose an index $i_{k}\in\mathbb{N}$ such that 
\begin{align}\label{eq:recov2}
\|w_{j_{i_{k}}}^{k}-u\|_{\lebe^{1}(\Omega)}<\frac{1}{k}\;\;\;\text{and}\;\;\;\int_{\Omega}F(\nabla w_{j_{i_{k}}}^{k})\dif x \leq \overline{\mathscr{F}}_{u_{0}}^{*}[u;\Omega] +\frac{2}{k}.
\end{align}
By our assumptions on $F$, we have the estimate \eqref{eq:liptype} with some $L\geq 0$. Recalling that $w_{j_{i_{k}}}^{k}\in u_{0}+\sobo_{0}^{1,q}(\Omega;\R^{N})$ and  mollifying $w_{j_{i_{k}}}^{k}-u_{0}$, we find $v_{k}\in u_{0}+\hold_{c}^{\infty}(\Omega;\R^{N})$ such that 
\begin{align}\label{eq:fire}
\|\nabla v_{k}\|_{\lebe^{q}(\Omega)}\leq \|\nabla w_{j_{i_{k}}}^{k}\|_{\lebe^{q}(\Omega)} + 2\|\nabla u_{0}\|_{\lebe^{q}(\Omega)}+1\;\;\;\text{and}\;\;\;\|v_{k}-w_{j_{i_{k}}}^{k}\|_{\sobo^{1,q}(\Omega)}<\frac{1}{kM_{k}},
\end{align}
where 
\begin{align*}
M_{k} \coloneqq  L\Big(\mathscr{L}^{n}(\Omega)^{\frac{1}{q}}+2\|\nabla w_{j_{i_{k}}}\|_{\lebe^{q}(\Omega)}+2\|\nabla u_{0}\|_{\lebe^{q}(\Omega)}+1\Big)^{q-1}. 
\end{align*}
In conclusion, we arrive at 
\begin{align*}
 \lefteqn{ \!\!\!\!\!\!\!\! \left\vert \int_{\Omega}F(\nabla v_{k})\dif x - \int_{\Omega}F(\nabla w_{j_{i_{k}}}^{k})\dif x  \right\vert \stackrel{\eqref{eq:liptype}}{\leq} L  \int_{\Omega}(1+|\nabla v_{k}|+|\nabla w_{j_{i_{k}}}^{k}|)^{q-1}|\nabla(v_{k}-w_{j_{i_{k}}}^{k})|\dif x} \\ 
& \leq L\bigg(\int_{\Omega}(1+|\nabla w_{j_{i_{k}}}^{k}|+|\nabla v_{k}|)^{q}\dif x \bigg)^{\frac{q-1}{q}}\bigg(\int_{\Omega}|\nabla(v_{k}-w_{j_{i_{k}}}^{k})|^{q}\dif x \bigg)^{\frac{1}{q}}  \\ 
& \!\!\!\! \stackrel{\eqref{eq:fire}_{1}}{\leq} L\Big(\mathscr{L}^{n}(\Omega)^{\frac{1}{q}}+2\|\nabla w_{j_{i_{k}}}^{k}\|_{\lebe^{q}(\Omega)}+ 2\|\nabla u_{0}\|_{\lebe^{q}(\Omega)}+1 \Big)^{q-1}\|\nabla(v_{k}-w_{j_{i_{k}}}^{k})\|_{\lebe^{q}(\Omega)} \\ 
& \!\!\!\! \stackrel{\eqref{eq:fire}_{2}}{\leq} \frac{1}{k}, 
\end{align*}
and so 
\begin{align}\label{eq:recovvo1}
\int_{\Omega}F(\nabla v_{k})\dif x \stackrel{\eqref{eq:recov2}_{2}}{\leq} \overline{\mathscr{F}}_{u_{0}}^{*}[u;\Omega] + \frac{3}{k}.  
\end{align}
On the other hand, $\eqref{eq:recov2}_{1}$ and  $\eqref{eq:fire}_{2}$ directly give us $v_{k}\to u$ strongly in $\lebe^{1}(\Omega;\R^{N})$. Moreover, by \eqref{eq:recovvo1} and \eqref{eq:1qgrowth}, the sequence $(v_{k})$ is bounded in $\sobo^{1,1}(\Omega;\R^{N})$, and so there exists a subsequence $(v_{k_{l}})$ of $(v_{k})$ and $v\in\bv(\Omega;\R^{N})$ such that $v_{k_{l}}\stackrel{*}{\rightharpoonup} v$ in $\bv(\Omega;\R^{N})$ as $l\to\infty$. Since $v_{k}\to u$ strongly in $\lebe^{1}(\Omega;\R^{N})$, we have $v=u$. Hence, setting $u_{l}\coloneqq v_{k_{l}}$, $u_{l}\stackrel{*}{\rightharpoonup} u$ in $\bv(\Omega;\R^{N})$ implies
\begin{align*}
\overline{\mathscr{F}}_{u_{0}}^{*}[u;\Omega] \leq \liminf_{l\to\infty} \mathscr{F}[u_{l};\Omega] \stackrel{\eqref{eq:recovvo1}}{\leq}  \limsup_{l\to\infty}\Big(\overline{\mathscr{F}}_{u_{0}}^{*}[u;\Omega] + \frac{3}{k_{l}}\Big)= \overline{\mathscr{F}}_{u_{0}}^{*}[u;\Omega]. 
\end{align*}
This is \eqref{eq:recoverymain}, and the proof of the lemma is complete. 
\end{proof}
\begin{lem}[Existence of relaxed minimizers]\label{lem:existmin}
There exists a minimizer $u\in\bv(\Omega;\R^{N})$ of $\overline{\mathscr{F}}_{u_{0}}^{*}[-;\Omega]$, meaning that $u$ satisfies 
\begin{align}\label{eq:minimizer}
\overline{\mathscr{F}}_{u_{0}}^{*}[u;\Omega] \leq \overline{\mathscr{F}}_{u_{0}}^{*}[v;\Omega]\qquad\text{for all}\;v\in\bv(\Omega;\R^{N}). 
\end{align}
\end{lem}
\begin{proof}
We firstly note that  $\overline{\mathscr{F}}_{u_{0}}^{*}[u_{0};\Omega]<\infty$ implies  $\overline{\mathscr{F}}_{u_{0}}^{*}[-;\Omega]\not\equiv\infty$ on $\bv(\Omega;\R^{N})$. It is clear by \eqref{eq:1qgrowth} that $\overline{\mathscr{F}}_{u_{0}}^{*}[-;\Omega]$ is bounded from below, and we denote $m\coloneqq \inf_{\bv(\Omega)}\overline{\mathscr{F}}_{u_{0}}^{*}[-;\Omega]$. We may thus pick a minimizing sequence $(v_{k})$ in $\bv(\Omega;\R^{N})$ such that 
$\overline{\mathscr{F}}_{u_{0}}^{*}[v_{k};\Omega]\to m$. By \eqref{eq:1qgrowth}, $(\D v_{k})$ is bounded in $\mathrm{RM}_{\mathrm{fin}}(\Omega;\R^{N\times n})$. By Lemma~\ref{lem:recovery}, for each $k\in\mathbb{N}$, there exists a sequence $(v_{k}^{j})$ in $u_{0}+\hold_{c}^{\infty}(\Omega;\R^{N})$  such that  $v_{k}^{j}\stackrel{*}{\rightharpoonup} v_{k}$ in $\bv(\Omega;\R^{N})$ as $j\to\infty$ together with 
\begin{align*}
\|v_{k}^{j}-v_{k}\|_{\lebe^{1}(\Omega)}<\frac{1}{j}\;\;\;\text{and}\;\;\;
\int_{\Omega}F(\nabla v_{k}^{j})\dif x  \leq \overline{\mathscr{F}}_{u_{0}}^{*}[v_{k};\Omega] + \frac{1}{j}\qquad\text{for all}\;j\in\mathbb{N}. 
\end{align*}
Hence, for each $k\in\mathbb{N}$, $w_{k}\coloneqq v_{k}^{k}$ satisfies 
\begin{align}\label{eq:brandner}
\|w_{k}-v_{k}\|_{\lebe^{1}(\Omega)}<\frac{1}{k}\;\;\;\text{and}\;\;\;
\int_{\Omega}F(\nabla w_{k})\dif x \leq \overline{\mathscr{F}}_{u_{0}}^{*}[v_{k};\Omega] + \frac{1}{k}. 
\end{align}
By \eqref{eq:1qgrowth} and Poincar\'{e}'s inequality, the sequence $(w_{k})$ in $u_{0}+\sobo_{0}^{1,q}(\Omega;\R^{N})$ is bounded in $\bv(\Omega;\R^{N})$. Thus, by the weak*-compactness theorem on $\bv(\Omega;\R^{N})$, there exists a subsequence $(w_{k_{i}})$ of $(w_{k})$ and $u\in\bv(\Omega;\R^{N})$ such that $w_{k_{i}}\stackrel{*}{\rightharpoonup} u$ in $\bv(\Omega;\R^{N})$. Then 
\begin{align*}
\overline{\mathscr{F}}_{u_{0}}^{*}[u;\Omega] \leq \liminf_{i\to\infty} \mathscr{F}[w_{k_{i}};\Omega] \stackrel{\eqref{eq:brandner}_{2}}{\leq} \liminf_ {i\to\infty} \Big(\overline{\mathscr{F}}_{u_{0}}^{*}[v_{k_{i}};\Omega] + \frac{1}{k_{i}}\Big) = m. 
\end{align*}
This implies \eqref{eq:minimizer}, and the proof is complete. 
\end{proof}

\begin{lem}\label{lem:RollingStones}
Let $u,u_{1},u_{2},\ldots \in u_{0}+\sobo_{0}^{1,q}(\Omega;\R^{N})$ be such that $u_{j}\to u$ strongly in $\lebe^{1}(\Omega;\R^{N})$. Then there holds 
\begin{align}\label{eq:LSC}
\mathscr{F}[u;\Omega] \leq \liminf_{j\to\infty}\mathscr{F}[u_{j};\Omega]. 
\end{align}
\end{lem}
\begin{proof}
We may assume that $\liminf_{j\to\infty}\mathscr{F}[u_{j};\Omega]<\infty$, as otherwise there is nothing to prove. Moreover, we let $(u_{j_{i}})$ be a subsequence of $(u_{j})$ such that
\begin{align}\label{eq:woodenroom}
\lim_{i\to\infty}\mathscr{F}[u_{j_{i}};\Omega]=\liminf_{j\to\infty}\mathscr{F}[u_{j};\Omega]. 
\end{align}
By our growth assumption \eqref{eq:1qgrowth} and Poincar\'{e}'s inequality, $(u_{j_{i}})$ is bounded in $\sobo^{1,1}(\Omega;\R^{N})$. Hence, there exists another subsequence $(u_{j_{i_{k}}})$ of $(u_{j_{i}})$ and $v\in\bv(\Omega;\R^{N})$ such that $u_{j_{i_{k}}}\stackrel{*}{\rightharpoonup} v$ in the weak*-sense on  $\bv(\Omega;\R^{N})$. Since $u_{j_{i_{k}}}\to u$ strongly in $\lebe^{1}(\Omega;\R^{N})$, we have $u=v$ and therefore $\D v=\nabla u\mathscr{L}^{n}$; in particular, $\D u_{j_{i_{k}}}\stackrel{*}{\rightharpoonup} \nabla u\mathscr{L}^{n}$ in $\mathrm{RM}_{\mathrm{fin}}(\Omega;\R^{N\times n})$.  Setting $\nu_{k}\coloneqq (\mathscr{L}^{n},\nabla  u_{j_{i_{k}}}\mathscr{L}^{n})$ and $\nu\coloneqq (\mathscr{L}^{n},\nabla u\mathscr{L}^{n})$, we thus have $\nu_{k}\stackrel{*}{\rightharpoonup}\nu$ in $\mathrm{RM}_{\mathrm{fin}}(\Omega;\R\times\R^{N\times n})$. Hence, recalling the linear perspective integrand $F^{\#}$ from \eqref{eq:linearperspective} and the discussion afterwards, Reshetnyak's lower semicontinuity theorem (see Proposition~\ref{prop:resh}) implies that 
\begin{align*}
\mathscr{F}[u;\Omega] = \int_{\Omega}F^{\#}\Big(\frac{\dif\nu}{\dif|\nu|} \Big)\dif|\nu| \leq \liminf_{k\to\infty}\int_{\Omega}F^{\#}\Big(\frac{\dif\nu_{k}}{\dif|\nu_{k}|} \Big)\dif|\nu_{k}|  =\liminf_{k\to\infty}\mathscr{F}[u_{j_{i_{k}}};\Omega].  
\end{align*}
In view of \eqref{eq:woodenroom}, this yields \eqref{eq:LSC}, and the proof is complete. 
\end{proof}
It is important to note that, different from the quasiconvex case, the previous lemma holds without restrictions on $q$; in the quasiconvex case, we would require $1\leq q<\frac{n}{n-1}$, see \cite{CHEKRI17,GMEKRI24}. We next provide a consistency result, and we point out that it is the identity \eqref{eq:consistency1} below which allows us to call $\overline{\mathscr{F}}_{u_{0}}^{*}[-;\Omega]$ an \emph{extension} of $\mathscr{F}[-;\Omega]$. Moreover, the relaxed functional does not feature a relaxation or Lavrentiev gap:
\begin{theorem}[Consistency and no-gap-result]\label{thm:consistency}
We have  
\begin{align}\label{eq:consistency1}
\mathscr{F}[u;\Omega] = \overline{\mathscr{F}}_{u_{0}}^{*}[u;\Omega] \qquad\text{for all}\;u\in u_{0}+\sobo_{0}^{1,q}(\Omega;\R^{N}),
\end{align}
and 
\begin{align}\label{eq:consistency2}
\inf_{ u_{0}+\sobo_{0}^{1,q}(\Omega;\R^{N})}\mathscr{F}[-;\Omega] = \min_{\bv(\Omega;\R^{N})}\overline{\mathscr{F}}_{u_{0}}^{*}[-;\Omega] \eqqcolon m \in (-\infty,\infty). 
\end{align}
\end{theorem}
\begin{proof}
For  \eqref{eq:consistency1},  we consider a function $u\in u_{0}+\sobo_{0}^{1,q}(\Omega;\R^{N})$ and an arbitrary sequence $(u_{j})$ in $u_{0}+\sobo_{0}^{1,q}(\Omega;\R^{N})$ with $u_{j}\stackrel{*}{\rightharpoonup} u$ in $\bv(\Omega;\R^{N})$. Then, in particular, $u_{j}\to u$ strongly in $\lebe^{1}(\Omega;\R^{N})$, and so Lemma~\ref{lem:RollingStones} yields $\mathscr{F}[u;\Omega]\leq \overline{\mathscr{F}}_{u_{0}}^{*}[u;\Omega]$. On the other hand, in the present situation, the constant sequence $(u_{j})=(u)$ is admissible in \eqref{eq:LSM}, whereby $\overline{\mathscr{F}}_{u_{0}}^{*}[u;\Omega]\leq \mathscr{F}[u;\Omega]$. Combining both inequalities, \eqref{eq:consistency1} follows. 

As to \eqref{eq:consistency2}, we note that '$\geq$' is clear by \eqref{eq:consistency1}. On the other hand, Lemma~\ref{lem:existmin} yields the existence of some function $u\in\bv(\Omega;\R^{N})$ with $\overline{\mathscr{F}}_{u_{0}}^{*}[u;\Omega] = m\in (-\infty,\infty)$. We then consider a recovery sequence $(u_{j})$ in $u_{0}+\hold_{c}^{\infty}(\Omega;\R^{N})$ as in Lemma~\ref{lem:recovery}. In conclusion, 
\begin{align*}
m = \overline{\mathscr{F}}_{u_{0}}^{*}[u;\Omega] \stackrel{\eqref{eq:recoverymain}}{=} \lim_{j\to\infty}\mathscr{F}[u_{j};\Omega] \geq \inf_{u_{0}+\sobo_{0}^{1,q}(\Omega;\R^{N})}\mathscr{F}[-;\Omega], 
\end{align*}
and so '$\leq$' follows too. The proof is complete. 
\end{proof}
As alluded to in the introduction, we conclude the present section with a discussion of elementary properties of maps $u\in\bv(\Omega;\R^{N})$ for which the relaxed functional is finite. Again, we stress that no integral representation of the functional~\eqref{eq:LSM} is required for this conclusion:

\begin{prop}[On finiteness of {$\overline{\mathscr{F}}_{u_{0}}^{*}[u;\Omega]$}]\label{prop:finiteness}
Suppose that $F\in\hold(\R^{N\times n})$
 satisfies \eqref{eq:1qgrowth}, and let 
 \begin{align}\label{eq:Csdefine}
\mathscr{C}_{s}\coloneqq \Big\{z=a\otimes b\colon\;a\in\mathbb{S}^{N-1},\,b\in\mathbb{S}^{n-1},\;\liminf_{t\to\infty}\frac{F(tz)}{t}=\infty\Big\}
 \end{align}
 be the set of all rank-one directions on which $F$ has \emph{proper superlinear growth}. If $u\in\bv(\Omega;\R^{N})$ satisfies $\overline{\mathscr{F}}_{u_{0}}^{*}[u;\Omega]<\infty$, then
 \begin{align}\label{eq:trivialise1}
|\D^{s}u|(A_{u}^{1})\coloneqq |\D^{s}u|\Big(\Big\{x\in\Omega\colon\;\frac{\dif \D u}{\dif|\D u|}(x)\in\mathscr{C}_{s} \Big\}\Big)=0 
 \end{align}
 and 
 \begin{align}\label{eq:trivialise2}
\mathscr{H}^{n-1}(A_{u}^{2})\coloneqq \mathscr{H}^{n-1}\Big(\Big\{x\in\partial\Omega\colon\;\mathrm{tr}_{\partial\Omega}(u_{0}-u)(x)\otimes\nu_{\partial\Omega}\in\mathscr{C}_{s}\Big\}\Big)=0.
 \end{align}
 \end{prop}
 \begin{proof}
Let $\Omega'\subset\R^{n}$ be open and bounded with Lipschitz boundary such that $\Omega\Subset\Omega'$. Extending $u_{0}$ to $\Omega'$, it is no loss of generality to directly assume that $u_{0}\in\sobo^{1,q}(\Omega';\R^{N})$. For an arbitrary function $w\in\bv(\Omega;\R^{N})$, we denote the extension of $w$ by the values of $u_0$ outside of~$\Omega$ by
\begin{align*}
\overline{w}\coloneqq \begin{cases} 
w&\;\text{in}\;\Omega,\\ 
u_{0}&\;\text{in}\;\Omega'\setminus\overline{\Omega}. 
\end{cases}
\end{align*}
We notice that, for any sequence $(w_{j})$ in $\bv(\Omega;\R^{N})$, there holds $w_{j}\stackrel{*}{\rightharpoonup} w$ in the weak*-sense on $\bv(\Omega;\R^{N})$ if and only if $\overline{w}_{j}\stackrel{*}{\rightharpoonup}\overline{w}$ in the weak*-sense on $\bv(\Omega';\R^{N})$. Since $u\in\bv(\Omega;\R^{N})$ satisfies  $\overline{\mathscr{F}}_{u_{0}}^{*}[u;\Omega]<\infty$ by assumption, Lemma~\ref{lem:recovery} provides us with a sequence $(u_{j})$ in $\sobo_{u_{0}}^{1,q}(\Omega;\R^{N})$ such that $u_{j}\stackrel{*}{\rightharpoonup} u$ in $\bv(\Omega;\R^{N})$ and $\mathscr{F}[u_{j};\Omega]\to\overline{\mathscr{F}}_{u_{0}}^{*}[u;\Omega]$ as $j\to\infty$. As a consequence, we have $\overline{u}_{j}\stackrel{*}{\rightharpoonup}\overline{u}$ in $\bv(\Omega';\R^{N})$, and so, in particular,  $(\mathscr{L}^{n},\D \overline{u}_{j})\stackrel{*}{\rightharpoonup}(\mathscr{L}^{n},\D \overline{u})$ in $\mathrm{RM}_{\mathrm{fin}}(\Omega';\R\times\R^{N\times n})$. The function $\overline{u}$ (and likewise the functions $\overline{u}_{j}$) have distributional gradients 
\begin{align}\label{eq:distgradcompute}
\begin{split}
&\D \overline{u} = \D u\mres\Omega + (\mathrm{tr}_{\partial\Omega}(u_{0}-u)\otimes\nu_{\partial\Omega})\mathscr{H}^{n-1}\mres\partial\Omega + \nabla u_{0}\mathscr{L}^{n}\mres(\Omega'\setminus\overline{\Omega}), \\ 
& \D \overline{u}_{j} = \nabla u_{j}\mathscr{L}^{n}\mres\Omega + \nabla u_{0}\mathscr{L}^{n}\mres(\Omega'\setminus\overline{\Omega}).
\end{split}
\end{align}
Based on Remark~\ref{rem:linearperspective}, see \eqref{eq:linearperspective}, these formulas in conjunction with Reshetnyak's lower semicontinuity theorem from  Proposition~\ref{prop:resh} give us 
\begin{align}
\int_{\Omega}F(\nabla u)\dif x  + \int_{\Omega}F^{\infty}\Big(\frac{\dif \D^{s}u}{\dif|\D^{s}u|}\Big)& \dif|\D^{s}u| + \int_{\partial\Omega}F^{\infty}(\mathrm{tr}_{\partial\Omega}(u_{0}-u)\otimes\nu_{\partial\Omega})\dif\mathscr{H}^{n-1} \notag \\ & + \int_{\Omega'\setminus\overline{\Omega}}F(\nabla u_{0})\dif x \stackrel{\eqref{eq:distgradcompute}_{1}}{=} \int_{\Omega'}F^{\#}\Big(\frac{\dif\,(\mathscr{L}^{n},\D \overline{u})}{\dif|(\mathscr{L}^{n},\D \overline{u})|}\Big)\dif|(\mathscr{L}^{n},\D \overline{u})| \notag \\ 
& \!\!\!\!\!\!\!\!\stackrel{\text{Prop.~\ref{prop:resh}}}{\leq} \liminf_{j\to\infty}\int_{\Omega'}F^{\#}\Big(\frac{\dif\,(\mathscr{L}^{n},\D \overline{u}_{j})}{\dif|(\mathscr{L}^{n},\D \overline{u}_{j})|}\Big)\dif|(\mathscr{L}^{n},\D \overline{u}_{j})| \label{eq:lowerbound1}\\ 
& \!\!\!\!\!\stackrel{\eqref{eq:distgradcompute}_{2}}{=} \liminf_{j\to\infty}\int_{\Omega}F(\nabla u_{j})\dif x + \int_{\Omega'\setminus\overline{\Omega}}F(\nabla u_{0})\dif x \notag\\ 
& = \overline{\mathscr{F}}_{u_{0}}^{*}[u;\Omega] + \int_{\Omega'\setminus\overline{\Omega}}F(\nabla u_{0})\dif x, \notag
\end{align}
where we have used in the last step that $(u_{j})$ is a recovery sequence for $\overline{\mathscr{F}}_{u_{0}}^{*}[u;\Omega]$. By assumption, we have $\overline{\mathscr{F}}_{u_{0}}^{*}[u;\Omega] < \infty$ and $u_{0}\in\sobo^{1,q}(\Omega';\R^{N})$, whereby the right-hand side of the previous inequality is finite. Since $F,F^{\infty}\geq 0$, the previous estimate particularly entails that 
\begin{align}\label{eq:havana}
\int_{A_{u}^{1}} F^{\infty}\Big(\frac{\dif \D^{s}u}{\dif|\D^{s}u|}\Big)\dif|\D^{s}u| + \int_{A_{u}^{2}}F^{\infty}(\mathrm{tr}_{\partial\Omega}(u_{0}-u)\otimes\nu_{\partial\Omega})\dif\mathscr{H}^{n-1} < \infty. 
\end{align}
By Alberti's rank-one-theorem, see \eqref{eq:alberti}, the density $\frac{\dif \D^{s}u}{\dif|\D^{s}u|}$ is contained in the rank-one-cone $|\D^{s}u|$-everywhere. By the very definition of $\mathscr{C}_{s}$, see \eqref{eq:Csdefine}, and those of $A_{u}^{1}$, $A_{u}^{2}$, \eqref{eq:havana} immediately yields $|\D^{s} u|(A_{u}^{1})=\mathscr{H}^{n-1}(A_{u}^{2})=0$. This is \eqref{eq:trivialise1} and \eqref{eq:trivialise2}, thereby completing the proof of the proposition. 
 \end{proof}
\begin{rem}[Proper superlinear growth from below]
If $1<p\leq q <\infty$ and \eqref{eq:1qgrowth} is modified to $\widetilde{\gamma}|z|^{p}\leq F(z)\leq \widetilde{\Gamma}(1+|z|^{q})$ for all $z\in\R^{N\times n}$ with $0<\widetilde{\gamma}\leq\widetilde{\Gamma}<\infty$, then 
\begin{align*}
    \mathscr{C}_{s} = \{z=a\otimes b\colon\; a\in\mathbb{S}^{N-1},\;b\in\mathbb{S}^{n-1}\}. 
\end{align*}
Moreover, if $u\in\bv(\Omega;\R^{N})$ is such that  $\overline{\mathscr{F}}_{u_{0}}^{*}[u;\Omega]<\infty$, then the lower growth bound on~$F$ implies that $u\in\sobo^{1,p}(\Omega;\R^{N})$, and there exists a sequence $(u_{j})$ in $\sobo_{u_0}^{1,q}(\Omega;\R^{N})$ such that $u_{j}\rightharpoonup u$ weakly in $\sobo^{1,p}(\Omega;\R^{N})$. It then follows that $\D^{s}u\equiv 0$, and so \eqref{eq:trivialise1} is trivially satisfied. Moreover, by the continuity properties of the trace operator with respect to weak convergence on $\sobo^{1,p}(\Omega;\R^{N})$, we have  $\mathrm{tr}_{\partial\Omega}(u)=\mathrm{tr}_{\partial\Omega}(u_{0})$ $\mathscr{H}^{n-1}$-a.e. on $\partial\Omega$, whereby \eqref{eq:trivialise2} follows too. This is in line with the natural fact that, in this situation, the relaxation to the larger space $\bv(\Omega;\R^{N})$ instead of $\sobo^{1,p}(\Omega;\R^{N})$ is not necessary.  
\end{rem}

\section{Proof of Theorem~\ref{thm:main}: Sobolev regularity in all dimensions}\label{sec:proofmain}
In this section, we establish Theorem~\ref{thm:main}, and so we briefly pause to clarify the structure of the proof. As discussed in the introduction, our strategy requires a finely adjusted Ekeland-type viscosity approximation strategy, to be given in Section~\ref{sec:Ekelandstart}. Section~\ref{sec:Cacc} proceeds to collect degenerate weighted second order bounds, which form the key background estimates for the higher integrability of relaxed minimizers in Section~\ref{sec:highint}. In order to obtain the localised form of the estimates from Theorem~\ref{thm:main}\ref{item:main2}--\ref{item:main3}, Section~\ref{sec:higher} establishes the pointwise convergence $\mathscr{L}^{n}$-a.e. for the gradients of the Ekeland sequence, thereby completing the proof of Theorem~\ref{thm:main}. Since we avoid integral representations, we finally prove a local minimality result in Section~\ref{sec:ELdimreduc} as a consequence of uniform higher integrability estimates. 
\subsection{Ekeland-type viscosity approximations}\label{sec:Ekelandstart}
Throughout this subsection, we tacitly assume that the convex variational integrand $F\in\hold^{1}(\R^{N\times n})$ satisfies the growth bound \eqref{eq:1qgrowth} with $1 \leq q < \infty$. Moreover, we recall that $\Omega\subset\R^{n}$ is open and bounded with Lipschitz boundary, and that $u_{0}\in\sobo^{1,q}(\Omega;\R^{N})$. 

Let  $u\in\bv(\Omega;\R^{N})$ be a relaxed minimizer of $\mathscr{F}[-;\Omega]$ so that, in particular,  
\begin{align}\label{eq:LSM1}  \overline{\mathscr{F}}_{u_{0}}^{*}[u;\Omega] \leq \overline{\mathscr{F}}_{u_{0}}^{*}[v;\Omega]\qquad\text{for all}\;v\in\bv(\Omega;\R^{N}). 
\end{align}
By Lemma~\ref{lem:recovery}, we find a sequence $(v_{j})$ in $u_{0}+\hold_{c}^{\infty}(\Omega;\R^{N})$ such that 
\begin{align}\label{eq:approx1}
v_{j}\stackrel{*}{\rightharpoonup}u\;\text{in}\;\bv(\Omega;\R^{N})\;\;\;\text{and}\;\;\;\overline{\mathscr{F}}_{u_{0}}^{*}[u;\Omega] = \lim_{j\to\infty}\int_{\Omega}F(\nabla v_{j})\dif x. 
\end{align}
Passing to a non-relabelled subsequence if required, we thus may assume that 
\begin{align}\label{eq:approx2}
\overline{\mathscr{F}}_{u_{0}}^{*}[u;\Omega] \leq \int_{\Omega}F(\nabla v_{j})\dif x < \overline{\mathscr{F}}_{u_{0}}^{*}[u;\Omega] + \frac{1}{10 j^{2}}\qquad\text{for all}\;j\in\mathbb{N}. 
\end{align}
In view of our stabilization approach, we define a convex $\hold^{2}$-function $\langle\cdot\rangle_{q} \colon  \R^{N\times n}\to \R_{\geq 0}$ by
\begin{align*}
\langle z \rangle_{q}\coloneqq \big((1+|z|^{2})^{\frac{1}{2}}-1\big)^{q}\qquad \text{for all } z\in\R^{N\times n}, 
\end{align*}
and briefly collect some of its properties in a form that shall be convenient later on.
\begin{lem}
There exist constants $L_{q},\Lambda_{q}>0$ such that the following hold for all $z,z'\in\R^{N\times n}$:
\begin{align}\label{eq:elementaryq}
\langle z \rangle_{q} \leq |z|^{q}, 
\end{align}
\begin{align}\label{eq:elementaryq1}
|z|\geq 1 \Longrightarrow (\sqrt{2}-1)^{q}|z|^{q}\leq \langle z\rangle_{q},
\end{align}
\begin{align}\label{eq:elementaryq2}
|\langle z\rangle_{q}-\langle z'\rangle_{q}|\leq L_{q}(1+|z|+|z'|)^{q-1}|z-z'|, 
\end{align}
\begin{align}\label{eq:q2ndderiv}
|\nabla^{2}\langle z\rangle_{q}|\leq \Lambda_{q}(1+|z|)^{q-2}, 
\end{align}
\begin{align}\label{eq:q2ndderivlb}
\langle \nabla^{2}(\langle z\rangle_{q}) z', z'\rangle\geq q\frac{((1+|z|^{2})^{\frac{1}{2}}-1)^{q-1}}{(1+|z|^{2})^{\frac{1}{2}}}|z'|^2\qquad\mbox{if $q\geq2$}. 
\end{align}
\end{lem}
\begin{proof}
Assertion \eqref{eq:elementaryq} directly follows from the definition of $\langle\cdot\rangle_{q}$, and \eqref{eq:elementaryq2} follows from Lemma~\ref{lem:liptype}. For \eqref{eq:elementaryq1}, note that if $|z|=1$, then $
\langle z\rangle_{q}=(\sqrt{2}-1)^{q}$. The function $[0,\infty)\ni t \mapsto \langle t\rangle_{1}$ is convex, and 
\begin{align*}
\frac{\dif}{\dif t}\langle t\rangle_{1} = \frac{t}{\sqrt{1+t^{2}}} \stackrel{t\geq 1}{\geq} \frac{1}{\sqrt{2}} \geq \sqrt{2}-1 = \frac{\dif}{\dif t}((\sqrt{2}-1)t)\qquad\text{for all } t\geq 1. 
\end{align*}
From here, \eqref{eq:elementaryq1} follows. 
Finally, we compute $\nabla^2(\langle z\rangle_q)$. Noting that $\langle z\rangle_q=f(v)$ with $f(s)=(s-1)^q$ and $v=\sqrt{1+|z|^2}$, we find
\begin{align*}
\langle \nabla^2(\langle z\rangle_q )z',z'\rangle & = q(q-1)(v-1)^{q-2}\frac{(\langle z,z'\rangle)^2}{v^2}+q(v-1)^{q-1}\Big(\frac{|z'|^2}{v}-\frac{(\langle z,z'\rangle)^2}{v^3}\Big)\\
& =q\frac{(v-1)^{q-1}}{v}|z'|^2+\frac{q(v-1)^{q-2}}{v^3}((q-2)v+1)(\langle z,z'\rangle)^2.
\end{align*}
From this identity, it is easy to deduce \eqref{eq:q2ndderiv} and \eqref{eq:q2ndderivlb}. The proof is complete. 
\end{proof}
We now define a  stabilized integrand via 
\begin{align}\label{eq:stabber_1}
F_{j}(z)\coloneqq  F(z) +  \frac{1}{10 \Cstabv j^{2}  }\langle z \rangle_{q} \qquad \text{for all } z\in\R^{N\times n},
\end{align}
where $S_j$ is defined in terms of $v_j$ as 
\begin{equation*}
\Cstabv \coloneqq  1 + \int_{\Omega}\langle\nabla v_j\rangle_{q}\dif x.
\end{equation*}
In view of our future purposes, we note that the key reason for working with $\langle\cdot\rangle_{q}$ instead of $|\cdot|^{q}$ in the second term in \eqref{eq:stabber_1} is that $\langle\cdot\rangle_{q}$ is of class $\hold^{2}$ regardless of the specific choice of $1\leq q<\infty$. We then consider the related stabilized functional given by
\begin{align}\label{eq:Fjstab1}
\mathscr{F}_{j}[v;\Omega] & \coloneqq \int_{\Omega}F_{j}(\nabla v)\dif x  \qquad \text{for } v \in \sobo_{u_{0}}^{1,q}(\Omega;\R^{N}) \coloneqq  u_{0}+\sobo_{0}^{1,q}(\Omega;\R^{N}).
\end{align}
By use of  $v_{j}\in \sobo_{u_{0}}^{1,q}(\Omega;\R^{N})$ in the third step, we obtain 
\begin{align*}
\inf_{\bv(\Omega;\R^{N})}\overline{\mathscr{F}}_{u_{0}}^{*}[-;\Omega] \stackrel{\eqref{eq:consistency2}}{=} \inf_{ \sobo_{u_{0}}^{1,q}(\Omega;\R^{N})}\mathscr{F}[-;\Omega] & \leq \inf_{ \sobo_{u_{0}}^{1,q}(\Omega;\R^{N})}\mathscr{F}_{j}[-;\Omega] \leq \mathscr{F}[v_{j};\Omega] + \frac{1}{10j^{2}} \\ & \!\!\!\stackrel{\eqref{eq:approx2}}{\leq} \overline{\mathscr{F}}_{u_{0}}^{*}[u;\Omega] + \frac{1}{5j^{2}}\\ & = \inf_{ \sobo_{u_{0}}^{1,q}(\Omega;\R^{N})}\mathscr{F}[-;\Omega] + \frac{1}{5j^{2}}. 
\end{align*}
We now smoothly approximate $u_{0}\in\sobo^{1,q}(\Omega;\R^{N})$ in the $\sobo^{1,q}$-norm. More precisely, let $\overline{u}_{0}\in\sobo^{1,q}(\R^{n};\R^{N})$ be an arbitrary but fixed extension of $u_{0}$ with $\|\overline{u}_{0}\|_{\sobo^{1,q}(\R^{n})}\leq c_{n,\Omega}\|u_{0}\|_{\sobo^{1,q}(\Omega)}$, where $c_{n,\Omega}\geq 1$ is a constant. With $\rho_{\varepsilon}(x)\coloneqq \varepsilon^{-n}\rho(\frac{x}{\varepsilon})$ denoting an $\varepsilon$-rescaled standard mollifier, we have for all $0<\varepsilon<1$ that 
\begin{align}\label{eq:adjust}
\|\rho_{\varepsilon}*\overline{u}_{0}\|_{\sobo^{1,q}(\Omega)} \leq \|\rho_{\varepsilon}*\overline{u}_{0}\|_{\sobo^{1,q}(\R^{n})} \leq \|\overline{u}_{0}\|_{\sobo^{1,q}(\R^{n})}\leq c_{n,\Omega}\|u_{0}\|_{\sobo^{1,q}(\Omega)} \eqqcolon M. 
\end{align}
Recalling the constant $\Gamma>0$ from  \eqref{eq:1qgrowth}, we define
\begin{align}\label{eq:Cdef}
\mathtt{C} \coloneqq  (\Gamma+1) (\mathscr{L}^{n}(\Omega) + M^{q})\qquad \text{and}\qquad 
\mathtt{C}' \coloneqq 5^q (M^q+ \mathscr{L}^{n}(\Omega) +1) (\mathtt{C}+1).
\end{align}
We then choose a constant $L_{F}>0$ such that \eqref{eq:liptype} holds with $L=L_{F}$. For each $j\in\mathbb{N}$, we let $0<\varepsilon_{j}<1$ be so small such that $u_{0,j}\coloneqq (\rho_{\varepsilon_{j}}*\overline{u}_{0})|_{\Omega}\in\hold^{\infty}(\overline{\Omega};\R^{N})$ satisfies 
\begin{equation}
\label{eq:berghammer}
\|u_{0}-u_{0,j}\|_{\sobo^{1,q}(\Omega)} < \frac{1}{\kappa_{j}},
\end{equation}
where
\begin{align*}
\kappa_{j} & \coloneqq  50j^{2}\Big[ L_F\big(3\mathscr{L}^{n}(\Omega)^{\frac{1}{q}}+2M+6 (10 \Cstabv j^{2}   (\mathtt{C}+1))^{\frac{1}{q}} \big)^{q-1}\Big]\\
& \qquad + 50j^{2}\Big[\Big(L_{F} + \frac{L_{q}}{10\Cstabv j^{2} }\Big) (\mathscr{L}^{n}(\Omega)^{\frac{1}{q}}+1+2\|\nabla v_{j}\|_{\lebe^{q}(\Omega)})^{q-1}\Big] \\ 
& \qquad + 50j^{2} \\
&  \eqqcolon \kappa_{j}^{(1)}+ \kappa_{j}^{(2)} +\kappa_{j}^{(3)}.
\end{align*}
The constant $\kappa_{j}$ is precisely adjusted in a way such that several emerging terms below take a particularly convenient form, and we shall comment on the entering of its single summands in detail throughout. Recalling that $v_{j}\in u_{0}+\hold_{c}^{\infty}(\Omega;\R^{N})$ for each $j \in \N$, we now  put $\psi_{j} \coloneqq v_{j}-u_{0}\in\hold_{c}^{\infty}(\Omega;\R^{N})$
 and define 
\begin{align}\label{eq:DirichletClassApprox}
\widetilde{v}_{j}\coloneqq  u_{0,j}+\psi_{j}\in\mathfrak{D}_{j}\coloneqq u_{0,j}+\sobo_{0}^{1,\max\{2,q\}}(\Omega;\R^{N}).
\end{align}
Based on $\eqref{eq:berghammer}$, we record that  
\begin{equation}\label{eq:compa}
\|v_{j}-\widetilde{v}_{j}\|_{\sobo^{1,q}(\Omega)} = \|v_{j}-u_{0,j} - \psi_{j} \|_{\sobo^{1,q}(\Omega)} = \|u_{0}-u_{0,j}\|_{\sobo^{1,q}(\Omega)} \leq \frac{1}{\kappa_{j}}. 
\end{equation}
In an intermediate step, we consider for fixed $j\in\mathbb{N}$ the variational principle 
\begin{align}\label{eq:varprininter}
\text{to minimize}\;\mathscr{F}_{j}[-;\Omega]\;\text{over}\;\mathfrak{D}_{j}. 
\end{align}
We observe that 
\begin{align}\label{eq:chiefo}
\begin{split}
\inf_{\mathfrak{D}_{j}}\mathscr{F}_{j}[-;\Omega] \stackrel{u_{0,j}\in\mathfrak{D}_{j}}{\leq} \mathscr{F}_{j}[u_{0,j};\Omega] & \stackrel{\eqref{eq:Fjstab1}}{=} \int_{\Omega}F(\nabla u_{0,j})\dif x + \frac{1}{10\Cstabv j^{2} }\int_{\Omega}\langle\nabla u_{0,j}\rangle_{q}\dif x \\
& \!\!\!\!\!\stackrel{\eqref{eq:1qgrowth}, \eqref{eq:elementaryq}}{\leq} (\Gamma+1)\int_{\Omega} (1+|\nabla u_{0,j}|^{q}) \dif x  \\ 
& \stackrel{\eqref{eq:adjust}}{\leq} (\Gamma+1) (\mathscr{L}^{n}(\Omega)+M^{q}) \stackrel{\eqref{eq:Cdef}}{=} \mathtt{C},
\end{split}
\end{align}
and it is crucial to note that $\mathtt{C}>0$ is independent of $j \in \N$. Obviously, by a similar reasoning, we also have
\begin{equation}
\label{eq:inf_estimate_C}
\inf_{\sobo_{u_{0}}^{1,q}(\Omega;\R^{N})}\mathscr{F}[-;\Omega] \leq \mathscr{F}[u_{0};\Omega] \leq \Gamma (\mathscr{L}^{n}(\Omega)+M^{q}) < \mathtt{C}.
\end{equation}

We next show that the original minimization problem~\eqref{eq:minimize} is approximated by the minimization problem~\eqref{eq:varprininter} with the stabilized functionals and the regularized Dirichlet boundary datum, and that the latter is almost minimized by the function $\widetilde{v}_{j}$.

\begin{lem}
\label{lemma_F_j_functional}
For each $j \in \N$ we have 
\begin{equation}
\label{eqn_inf_approximation}
    \bigg| \inf_{ \sobo_{u_{0}}^{1,q}(\Omega;\R^{N})}\mathscr{F}[-;\Omega] - \inf_{\mathfrak{D}_{j}}\mathscr{F}_{j}[-;\Omega] \bigg| \leq \frac{1}{j^2},
\end{equation}
and the function $\widetilde{v}_{j} \in \mathfrak{D}_{j}$ introduced above is  \emph{almost minimal}  with 
\begin{equation}
\label{eqn_v_tilde_almost_minimal}
    \mathscr{F}_{j}[\widetilde{v}_{j};\Omega] \leq \inf_{\mathfrak{D}_{j}}\mathscr{F}_{j}[-;\Omega] + \frac{1}{4j^{2}} .
\end{equation}
\end{lem}

\begin{proof}
In view of \eqref{eq:varprininter}, we choose a sequence $(\varphi_{k})$ in $\sobo_{0}^{1,\max\{2,q\}}(\Omega;\R^{N})$ such that 
\begin{align}\label{eq:unismall0}
\mathscr{F}_{j}[u_{0,j}+\varphi_{k};\Omega] \to \inf_{\mathfrak{D}_{j}}\mathscr{F}_{j}[-;\Omega] \stackrel{\eqref{eq:chiefo}}{\leq} \mathtt{C} \qquad \text{as } k\to\infty,
\end{align}
whereby it is no loss of generality to assume that 
\begin{align}\label{eq:unismall}
\mathscr{F}_{j}[u_{0,j}+\varphi_{k};\Omega] \leq \mathtt{C}+1\qquad\text{for all}\;k\in\mathbb{N}. 
\end{align}
Hence, the definition~\eqref{eq:Fjstab1} of~$\mathscr{F}_{j}$ with the  integrand~$F_j$ from~\eqref{eq:stabber_1} in combination with inequality~\eqref{eq:elementaryq1} implies 
\begin{equation*}
  \frac{ (\sqrt{2}-1)^{q}}{10 \Cstabv j^{2}} \int_{\Omega} \big( |\nabla(u_{0,j}+\varphi_{k})|^q -1 \big) \dif x 
   \leq  \frac{1}{10 \Cstabv j^{2}} \int_{\Omega} \langle\nabla(u_{0,j}+\varphi_{k})\rangle_{q} \dif x  \stackrel{\eqref{eq:unismall}}{\leq} \mathtt{C}+1 \qquad\text{for all}\;k\in\mathbb{N},
\end{equation*}
and with the elementary inequality $\tfrac{1}{3} \leq \sqrt{2}-1$, we then arrive at
\begin{align}
\label{eq:stonetemplepilots} 
\int_{\Omega}|\nabla(u_{0,j} + \varphi_{k})|^{q}\dif x \leq \mathscr{L}^{n}(\Omega) + 10  \Cstabv j^{2} 3^q (\mathtt{C}+1) \qquad\text{for all}\;k\in\mathbb{N}.
\end{align}
Now, for any $\varphi\in\sobo_{0}^{1,\max\{2,q\}}(\Omega;\R^{N})$, we have that 
\begin{align*}
\lefteqn{\hspace{-0.5cm}|\mathscr{F}[u_{0}+\varphi;\Omega]  - \mathscr{F}[u_{0,j} +\varphi;\Omega]|} \\
& \!\!\!\!\stackrel{\eqref{eq:liptype}}{\leq} L_{F}\int_{\Omega}(1+|\nabla u_{0}|+|\nabla u_{0,j}|+2|\nabla(u_{0,j}+\varphi)|)^{q-1}|\nabla(u_{0}-u_{0,j})|\dif x \\ 
&  \leq  L_{F}\Big(\int_{\Omega}(1+|\nabla u_{0}|+|\nabla u_{0,j}|+2|\nabla(u_{0,j}+\varphi)|)^{q}\dif x\Big)^{\frac{q-1}{q}}\|u_{0}-u_{0,j}\|_{\sobo^{1,q}(\Omega)}\\ 
& \leq  L_{F}\Big(\mathscr{L}^{n}(\Omega)^\frac{1}{q}+\|\nabla u_{0}\|_{\lebe^{q}(\Omega)}+\|\nabla u_{0,j}\|_{\lebe^{q}(\Omega)}+2\|\nabla (u_{0,j}+\varphi)\|_{\lebe^{q}(\Omega)}\Big)^{q-1} \\
& \qquad \times \|u_{0}-u_{0,j}\|_{\sobo^{1,q}(\Omega)} \\ 
& \!\!\!\! \stackrel{\eqref{eq:adjust}}{\leq} L_{F}(\mathscr{L}^{n}(\Omega)^{\frac{1}{q}}+2M +2\|\nabla(u_{0,j}+\varphi)\|_{\lebe^{q}(\Omega)})^{q-1}\|u_{0}-u_{0,j}\|_{\sobo^{1,q}(\Omega)}.
\end{align*}
In turn, we obtain
\begin{align}
\label{eqn_inf_approx_RW}
\begin{split}
\lefteqn{\hspace{-0.8cm}\inf_{\sobo_{u_{0}}^{1,q}(\Omega;\R^{N})} \mathscr{F}[-;\Omega] \leq \mathscr{F}[u_{0}+\varphi;\Omega]} \\ 
& \leq (\mathscr{F}[u_{0}+\varphi;\Omega] - \mathscr{F}[u_{0,j}+\varphi;\Omega]) + \mathscr{F}[u_{0,j}+\varphi;\Omega] \\ 
& \leq L_{F}(\mathscr{L}^{n}(\Omega)^{\frac{1}{q}}+2M +2\|\nabla(u_{0,j}+\varphi)\|_{\lebe^{q}(\Omega)})^{q-1}\|u_{0}-u_{0,j}\|_{\sobo^{1,q}(\Omega)} \\
& \qquad + \mathscr{F}[u_{0,j}+\varphi;\Omega]. 
\end{split}
\end{align} 
Applying the preceding chain of inequalities to $\varphi=\varphi_{k}$ for $k \in \N$ as fixed above and recalling that $F\leq F_{j}$, we therefore arrive at 
\begin{align*}
\inf_{\sobo_{u_{0}}^{1,q}(\Omega;\R^{N})} \mathscr{F}[-;\Omega] 
& \leq L_{F}(\mathscr{L}^{n}(\Omega)^{\frac{1}{q}}+2M +2\|\nabla(u_{0,j}+\varphi_k)\|_{\lebe^{q}(\Omega)})^{q-1}\|u_{0}-u_{0,j}\|_{\sobo^{1,q}(\Omega)}\\
&\qquad+ \mathscr{F}_{j}[u_{0,j}+\varphi_{k};\Omega] \\\  
& \!\!\!\! \stackrel{\eqref{eq:stonetemplepilots}}{\leq} L_{F}\Big(3\mathscr{L}^{n}(\Omega)^{\frac{1}{q}}+2M+2\big(10 \Cstabv j^{2}  3^q (\mathtt{C}+1)\big)^{\frac{1}{q}} \Big)^{q-1}\|u_{0}-u_{0,j}\|_{\sobo^{1,q}(\Omega)}\\
&\qquad+ \mathscr{F}_{j}[u_{0,j}+\varphi_{k};\Omega] .
\end{align*}
Since the first term on the right-hand side is bounded by $\tfrac{1}{50j^{2}}$ due to ~\eqref{eq:berghammer} (see $\kappa_{j}^{(1)}$) and the second term converges to $\inf_{\mathfrak{D}_{j}}\mathscr{F}_{j}[-;\Omega] $ according to~\eqref{eq:unismall0}, we infer from $\widetilde{v}_{j}\in\mathfrak{D}_{j}$ (see \eqref{eq:DirichletClassApprox}) that
\begin{align}
\label{eq:lonelyboy} 
\inf_{\sobo_{u_{0}}^{1,q}(\Omega;\R^{N})} \mathscr{F}[-;\Omega] & \leq \frac{1}{50j^{2}} + \inf_{\mathfrak{D}_{j}}\mathscr{F}_{j}[-;\Omega] \\
& \leq \frac{1}{50j^{2}} + \mathscr{F}_{j}[\widetilde{v}_{j};\Omega] \notag \\ 
& \leq \frac{1}{50j^{2}} + (\mathscr{F}_{j}[\widetilde{v}_{j};\Omega]-\mathscr{F}_{j}[v_{j};\Omega])+ \mathscr{F}_{j}[v_{j};\Omega] \eqqcolon \mathrm{I}_{j} + \mathrm{II}_{j} + \mathrm{III}_{j}. \notag
\end{align}

Let us now consider the single terms separately. For $\mathrm{II}_{j}$, we have 
\begin{align}
\mathrm{II}_{j} & \leq \int_{\Omega}|F(\nabla \widetilde{v}_{j})-F(\nabla v_{j})|\dif x + \frac{1}{10 \Cstabv j^{2} } \int_{\Omega}|\langle\nabla\widetilde{v}_{j}\rangle_{q}-\langle\nabla v_{j}\rangle_{q}|\dif x \notag\\ 
& \!\!\!\!\!\!\!\!\!\stackrel{\eqref{eq:liptype},\eqref{eq:elementaryq2}}{\leq} \Big(L_{F} + \frac{L_{q}}{10\Cstabv j^{2} }\Big)\int_{\Omega}(1+|\nabla\widetilde{v}_{j}|+|\nabla v_{j}|)^{q-1}|\nabla\widetilde{v}_{j}-\nabla v_{j}|\dif x \notag\\ 
& \leq  \Big(L_{F} + \frac{L_{q}}{10\Cstabv j^{2} }\Big) \big(\mathscr{L}^{n}(\Omega)^{\frac{1}{q}}+\|\nabla\widetilde{v}_{j}\|_{\lebe^{q}(\Omega)}+\|\nabla v_{j}\|_{\lebe^{q}(\Omega)}\big)^{q-1}\|v_{j}-\widetilde{v}_{j}\|_{\sobo^{1,q}(\Omega)}\notag\\ 
& \!\!\!\!\!\!\!\!\!\!\!\!\!\stackrel{\eqref{eq:compa},\,\text{see $\kappa_{j}^{(3)}$}}{\leq} \Big(L_{F} + \frac{L_{q}}{10 \Cstabv j^{2} }\Big) \big(\mathscr{L}^{n}(\Omega)^{\frac{1}{q}}+\frac{1}{50j^{2}}+2\|\nabla v_{j}\|_{\lebe^{q}(\Omega)} \big)^{q-1}\|v_{j}-\widetilde{v}_{j}\|_{\sobo^{1,q}(\Omega)}\notag\\ 
& \!\!\!\!\!\!\!\!\!\!\! \!\!\stackrel{\eqref{eq:compa},\,\text{see $\kappa_{j}^{(2)}$}}{\leq} \frac{1}{50j^{2}}.\notag
\end{align}
On the other hand, we have 
\begin{align}\label{eq:lonelyboy3}
\mathrm{III}_{j} \leq \mathscr{F}[v_{j};\Omega] + \frac{1}{10j^{2}} \stackrel{\eqref{eq:consistency2},\eqref{eq:approx2}}{\leq} \inf_{ \sobo_{u_{0}}^{1,q}(\Omega;\R^{N})}\mathscr{F}[-;\Omega] + \frac{1}{5j^{2}}. 
\end{align}
Combining the previous estimates with \eqref{eq:lonelyboy}, we arrive at 
\begin{align}\label{eq:vicecity}
\begin{split}
\inf_{ \sobo_{u_{0}}^{1,q}(\Omega;\R^{N})}\mathscr{F}[-;\Omega] & \leq \frac{1}{50j^{2}}+\inf_{\mathfrak{D}_{j}}\mathscr{F}_{j}[-;\Omega] \\
 & \leq \frac{1}{50j^{2}} + \mathscr{F}_{j}[\widetilde{v}_{j};\Omega] \stackrel{\eqref{eq:lonelyboy}-\eqref{eq:lonelyboy3}}{\leq} \inf_{ \sobo_{u_{0}}^{1,q}(\Omega;\R^{N})}\mathscr{F}[-;\Omega] + \frac{1}{4j^2}, 
\end{split}
\end{align}
which clearly implies the claims~\eqref{eqn_inf_approximation} and~\eqref{eqn_v_tilde_almost_minimal}.
\end{proof}

We now introduce a second stabilized integrand via 
\begin{align}\label{eq:stabber}
G_{j}(z)\coloneqq  F_{j}(z) +  \frac{1}{2 \CstabG j^{2}}(1+|z|^{2}) \qquad \text{for all } z\in\R^{N\times n},
\end{align}
where $\CstabG$ is defined in terms of $\widetilde{v}_{j}$ as 
\begin{align*}
 \CstabG \coloneqq 1+ \int_{\Omega} (1+|\nabla\widetilde{v}_{j}|^{2})\dif x. 
\end{align*}
This, in turn, gives rise to the following functional on $\sobo^{-1,1}(\Omega;\R^{N})$: 
\begin{align}\label{eq:infext}
\mathscr{G}_{j}[v;\Omega] \coloneqq  \begin{cases} 
\displaystyle \int_{\Omega}G_{j}(\nabla v)\dif x &\;\text{if}\;v\in\mathfrak{D}_{j}, \\ 
+\infty&\;\text{if}\;v\in\sobo^{-1,1}(\Omega;\R^{N})\setminus\mathfrak{D}_{j}. 
\end{cases}
\end{align}
Towards the application of the Ekeland variational principle, we require the following lemma. 
\begin{lem}[Lower semicontinuity]\label{lem:LSCGj}
For each $j\in\mathbb{N}$, $\mathscr{G}_{j}[-;\Omega]$ is lower semicontinuous with respect to the norm topology on $\sobo^{-1,1}(\Omega;\R^{N})$. 
\end{lem}
\begin{proof}
Let $w,w_{1},w_{2},\ldots \in\sobo^{-1,1}(\Omega;\R^{N})$ be such that $w_{j}\to w$ as $j \to \infty$ with respect to $\|\cdot\|_{\sobo^{-1,1}(\Omega)}$. It is no loss of generality to assume that $\liminf_{k\to\infty}\mathscr{G}_{j}[w_{k};\Omega]<\infty$. Moreover, passing to a suitable non-relabelled subsequence, we may further suppose that the liminf is a limit indeed, and that $w_{k}\in\mathfrak{D}_{j}$ for all $k\in\mathbb{N}$. Since $G_{j}$ has $\max\{2,q\}$-growth from below at infinity, see \eqref{eq:Fjstab1} and \eqref{eq:stabber}, it follows by Poincar\'{e}'s inequality that $(w_{k})$ is bounded in $\sobo^{1,\max\{2,q\}}(\Omega;\R^{N})$. By the Rellich--Kondrachov theorem, there exists a subsequence $(w_{k_{i}})$ of $(w_{k})$ and some $\widetilde{w}\in\sobo^{1,\max\{2,q\}}(\Omega;\R^{N})$ such that $w_{k_{i}}\rightharpoonup \widetilde{w}$ weakly in $\sobo^{1,\max\{2,q\}}(\Omega;\R^{N})$, whereby  $\widetilde{w}\in\mathfrak{D}_{j}$ too, and  $w_{k_{i}}\to \widetilde{w}$ strongly in $\lebe^{\max\{2,q\}}(\Omega;\R^{N})$. Since $\lebe^{\max\{2,q\}}(\Omega;\R^{N})\hookrightarrow\sobo^{-1,1}(\Omega;\R^{N})$, we have $w_{k_{i}}\to\widetilde{w}$ in $\sobo^{-1,1}(\Omega;\R^{N})$ and so $w=\widetilde{w}$ by uniqueness of limits. Therefore, classical lower semicontinuity results on convex functionals with $\max\{2,q\}$-growth (see, e.g., \cite[Thm. 5.7]{GIUSTI03}) imply that 
\begin{align*}
\mathscr{G}_{j}[w;\Omega] = \mathscr{G}_{j}[\widetilde{w};\Omega] \stackrel{\widetilde{w}\in\mathfrak{D}_{j}}{=} \int_{\Omega}G_{j}(\nabla \widetilde{w})\dif x \leq \liminf_{i\to\infty}\int_{\Omega}G_{j}(\nabla w_{k_{i}})\dif x = \liminf_{k\to\infty}\mathscr{G}_{j}[w_{k};\Omega]. 
\end{align*}
This completes the proof. 
\end{proof}
To proceed, we record that  
\begin{align}\label{eq:handlewithcare}
\begin{split}
\mathscr{G}_{j}[\widetilde{v}_{j};\Omega] & \stackrel{\eqref{eq:stabber}}{\leq}  \frac{1}{2j^{2}} + \mathscr{F}_{j}[\widetilde{v}_{j};\Omega] \stackrel{\eqref{eqn_v_tilde_almost_minimal}}{\leq} \frac{1}{j^{2}} + \inf_{\mathfrak{D}_{j}}\mathscr{F}_{j}[-;\Omega] \\ & \!\stackrel{F_{j}\leq G_{j}}{\leq} \frac{1}{j^{2}} + \inf_{\mathfrak{D}_{j}}\mathscr{G}_{j}[-;\Omega] \stackrel{\eqref{eq:infext}}{=} \frac{1}{j^{2}} + \inf_{\sobo^{-1,1}(\Omega;\R^{N})}\mathscr{G}_{j}[-;\Omega],
\end{split}
\end{align}
and $\mathscr{G}_{j}[-;\Omega]\not\equiv\infty$ on $\sobo^{-1,1}(\Omega;\R^{N})$. In particular, all hypotheses of the Ekeland variational principle from Lemma~\ref{lem:Ekeland} are satisfied. Hence, we obtain an element $u_{j}\in\sobo^{-1,1}(\Omega;\R^{N})$ with 
\begin{equation}\label{eq:Ekeland1a}
\|u_{j}-\widetilde{v}_{j}\|_{\sobo^{-1,1}(\Omega)}\leq\frac{1}{j}
\end{equation}
and 
\begin{equation}\label{eq:Ekeland1b}
\mathscr{G}_{j}[u_{j};\Omega] \leq \mathscr{G}_{j}[w;\Omega] + \frac{1}{j}\|w-{u}_{j}\|_{\sobo^{-1,1}(\Omega)} \text{ for all } w\in\sobo^{-1,1}(\Omega;\R^{N}).
\end{equation}
Applying the latter inequality to $w=\widetilde{v}_{j}$, we find that 
\begin{align}\label{eq:rambold1}
\begin{split}
\mathscr{G}_{j}[u_{j};\Omega] & \stackrel{\eqref{eq:Ekeland1b}}{\leq} \mathscr{G}_{j}[\widetilde{v}_{j};\Omega] + \frac{1}{j} \|\widetilde{v}_{j}-{u}_{j}\|_{\sobo^{-1,1}(\Omega)} \\
& \stackrel{\eqref{eq:Ekeland1a}}{\leq} \mathscr{G}_{j}[\widetilde{v}_{j};\Omega] + \frac{1}{j^{2}} \stackrel{\eqref{eq:stabber}}{\leq} \mathscr{F}_{j}[\widetilde{v}_{j};\Omega] + \frac{2}{j^{2}} \\ & \stackrel{\eqref{eqn_v_tilde_almost_minimal}}{\leq} \inf_{\mathfrak{D}_{j}}\mathscr{F}_{j}[-;\Omega] + \frac{3}{j^{2}}\stackrel{\eqref{eqn_inf_approximation}}{\leq}\inf_{\sobo_{u_{0}}^{1,q}(\Omega;\R^{N})}\mathscr{F}[-;\Omega] +  \frac{4}{j^{2}}. 
\end{split}
\end{align}
In particular, we have $\mathscr{G}_{j}[u_{j};\Omega]<\infty$,  and so $u_{j}\in\mathfrak{D}_{j}$ by the very definition of $\mathscr{G}_{j}[-;\Omega]$. For future reference, we record the following consequence of the above construction: 
\begin{prop}\label{prop:summary}
Let $F\in\hold^{1}(\R^{N\times n})$ satisfy \eqref{eq:1qgrowth} with $1\leq q<\infty$, and let $(u_{j})$ be the sequence constructed above. Then the following statements hold: 
\begin{enumerate}
    \item\label{item:EkelandSeq0} \emph{Uniform $\lebe^{1}$-bound.} With $\gamma>0$ as in \eqref{eq:1qgrowth}, we have 
    \begin{align}\label{eq:unifboundL1} 
    \|\nabla u_{j}\|_{\lebe^{1}(\Omega)} \leq \frac{1}{\gamma}\Big(\frac{4}{j^{2}} + \inf_{\sobo_{u_{0}}^{1,q}(\Omega;\R^{N})}\mathscr{F}[-;\Omega]\Big)\qquad\text{for all}\;j\in\mathbb{N}. 
    \end{align}
    \item\label{item:EkelandSeq1} \emph{Convergence.} A \textup{(}non-relabelled\textup{)} subsequence $(u_{j})$ converges \emph{in the weak*-sense on $\bv(\Omega;\R^{N})$} to the relaxed minimizer $u\in\bv(\Omega;\R^{N})$ as fixed at the beginning of the present Subsection~\ref{sec:Ekelandstart}. Moreover, we have 
    \begin{align}\label{eq:uptown}
    \lim_{j\to\infty}\mathscr{G}_{j}[u_{j};\Omega]=\inf_{\sobo_{u_{0}}^{1,q}(\Omega;\R^{N})}\mathscr{F}[-;\Omega] = \min_{\bv(\Omega;\R^{N})}\overline{\mathscr{F}}_{u_{0}}^{*}[-;\Omega]=\overline{\mathscr{F}}_{u_{0}}^{*}[u;\Omega]
    \end{align}
    and
    \begin{align}\label{eq:constitution}
    \limsup_{j\to\infty} \bigg(\frac{1}{10 \Cstabv j^{2} }\int_{\Omega}\langle\nabla u_{j}\rangle_{q}\dif x + \frac{1}{2 \CstabG j^{2}}\int_{\Omega}(1+|\nabla u_{j}|^{2})\dif x \bigg) =0. 
    \end{align}\item\label{item:EkelandSeq2} \emph{Perturbed Euler--Lagrange inequality.} For any $j\in\mathbb{N}$, we have 
\begin{align}\label{eq:EL1}
\left\vert\int_{\Omega}\langle\nabla G_{j}(\nabla u_{j}),\nabla\varphi\rangle\dif x \right\vert \leq \frac{1}{j}\|\varphi\|_{\sobo^{-1,1}(\Omega)}\qquad\text{for all}\;{\varphi\in\sobo_{c}^{1,\max\{2,q\}}(\Omega;\R^{N})}. 
\end{align}
\end{enumerate}
    
\end{prop}
\begin{proof}
On~\ref{item:EkelandSeq0}. Based on the estimate $F\leq F_{j}\leq G_{j}$, the estimate \eqref{eq:unifboundL1} follows by combining \eqref{eq:1qgrowth} with \eqref{eq:rambold1}. This establishes~\ref{item:EkelandSeq0}. 

On~\ref{item:EkelandSeq1}. We recall that $u_{j}\in\mathfrak{D}_{j}$ for all $j\in\mathbb{N}$. Therefore, Poincar\'{e}'s inequality gives us
\begin{align*}
\| u_{j}\|_{\lebe^{1}(\Omega)} & \leq \|u_{j}-u_{0,j}\|_{\lebe^{1}(\Omega)} + \|u_{0,j}\|_{\lebe^{1}(\Omega)} \\ & \leq c\|\nabla(u_{j}-u_{0,j})\|_{\lebe^{1}(\Omega)}+\|u_{0,j}\|_{\lebe^{1}(\Omega)} \leq c\|u_{0,j}\|_{\sobo^{1,1}(\Omega)} + c \|\nabla u_{j}\|_{\lebe^{1}(\Omega)}, 
\end{align*}
where $c\geq 1$ is independent of $j\in\mathbb{N}$. By \eqref{eq:adjust}, the first term is uniformly bounded in $j \in \N$. By~\ref{item:EkelandSeq0}, we thus conclude that $(u_{j})$ is bounded in $\sobo^{1,1}(\Omega;\R^{N})$. On the one hand, by the weak*-compactness theorem on $\bv$, there exists $v\in\bv(\Omega;\R^{N})$ such that, for some non-relabelled subsequence, $u_{j}\stackrel{*}{\rightharpoonup} v$ in $\bv(\Omega;\R^{N})$ as $j\to\infty$; in particular, $u_{j}\to v$ in $\sobo^{-1,1}(\Omega;\R^{N})$. On the other hand, \eqref{eq:approx1}, \eqref{eq:compa} and  \eqref{eq:Ekeland1a} imply that $u_{j}\to u$ in $\sobo^{-1,1}(\Omega;\R^{N})$. By uniqueness of limits in $\sobo^{-1,1}(\Omega;\R^{N})$, we deduce that $u=v$, and the first claim of~\ref{item:EkelandSeq1} follows. Turning to \eqref{eq:uptown} and recalling Theorem~\ref{thm:consistency}, we only have to prove the first equality. The estimate \eqref{eq:rambold1} implies '$\leq$' in \eqref{eq:uptown}. Moreover, 
\begin{align*}
\inf_{\sobo_{u_0}^{1,q}(\Omega;\R^{N})}\mathscr{F}[-;\Omega] \stackrel{\eqref{eqn_inf_approximation}}{\leq}  \frac{1}{j^{2}}+\inf_{\mathfrak{D}_{j}}\mathscr{F}_{j}[-;\Omega] \stackrel{F_{j}\leq G_{j}}{\leq}  \frac{1}{j^{2}}+ \inf_{\mathfrak{D}_{j}}\mathscr{G}_{j}[-;\Omega] \stackrel{u_{j}\in\mathfrak{D}_{j}}{\leq} \frac{1}{j^{2}}+\mathscr{G}_{j}[u_{j};\Omega], 
\end{align*}
from where the first equality of \eqref{eq:uptown} follows.  For \eqref{eq:constitution}, we first note that the inequalities $\tfrac{1}{10 \Cstabv j^{2}  }\langle z \rangle_{q} \leq F_{j}(z) \leq G_{j}(z)$,~\eqref{eq:elementaryq1} and the estimate~\eqref{eq:rambold1} in conjunction with the bound~\eqref{eq:inf_estimate_C} give us for $j\geq2$:
\begin{align*}
\int_{\Omega}|\nabla u_{j}|^{q}\dif x & \leq  \mathscr{L}^{n}(\Omega) +  \frac{1}{(\sqrt{2}-1)^{q}} \int_{\Omega}\langle\nabla u_{j} \rangle_{q} \dif x  \\
& \leq \mathscr{L}^{n}(\Omega) + \frac{10\Cstabv j^{2}  }{(\sqrt{2}-1)^{q}} \mathscr{G}_{j}[u_{j};\Omega] \\
& \leq \mathscr{L}^{n}(\Omega) + 10 \Cstabv j^{2} 3^q (\mathtt{C} + 1).
\end{align*}
To proceed, we employ \eqref{eqn_inf_approx_RW} with the choice $\varphi \coloneqq u_j - u_{j,0}$, and so derive from the previous estimate for $j\geq2$:
\begin{align*}
& \!\!\!\!\!\!\!\!\!\!\! \inf_{\sobo_{u_{0}}^{1,q}(\Omega;\R^{N})}\mathscr{F}[-;\Omega] \\
& \leq L_{F}(\mathscr{L}^{n}(\Omega)^{\frac{1}{q}}+2M +2\|\nabla u_{j}\|_{\lebe^{q}(\Omega)})^{q-1}\|u_{0}-u_{0,j}\|_{\sobo^{1,q}(\Omega)} + \mathscr{F}[u_{j};\Omega] \\
& \leq L_{F}\Big( 3 \mathscr{L}^{n}(\Omega)^{\frac{1}{q}} + 2M + 2\big(10\Cstabv j^{2}   3^q (\mathtt{C} + 1)\big)^{\frac{1}{q}}\Big)^{q-1} \|u_{0}-u_{0,j}\|_{\sobo^{1,q}(\Omega)} + \mathscr{F}[u_{j};\Omega] \\ 
& \!\!\!\!\!\!\!\!\!\!\!\!\!\stackrel{\eqref{eq:berghammer},\;\text{see $\kappa_{j}^{(1)}$}}{\leq} \frac{1}{50j^{2}} + \mathscr{F}[u_{j};\Omega] \stackrel{F\leq G_{j}}{\leq} \frac{1}{50j^{2}} + \mathscr{G}_{j}[u_{j};\Omega] \stackrel{\eqref{eq:uptown}}{\longrightarrow} \inf_{\sobo_{u_{0}}^{1,q}(\Omega;\R^{N})}\mathscr{F}[-;\Omega].
\end{align*}
By the very definition of $G_{j}$, this clearly implies 
\eqref{eq:constitution}. 

On~\ref{item:EkelandSeq2}. Consider $\varphi\in\sobo_{c}^{1,\max\{2,q\}}(\Omega;\R^{N})$, and let $0 < \varepsilon < 1$ be arbitrary. Testing the inequality \eqref{eq:Ekeland1b} with $w=u_{j}\pm\varepsilon\varphi$ and dividing by $\varepsilon$, we arrive at 
\begin{align}\label{eq:preELeq}
\left\vert \frac{1}{\varepsilon}\int_{\Omega} \big( G_{j}(\nabla (u_{j}\pm\varepsilon\varphi))-G_{j}(\nabla u_{j}) \big) \dif x \right\vert \leq \frac{1}{j}\|\varphi\|_{\sobo^{-1,1}(\Omega)}. 
\end{align}
Since $G_{j}$ has $\max\{2,q\}$-growth from above and below, Lemma~\ref{lem:liptype} implies the existence of a constant $L(j)>0$ such that 
\begin{equation*}
\frac{1}{\varepsilon}|G_{j}(\nabla(u_{j}\pm\varepsilon\varphi))-G_{j}(\nabla u_{j})| \leq L(j)\big(1+2|\nabla u_{j}|+|\nabla \varphi|)^{\max\{2,q\}-1}|\nabla\varphi|
\end{equation*}
holds $\mathscr{L}^{n}$-a.e. in $\Omega$. By H\"{o}lder's inequality and $u_{j},\varphi\in\sobo^{1,\max\{2,q\}}(\Omega;\R^{N})$, the function on the right-hand side of the previous inequality belongs to $\lebe^{1}(\Omega)$. Sending $\varepsilon\searrow 0$ in \eqref{eq:preELeq} then yields 
\eqref{eq:EL1} by dominated convergence. The proof is complete. 
\end{proof}
Some remarks on the above strategy are in order. 
\begin{rem}[Comparison of the above strategy with standard linear growth]\label{rem:strategy}
Compared with previous contributions in the purely linear growth context, the much more intricate overall structure of the above approximation is due to \eqref{eqn_inf_approx_RW}\emph{ff}.. It is here where we link the infima of the original problem and the approximate problems; in particular, the set-up has to be sufficiently robust in order to not create relaxation gaps during this process. By way of comparison, we sketch one key point and briefly recall the analogous situation in the standard linear growth case. Here, $q=1$, and so $F$ is Lipschitz continuous by Lemma~\ref{lem:liptype}. Based on  suitable regularizations $u_{0,j}$ of the boundary values, one obtains for $\varphi\in\sobo_{0}^{1,1}(\Omega;\R^{N})$ that 
\begin{align}\label{eq:steppinout}
\begin{split}
\inf_{\sobo_{u_0}^{1,1}(\Omega;\R^{N})}\mathscr{F}[-;\Omega] & \leq (\mathscr{F}[u_{0}+\varphi;\Omega]-\mathscr{F}[u_{0,j}+\varphi;\Omega]) + \mathscr{F}[u_{0,j}+\varphi;\Omega] \\ 
& \leq \mathrm{Lip}(F)\|u_{0}-u_{0,j}\|_{\sobo^{1,1}(\Omega)} + \mathscr{F}[u_{0,j}+\varphi;\Omega] \\ 
& \leq \frac{2}{j^{2}} + \mathscr{F}[u_{0,j}+\varphi;\Omega], 
\end{split}
\end{align}
where the ultimate inequality follows by construction (see \cite[Sec. 5]{BECSCH13} or \cite[Sec. 4]{GMEINEDER20}). Subsequently infimizing the right-hand side of \eqref{eq:steppinout} over $\varphi\in\sobo_{0}^{1,1}(\Omega;\R^{N})$ then yields 
\begin{align*}
\inf_{\sobo_{u_0}^{1,1}(\Omega;\R^{N})}\mathscr{F}[-;\Omega] \leq \frac{2}{j^{2}} + \inf_{u_{0,j}+\sobo_{0}^{1,1}(\Omega;\R^{N})}\mathscr{F}[-;\Omega], 
\end{align*}
and so the Lipschitz property of $F$ implies that $\varphi$ has no additional impact beyond entering the controllable functional. If we try to argue analogously in the present situation of growth assumption \eqref{eq:1qgrowth} with $q>1$, we obtain 
\begin{align}\label{eq:problem}
\inf_{\sobo_{u_0}^{1,q}(\Omega;\R^{N})}\mathscr{F}[-;\Omega] \leq \frac{4}{j^{2}} + c\,\|\nabla\varphi\|_{\lebe^{q}(\Omega)}^{q-1}\|u_{0}-u_{0,j}\|_{\sobo^{1,q}(\Omega)} + \mathscr{F}[u_{0,j}+\varphi;\Omega]
\end{align}
by a suitable adjustment of the constants. Again, we wish to infimize over $\varphi$, finally aiming to obtain $\inf_{u_{0,j}+\sobo_{0}^{1,q}(\Omega;\R^{N})}\mathscr{F}[-;\Omega]$ on the right-hand side, which is crucial for our approach. During this infimization process, it might happen that $\|\nabla\varphi\|_{\lebe^{q}(\Omega)}$ blows up (recall that we only have \eqref{eq:1qgrowth} with $q>1$), in which case the comparability assertions from Lemma~\ref{lemma_F_j_functional} get lost; in the situation of \eqref{eq:1qgrowth}, the corresponding functionals only provide us with $\lebe^{1}$-bounds on $\nabla\varphi$, and the latter is not sufficient to get quantitative bounds on $\|\nabla\varphi\|_{\lebe^{q}(\Omega)}$ as appearing in \eqref{eq:problem}. In particular, the resulting inequality turns out useless. To repair this shortcoming, it is necessary to get some quantitative handling on the blow-up for minimizing sequences.  Therefore, they have to be \emph{enforced}. On the other hand, by non-uniqueness of relaxed minimizers, an Ekeland-type approximation is necessary here; despite being finely adjusted, precursors thereof as developed in  \cite{BECEITGME,BECSCH13,GMEKRI19,GMEINEDER20} in the purely linear growth context do \emph{not} suffice here, see \eqref{eq:steppinout}--\eqref{eq:problem}. In this sense, the superlinear growth from above and non-uniqueness can be understood as two effects which are even more coupled than in the linear growth case $q=1$. Proposition~\ref{prop:summary} in turn asserts that both effects can be handled simultaneously in a form that it is amenable to the gradient estimates to be established in Section~\ref{sec:Cacc}\emph{ff}. below. 
\end{rem}
\begin{rem}[One versus two stabilizations]\label{rem:onetwo}
The second stabilization \eqref{eq:stabber} is only required when $q<2$, and primarily serves as a tool in obtaining (non-uniform) $\sobo_{\locc}^{2,2}(\Omega;\R^{N})$-estimates in the following subsection. Since our main result, Theorem~\ref{thm:main}, also covers the case $1\leq q\leq 2$, the passage to $G_{j}$ is important indeed. 
\end{rem}
\begin{rem}[$q=1$ versus $q>1$]\label{rem:recovery}
From a technical perspective, working solely with the LSM-extension \eqref{eq:LSM} also comes with simplifications in the more classical case $q=1$, where all previous contributions stick to the integral representation throughout. In this situation, one usually starts an Ekeland approximation scheme by finding a sequence $(v_{j})$ in $\sobo_{u_{0}}^{1,1}(\Omega;\R^{N})$ such that $v_{j}\to u$ area-strictly on $\bv(\Omega;\R^{N})$; see, e.g., \cite{BECSCH13,GMEINEDER20}. This particularly implies that $\overline{\mathscr{F}}_{u_{0}}^{*}[v_{j};\Omega]\to\overline{\mathscr{F}}_{u_{0}}^{*}[u;\Omega]$ by the continuity part of Reshetnyak's theorem. In the case $q>1$, the latter is not available, see Remark~\ref{rem:continuitypartResh} and Example~\ref{ex:reshnon}, but \eqref{eq:approx1} shows that this not required for $q=1$ either when working with \eqref{eq:LSM} instead. 
\end{rem}
\subsection{A Caccioppoli-type inequality}\label{sec:Cacc} We now record a Caccioppoli-type inequality for the Ekeland-type sequence constructed in the previous subsection, see Proposition~\ref{prop:cacc} below. To this end, we require a (non-uniform) local regularity result for its single members. Because of the non-standard growth bound \eqref{eq:1qgrowth} and the fact that Lemma~\ref{lem:Ekeland} only provides us with an Euler--Lagrange-inequality rather than an equation, we include the proof for the reader's convenience: 
\begin{lem}[Non-uniform higher local regularity]\label{lem:nonuni}
Let $F\in\hold^{2}(\R^{N\times n})$ satisfy \eqref{eq:1qgrowth} with some $q> 1$, and suppose that there exists a constant $\Lambda>0$ such that 
\begin{align}\label{eq:growthboundf''Sec3}
0 < \langle \nabla^{2}F(z)\xi,\xi\rangle \leq \Lambda(1+|z|^{2})^{\frac{q-2}{2}}|\xi|^{2}\qquad \text{for all } z,\xi\in\R^{N\times n}. 
\end{align}
Then the Ekeland-type vanishing viscosity sequence $(u_{j})$ from Section~\ref{sec:Ekelandstart} satisfies
\begin{align}\label{eq:intermedreg}
u_{j}\in\sobo_{\locc}^{2,2}(\Omega;\R^{N})\;\;\;\text{and}\;\;\;\langle\nabla^{2}G_{j}(\nabla u_{j})\partial_{\ell}\nabla u_{j},\partial_{\ell}\nabla u_{j}\rangle\in\lebe_{\locc}^{1}(\Omega) 
\end{align}
for all $\ell\in\{1,\ldots ,n\}$ and all $j\in\mathbb{N}$. In particular, $\nabla G_{j}(\nabla u_{j})\in\sobo_{\locc}^{1,1}(\Omega;\R^{N\times n})$ for all $j\in\mathbb{N}$.
\end{lem}
\begin{proof}
We start by noting that, because of \eqref{eq:growthboundf''Sec3} and the definition of $G_{j}$ (see \eqref{eq:Fjstab1} and \eqref{eq:stabber}), we have that 
\begin{align}\label{eq:lifegoeson}
\lambda_{j}(1+|z|^{2})^{\frac{\max\{q,2\}-2}{2}}|\xi|^{2}\leq \langle \nabla^{2}G_{j}(z)\xi,\xi\rangle \leq \Lambda_{j}(1+|z|^{2})^{\frac{\max\{q,2\}-2}{2}}|\xi|^{2} 
\end{align}
for all $z,\xi\in\R^{N\times n}$, 
where $0<\lambda_{j}\leq\Lambda_{j}<\infty$ are constants. Indeed, the upper bound in \eqref{eq:lifegoeson} follows directly from the definition of $G_j$ together with \eqref{eq:q2ndderiv}. For the lower bound we distinguish the cases $1\leq q\leq2$ and $q>2$. For $q\leq2$, the lower bound in \eqref{eq:lifegoeson} is a consequence of the convexity of $F_j$ and of the quadratic stabilization in~$G_j$. For $q>2$, we use \eqref{eq:q2ndderivlb} in the form
$$
 \langle \nabla^{2}G_{j}(z)\xi,\xi\rangle\geq c_j(1+|z|^2)^\frac{q-2}{2}\Big((1+|z|^2)^\frac{2-q}{2}+(1-(1+|z|^2)^{-\frac{1}{2}})^{q-1}\Big)|\xi|^2
$$
for some $c_j>0$ together with the observation $\inf_{t\geq1}(t^{2-q}+(1-t^{-1})^{q-1})>0$. Now let $x_{0}\in\Omega$. We  choose $0<r<\frac{1}{3}\mathrm{dist}(x_{0},\partial\Omega)$, and let $\eta\in\hold_{c}^{\infty}(\Omega;[0,1])$ be such that $\mathbbm{1}_{\ball_{r}(x_{0})}\leq \eta \leq \mathbbm{1}_{\ball_{2r}(x_{0})}$ and $|\nabla\eta|\leq \frac{4}{r}$. For all sufficiently small $h>0$ and every $\ell\in\{1,\ldots ,n\}$, the choice $
\varphi \coloneqq \Delta_{\ell,h}^{-}(\eta^{2}\Delta_{\ell,h}u_{j})\in\sobo_{c}^{1,\max\{2,q\}}(\Omega;\R^{N})$ is admissible in \eqref{eq:EL1}. The integration by parts formula for finite difference quotients then yields  
\begin{align}\label{eq:ericclapton}
\left\vert \int_{\Omega}\langle\Delta_{\ell,h}\nabla G_{j}(\nabla u_{j}),\nabla(\eta^{2}\Delta_{\ell,h}u_{j})\rangle\dif x \right\vert \leq \frac{1}{j}\|\nabla u_{j}\|_{\lebe^{1}(\Omega)},
\end{align}
where we have used \eqref{eq:diffquotW-11} for the $\|\cdot\|_{\sobo^{-1,1}(\Omega)}$-term on the right-hand side of \eqref{eq:EL1} in conjunction with the standard $\lebe^1$-estimate of finite difference quotients of a function via its gradient. Next, for $\mathscr{L}^{n}$-a.e. $x\in\ball_{2r}(x_{0})$, we define an elliptic bilinear form $\mathrm{B}_{j,\ell,h,x}$ on $(\R^{N\times n})^{2}$ by 
\begin{align*}
\mathrm{B}_{j,\ell,h,x}[\xi_{1},\xi_{2}] \coloneqq  \left\langle\int_{0}^{1} (\nabla^{2}G_{j})(\nabla u_{j}(x) + th\Delta_{\ell,h}\nabla u_{j}(x)) \xi_{1}\dif t, \xi_{2}\right\rangle\qquad \text{for all } \xi_{1}, \xi_{2}\in\R^{N\times n}. 
\end{align*}
From \cite[Lem. 2.VI]{CAMPANATO82}, we deduce that there exists a constant $\widetilde{\lambda}_{q}>0$ such that 
\begin{align*}
\widetilde{\lambda}_{q} (1+|z|^{2}+|z'|^{2})^{\frac{\max\{q,2\}-2}{2}} \leq \int_{0}^{1}(1+|(t-1)z+tz'|^{2})^{\frac{\max\{q,2\}-2}{2}}\dif t\qquad\text{for all } z,z'\in\R^{N\times n}. 
\end{align*}
In combination with the lower bound in~\eqref{eq:lifegoeson}, we then find
\begin{align}\label{eq:jugend}
\lambda_{j} \widetilde{\lambda}_{q} (1+|\nabla u_{j}(x)|^{2}+|\nabla u_{j}(x+he_{\ell})|^{2})^{\frac{\max\{q,2\}-2}{2}}|\xi|^{2} \leq \mathrm{B}_{j,\ell,h,x}[\xi,\xi],
\end{align}
while the upper bound in~\eqref{eq:lifegoeson} trivially implies the existence of some $\Lambda'_{j}>0$ such that 
\begin{align}\label{eq:B_upper}
\mathrm{B}_{j,\ell,h,x}[\xi,\xi] \leq \Lambda'_j (1+|\nabla u_{j}(x)|^{2}+|\nabla u_{j}(x+he_{\ell})|^{2})^{\frac{\max\{q,2\}-2}{2}}|\xi|^{2},
\end{align}
for $\mathscr{L}^{n}$-a.e. $x\in\ball_{2r}(x_{0})$ and all $\xi\in\R^{N\times n}$. Rewriting the left-hand side of \eqref{eq:ericclapton} by use of $\mathrm{B}_{j,\ell,h,x}$, we find via  \eqref{eq:jugend},~\eqref{eq:B_upper} and the application of Young's inequality (with an absorption argument) to the bilinear forms $\mathrm{B}_{j,\ell,h,x}$: 
\begin{align*}
\lambda_{j} \widetilde{\lambda}_{q} \int_{\ball_{r}(x_{0})}& (1+|\nabla u_{j}|^{2}+|\nabla u_{j}(\cdot+he_{\ell})|^{2})^{\frac{\max\{q,2\}-2}{2}}|\Delta_{\ell,h}\nabla u_{j}|^{2}\dif x \\ & \leq 
\int_{\ball_{2r}(x_{0})} \mathrm{B}_{j,\ell,h,\cdot}[\eta\Delta_{\ell,h}\nabla u_{j}, \eta\Delta_{\ell,h}\nabla u_{j}]\dif x \\ & \leq 4\int_{\ball_{2r}(x_{0})}\mathrm{B}_{j,\ell,h,\cdot}[\eta(\Delta_{\ell,h} u_{j})\otimes\nabla\eta,\eta(\Delta_{\ell,h}u_{j})\otimes\nabla\eta]\dif x + \frac{2}{j} \|\nabla u_{j}\|_{\lebe^{1}(\Omega)} \\ 
& \leq \frac{64\Lambda'_{j}}{r^{2}}\bigg(\int_{\ball_{2r}(x_{0})}(1+|\nabla u_{j}|^{2}+|\nabla u_{j}(\cdot+he_{\ell})|^{2})^{\frac{\max\{q,2\}}{2}}\dif x\bigg)^{\frac{\max\{q,2\}-2}{\max\{q,2\}}}\times \\ 
& \qquad \times \bigg( \int_{\ball_{2r}(x_{0})}|\Delta_{\ell,h}u_{j}|^{\max\{q,2\}}\dif x\bigg)^{\frac{2}{\max\{q,2\}}} + \frac{2}{j} \|\nabla u_{j}\|_{\lebe^{1}(\Omega)}. 
\end{align*}
Since $u_{j}\in\sobo^{1,\max\{2,q\}}(\Omega;\R^{N})$, the right-hand side of the previous chain of inequalities is uniformly bounded in $|h|\ll 1$. By arbitrariness of $x_{0}$ and $\ell\in\{1,\ldots ,n\}$, this implies that $u_{j}\in\sobo_{\locc}^{2,2}(\Omega;\R^{N})$ as claimed. In consequence, passing to a suitable subsequence in $h$, we may assume that $\Delta_{\ell,h}\nabla u_{j}\to \partial_{\ell}\nabla u_{j}$ $\mathscr{L}^{n}$-a.e. in $\ball_{r}(x_{0})$. Hence, Fatou's lemma gives us 
\begin{align*}
\int_{\ball_{r}(x_{0})}(1+|\nabla u_{j}|^{2})^{\frac{\max\{q,2\}-2}{2}}|\partial_{\ell}\nabla u_{j}|^{2}\dif x < \infty,  
\end{align*}
and this implies then the second claim in \eqref{eq:intermedreg} in view of \eqref{eq:lifegoeson}. Lastly, since $|\nabla G_{j}(z)|\leq c(1+|z|^{\max\{q,2\}-1})$, it follows that $\nabla G_{j}(\nabla u_{j})\in\lebe_{\locc}^{1}(\Omega;\R^{N\times n})$. Moreover, by the upper bound in \eqref{eq:lifegoeson} and by H\"{o}lder's inequality, we have for any compact subset $K\subset\Omega$: 
\begin{align*}
\lefteqn{\int_{K} |\partial_{\ell}(\nabla G_{j}(\nabla u_{j}))|\dif x  = \int_{K}|\nabla^{2}G_{j}(\nabla u_{j})(\partial_{\ell}\nabla u_{j})|\dif x} \\ 
& \leq \Lambda_{j} \int_{K}(1+|\nabla u_{j}|^{2})^{\frac{\max\{q,2\}-2}{2}}|\partial_{\ell}\nabla u_{j}| \dif x \\ 
& \leq \Lambda_{j} \bigg(\int_{K}(1+|\nabla u_{j}|^{2})^{\frac{\max\{q,2\}-2}{2}}|\partial_{\ell}\nabla u_{j}|^{2} \dif x\bigg)^{\frac{1}{2}}\bigg(\int_{K}(1+|\nabla u_{j}|^{2})^{\frac{\max\{q,2\}-2}{2}}\dif x\bigg)^{\frac{1}{2}}, 
\end{align*}
and the ultimate expression is finite due to \eqref{eq:intermedreg} and $u_{j}\in\sobo^{1,\max\{2,q\}}(\Omega;\R^{N})$. Hence, $\nabla G_{j}(\nabla u_{j})\in\sobo_{\locc}^{1,1}(\Omega;\R^{N\times n})$, and the proof of the lemma is complete.  
\end{proof}
We are now ready to state and prove the main result of the present subsection. 
\begin{prop}[of Caccioppoli-type]\label{prop:cacc} 
Let $F\in\hold^{2}(\R^{N\times n})$ satisfy \eqref{eq:1qgrowth} and \eqref{eq:muell1q} with some $q>1$ and $\mu\geq 1$. Then the following statements hold: 
\begin{enumerate}
    \item\label{item:Cacc1} There exists a constant $c=c(n,N,q)>0$ such that for all $j\in\mathbb{N}$, $\ell\in\{1,\ldots,n\}$,  $\xi\in\R^{N}$ and all $\eta\in\sobo_{c}^{1,\infty}(\Omega)$ there holds
\begin{align}
     \int_{\Omega}\eta^{2}\langle \nabla^{2}F(\nabla u_{j})\partial_{\ell}\nabla u_{j},\partial_{\ell}\nabla u_{j}\rangle \dif x & \leq 2\left\vert\int_{\Omega} \langle \nabla^{2}F(\nabla u_{j})\partial_{\ell}\nabla u_{j},\eta(\partial_{\ell} u_{j}-\xi)\otimes\nabla \eta\rangle \dif x\right\vert \notag\\ 
    & \qquad + \frac{c\,(1+|\xi|^{2})}{10  \Cstabv j^{2}  }\int_{\Omega}(1+|\nabla u_{j}|^{2})^{\frac{q}{2}}|\nabla\eta|^{2}\dif x \notag \\ 
    & \qquad + \frac{c\,(1+|\xi|^{2})}{\CstabG j^{2}}\int_{\Omega}(1+|\nabla u_{j}|^{2})|\nabla\eta|^{2}\dif x \label{eq:weighted0}\\ 
    & \qquad + \frac{1}{j} \|\eta\|_{\lebe^{\infty}(\Omega)}^{2} \big( \|\nabla u_{j}\|_{\lebe^{1}(\Omega)}+\mathscr{L}^{n}(\Omega)|\xi|\big). \notag
    \end{align}
    \item\label{item:Cacc2} There exists a constant $c=c(\lambda,\Lambda,n,N,q)>0$ such that for all $j\in\mathbb{N}$ and all $\eta\in\sobo_{c}^{1,\infty}(\Omega)$ there holds
\begin{align}\label{eq:weighted}
\begin{split}
\sum_{\ell=1}^{n} \int_{\Omega}\eta^{2}\frac{|\partial_{\ell}\nabla u_{j}|^{2}}{(1+|\nabla u_{j}|^{2})^{\frac{\mu}{2}}}\dif x & \leq c\,\bigg(\int_{\Omega}(1+|\nabla u_{j}|^{2})^{\frac{q}{2}}|\nabla\eta|^{2}\dif x \bigg. \\ &  \bigg. \qquad + \frac{1}{\CstabG j^{2}}\|\nabla\eta\|_{\lebe^{\infty}(\Omega)}^{2}\int_{\Omega}|\nabla u_{j}|^{2}\dif x + \frac{1}{j} \|\eta\|_{\lebe^{\infty}(\Omega)}^{2} \|\nabla u_{j}\|_{\lebe^{1}(\Omega)} \bigg).
\end{split}
\end{align}
\end{enumerate}
\end{prop}
\begin{proof}
On~\ref{item:Cacc1}. Let $j \in \N$ and $\ell\in\{1,\ldots ,n\}$. Since we have $\nabla G_{j}(\nabla u_{j})\in\sobo_{\locc}^{1,1}(\Omega;\R^{N\times n})$ by Lemma~\ref{lem:nonuni}, we obtain for any $\psi\in\hold_{c}^{\infty}(\Omega;\R^{N})$ by an integration by parts that 
\begin{align}\label{eq:hbftab}
\begin{split}
 \left\vert \int_{\Omega}\langle \nabla^{2}G_{j}(\nabla u_{j})\partial_{\ell}\nabla u_{j},\nabla\psi\rangle\dif x \right\vert & = \left\vert \int_{\Omega}\langle \nabla G_{j}(\nabla u_{j}),\nabla\partial_{\ell}\psi\rangle\dif x \right\vert \\ 
& \!\!\!\! \stackrel{\eqref{eq:EL1}}{\leq}  \frac{1}{j}\|\partial_{\ell}\psi\|_{\sobo^{-1,1}(\Omega)} \stackrel{\eqref{eq:W-11cancel}}{\leq} \frac{1}{j}\|\psi\|_{\lebe^{1}(\Omega)}. 
\end{split}
\end{align}
Recalling $\nabla u_{j}\in(\sobo_{\locc}^{1,2}\cap\lebe^{q})(\Omega;\R^{N\times n})$ and the weighted estimate underlying  \eqref{eq:intermedreg}, we find via \eqref{eq:lifegoeson} and smooth approximation that 
\begin{align}\label{eq:hbftab1}
 \left\vert \int_{\Omega}\langle \nabla^{2}G_{j}(\nabla u_{j})\partial_{\ell}\nabla u_{j},\nabla\psi\rangle\dif x \right\vert \leq \frac{1}{j}\|\psi \|_{\lebe^{1}(\Omega)}\qquad\text{for all}\;\psi\in\sobo_{c}^{1,\max\{q,2\}}(\Omega;\R^{N}). 
 \end{align}
In order to justify the use of specific test functions for the derivation of the desired Caccioppoli-type inequalities, we let $\eta\in\sobo_{c}^{1,\infty}(\Omega)$ and $\xi\in\R^{N}$ be arbitrary. We choose a sequence $(h_{i})$ in $(0,1)$ with $h_{i}\searrow 0$ and the following properties: 
\begin{align}\label{eq:kriminaltechniker}
\begin{split}
 \eta^{2}(\Delta_{\ell,h_{i}}u_{j}-\xi) \to \eta^{2}(\partial_{\ell}u_{j}-\xi)&\;\qquad\text{strongly in $\lebe^{\max\{q,2\}}(\spt(\eta);\R^{N})$} \\ 
 \nabla(\eta^{2}(\Delta_{\ell,h_{i}}u_{j}-\xi))\to \nabla(\eta^{2}(\partial_{\ell}u_{j}-\xi)) &\;\qquad\text{$\mathscr{L}^{n}$-a.e. in $\Omega$}.
\end{split}
\end{align}
Put $\mu\coloneqq (1+|\nabla u_{j}|^{2})^{\frac{\max\{q,2\}-2}{2}}\mathscr{L}^{n}\mres\spt(\eta)$ and consider the following functional on the associated weighted Lebesgue space $\lebe_{\mu}^{2}(\Omega;\R^{N\times n})$: 
\begin{align*}
\Xi(w)\coloneqq \int_{\spt(\eta)}\langle\nabla^{2}G_{j}(\nabla u_{j})\partial_{\ell}\nabla u_{j},w\rangle\dif x \qquad \text{for all } w\in\lebe_{\mu}^{2}(\Omega;\R^{N\times n}).
\end{align*}
Because of \eqref{eq:intermedreg}, the growth bound \eqref{eq:lifegoeson} and the Cauchy--Schwarz inequality, it follows that $\Xi\in\lebe_{\mu}^{2}(\Omega;\R^{N\times n})'$. Moreover, it follows from the proof of Lemma~\ref{lem:nonuni} that 
\begin{align*}
\sup_{i\in\mathbb{N}} \|\nabla(\eta^{2}(\Delta_{\ell,h_{i}}u_{j}-\xi))\|_{\lebe_{\mu}^{2}(\Omega)} < \infty, 
\end{align*}
whereby the Banach--Alaoglu--Bourbaki theorem implies the existence of some $w\in\lebe_{\mu}^{2}(\Omega;\R^{N\times n})$ such that, for a non-relabelled subsequence,  
\begin{align}\label{eq:whatdoesnot}
\nabla(\eta^{2}(\Delta_{\ell,h_{i}}u_{j}-\xi))\rightharpoonup w \qquad\text{weakly in $\lebe_{\mu}^{2}(\Omega;\R^{N\times n})$}. 
\end{align}
Since $\max\{q,2\}\geq 2$, we then deduce $w=\nabla(\eta^{2}(\partial_{\ell}u_{j}-\xi))$ from $\eqref{eq:kriminaltechniker}_{2},\eqref{eq:whatdoesnot}$ and Lemma~\ref{lem:weighted}. Based on $\Xi\in\lebe_{\mu}^{2}(\Omega;\R^{N\times n})'$, we find that
\begin{align*}
\lefteqn{\int_{\Omega}\langle\nabla^{2}G_{j}(\nabla u_{j})\partial_{\ell}\nabla u_{j}, \nabla(\eta^{2}(\partial_{\ell}u_{j}-\xi))\rangle \dif x } \\ 
& = \lim_{i\to\infty} \int_{\Omega}\langle\nabla^{2}G_{j}(\nabla u_{j})\partial_{\ell}\nabla u_{j},\nabla(\eta^{2}(\Delta_{\ell,h_{i}}u_{j}-\xi))\rangle \dif x & \\ 
& \!\!\!\!  \stackrel{\eqref{eq:hbftab1}}{\leq} \lim_{i\to\infty} \frac{1}{j} \|\eta^{2}(\Delta_{\ell,h_{i}}u_{j}-\xi)\|_{\lebe^{1}(\Omega)} \stackrel{\eqref{eq:kriminaltechniker}_{1}}{=} \frac{1}{j}\|\eta^{2}(\partial_{\ell}u_{j}-\xi)\|_{\lebe^{1}(\Omega)}, & 
\end{align*}
where we have used that $\eta^{2}(\Delta_{\ell,h_{i}}u_{j}-\xi)\in\sobo_{c}^{1,\max\{q,2\}}(\Omega;\R^{N})$ is admissible in \eqref{eq:hbftab1}. Writing out the definition of $G_{j}$, this gives us 
\begin{align*}
\int_{\Omega}\eta^{2}\langle \nabla^{2}F(\nabla u_{j})& \partial_{\ell}\nabla u_{j},\partial_{\ell}\nabla u_{j}\rangle \dif x + \frac{1}{10\Cstabv j^{2}  }\int_{\Omega}\eta^{2}\langle\nabla^{2}\langle\nabla u_{j}\rangle_{q}\partial_{\ell}\nabla u_{j},\partial_{\ell}\nabla u_{j}\rangle \dif x \\ 
& \qquad + \frac{1}{2\CstabG j^{2}}\int_{\Omega}\eta^{2}|\partial_{\ell}\nabla u_{j}|^{2}\dif x \\ 
& \leq 2\left\vert\int_{\Omega} \langle \nabla^{2}F(\nabla u_{j})\partial_{\ell}\nabla u_{j},\eta(\partial_{\ell} u_{j}-\xi)\otimes\nabla \eta\rangle \dif x\right\vert \\ 
& \qquad + \frac{1}{5\Cstabv j^{2}  }\int_{\Omega}\eta\langle\nabla^{2}\langle\nabla u_{j}\rangle_{q}\partial_{\ell}\nabla u_{j},(\partial_{\ell}u_{j}-\xi)\otimes\nabla\eta\rangle \dif x \\ 
& \qquad + \frac{2}{\CstabG j^{2}}\int_{\Omega}\eta\langle \partial_{\ell}\nabla u_{j},(\partial_{\ell}u_{j}-\xi)\otimes\nabla\eta\rangle \dif x + \frac{1}{j}\|\eta^{2}(\partial_{\ell}u_{j}-\xi)\|_{\lebe^{1}(\Omega)}. 
\end{align*}
We organise the preceding inequality as 
\begin{align}\label{eq:splitto}
\mathrm{I}+\mathrm{II}+\mathrm{III} \leq \mathrm{IV} + \mathrm{V} + \mathrm{VI} + \mathrm{VII}. 
\end{align}
We start by considering $\mathrm{V}$ and $\mathrm{VI}$. Recalling that $\langle\cdot\rangle_{q}$ is convex and of class $\hold^{2}$, we may apply Young's inequality $\mathscr{L}^{n}$-a.e. to the positive definite bilinear forms 
\begin{align*}
(z_{1},z_{2})\mapsto \langle\nabla^{2}\langle\nabla v\rangle_{q}z_{1},z_{2}\rangle + \langle z_{1},z_{2}\rangle \qquad \text{for all }z_{1},z_{2}\in\R^{N\times n}
\end{align*}
and absorb all second derivatives quantities in $\mathrm{V}$ and $\mathrm{VI}$ into $\mathrm{II}$ and $\mathrm{III}$. Since the resulting terms on the left-hand side are non-negative, we may ignore them in the sequel. The remaining first order terms $\widetilde{\mathrm{V}}$ and $\widetilde{\mathrm{VI}}$ emerging from Young's inequality can be estimated as follows:
\begin{align}\label{eq:tabea1}
\begin{split}
\widetilde{\mathrm{V}}  +\widetilde{\mathrm{VI}} \; & \!\!\stackrel{\eqref{eq:q2ndderiv}}{\leq} \frac{c}{10\Cstabv j^{2}  }\int_{\Omega}(1+|\nabla u_{j}|^{2})^{\frac{q-2}{2}}|(\partial_{\ell}u_{j}-\xi)|^{2}|\nabla\eta|^{2}\dif x \\ 
& \qquad + \frac{c}{\CstabG j^{2}}\int_{\Omega}|(\partial_{\ell}u_{j}-\xi)|^{2}|\nabla\eta|^{2}\dif x \\ 
& \leq \frac{c\,(1+|\xi|^{2})}{10\Cstabv j^{2}  }\int_{\Omega}(1+|\nabla u_{j}|^{2})^{\frac{q}{2}}|\nabla\eta|^{2}\dif x \\ 
& \qquad + \frac{c\,(1+|\xi|^{2})}{\CstabG j^{2}}\int_{\Omega}(1+|\nabla u_{j}|^{2})|\nabla\eta|^{2}\dif x. 
\end{split}
\end{align}
Finally, we have 
\begin{align*}
\mathrm{VII} \leq \frac{1}{j} \|\eta\|_{\lebe^{\infty}(\Omega)}^{2} \big(\|\nabla u_{j}\|_{\lebe^{1}(\Omega)} + \mathscr{L}^{n}(\Omega)|\xi|\big). 
\end{align*}
Putting these estimates together, we then infer \eqref{eq:weighted0} from \eqref{eq:splitto}.
The assertion of~\ref{item:Cacc2} can be obtained from the statement in~\ref{item:Cacc1} by setting $\xi=0$, applying the Cauchy--Schwarz and Young inequality, absorbing all second order quantities into the left-hand side and using~\eqref{eq:muell1q}. This completes the proof. 
\end{proof}
\begin{rem}
The preceding lemma is formulated in a way such that it is simultaneously applicable in the cases $q\geq 2$ and $1\leq q\leq 2$. However, if $1\leq q\leq 2$, then the above proof simplifies considerably: In this case, Lemma~\ref{lem:nonuni} directly implies that $\nabla G_{j}(\nabla u_{j})\in\sobo_{\locc}^{1,2}(\Omega;\R^{N\times n})$. Thus, in this case, we may use $\psi=\eta^{2}(\partial_{\ell}u_{j}-\xi)$ as an admissible test function in \eqref{eq:hbftab}, without the approximation argument via finite difference quotients.
\end{rem}
\subsection{Uniform bounds and higher gradient integrability}\label{sec:highint}
We now establish the higher integrability assertions from Theorem~\ref{thm:main}. In order to maintain the linear structure of the proof and to avoid forward references, we firstly provide a variant of Theorem~\ref{thm:main} that avoids the appearance of $\nabla u$ on the right-hand sides of \eqref{eq:finalboundA1}--\eqref{eq:finalboundA1log}. The latter shall be a consequence of stronger assertions on the convergence of the weak gradients, see Section~\ref{sec:higher} below. 

\begin{prop}[Higher gradient integrability]\label{prop:semimain}
In the situation of Theorem~\ref{thm:main}, let $u\in\bv(\Omega;\R^{N})$ be a relaxed minimizer of $\mathscr{F}$ subject to the Dirichlet datum $u_{0}\in\sobo^{1,q}(\Omega;\R^{N})$. Then we have $u\in \sobo^{1,1}(\Omega;\R^{N})$, and there exists a constant $c=c(\gamma,\Gamma,\lambda,\Lambda,n,q,\mu)>0$ and an exponent $\mathtt{d}=\mathtt{d}(n,\mu,q)\in[1,\infty)$ such that the Ekeland-type approximation sequence $(u_{j})$ from Section~\ref{sec:Ekelandstart} satisfies the following estimates for all balls $\ball_{R}(x_{0})\Subset\Omega$: 
\begin{enumerate}
    \item\label{item:propmain1} If $1\leq\mu<1+\frac{2}{n}$ and $n\geq 3$, then 
    \begin{align*}
    \bigg(\dashint_{\ball_{R/2}(x_0)}|\nabla u|^{\frac{(2-\mu)n}{n-2}}\dif x\bigg)^\frac{n-2}{(2-\mu)n}     \leq c\,\limsup_{j\to\infty}\bigg(1+\bigg(\dashint_{{\ball_{R}(x_0)}}|\nabla u_{j}|\dif x \bigg)^{\mathtt{d}}\bigg) < \infty.
    \end{align*}
    \item\label{item:propmain2} If $1\leq \mu<2$ and $n=2$, then 
\begin{align*}
{\;\;-}{\!\!\!\|\,|\nabla u|^{2-\mu}\|}_{\exp\lebe^{1}(\ball_{R/2}(x_{0}))}^{\frac{1}{2-\mu}}
 &   \leq c\,\limsup_{j\to\infty}\bigg(1+\bigg(\dashint_{{\ball_{R}(x_0)}}|\nabla u_{j}|\dif x \bigg)^{\mathtt{d}}\bigg) < \infty.
\end{align*}
    \item\label{item:propmain3}  If $n=2$ and $\mu=2$, then we have for every $1\leq t <\infty$ that
\begin{align*}
\bigg(\dashint_{\ball_{R/2}(x_0)}|\nabla u|^{t}\dif x\bigg)^\frac{1}{t} &   \leq \limsup_{j\to\infty} \exp\bigg(c\,t\,\bigg(1+\bigg(\dashint_{\ball_{R}(x_0)}|\nabla u_{j}|\dif x \bigg)^{\mathtt{d}}\bigg)\bigg) < \infty.
\end{align*}
\end{enumerate}
\end{prop}

\begin{proof}[Proof of Proposition~\ref{prop:semimain}~\ref{item:propmain1} and~\ref{item:propmain2}]  Let $u\in\bv(\Omega;\R^{N})$ be a relaxed minimizer as in the proposition, and let $(u_{j})$ be the Ekeland-type approximation sequence constructed in Section~\ref{sec:Ekelandstart}. Moreover, let $\ball_{R}=\ball_{R}(x_{0})\Subset\Omega$ be an open ball. We now divide the proof of this part into four steps.

\emph{Step 1. $V$-function estimates.} Recalling that $\mu <2$, we introduce the auxiliary function 
\begin{align}\label{eq:Vfunction}
{V_{\mu} \colon \R^{N\times n}\to\R^{N\times n}\qquad\mbox{with}\qquad V_{\mu}(z)\coloneqq (1+|z|^{2})^{-\frac{\mu}{4}}z} \qquad \text{for all } z \in \R^{N\times n}
\end{align}
(see \cite{CARFUSMIN98}) and collect some preliminary observations: First, a direct computation yields that 
\begin{align}\label{eq:Vest}
|\nabla (V_{\mu}(\nabla w))|^{2} \leq (1+|\nabla w|^{2})^{-\frac{\mu}{2}} \sum_{\ell=1}^n|\partial_{\ell}\nabla w|^{2}\qquad\text{for all } w\in\sobo_{\locc}^{2,2}(\Omega;\R^{N}). 
\end{align}
We next record that~$V_\mu$ is a diffeomorphism on $\R^{N\times n}$. Indeed, the continuity is obvious and we can write for all $z\neq0$
\begin{align}\label{eq:homeo1}
V_\mu(z)=f_\mu(|z|)\frac{z}{|z|}\qquad\mbox{with}\qquad f_\mu(s) \coloneqq (1+s^{2})^{-\frac{\mu}{4}}s.
\end{align}
A direct computation yields $f_{\mu}'(s)=(1+s^2)^{-\frac{4+\mu}{4}}(1+s^2(1-\frac{\mu}{2}))>0$ and thus the claim easily follows from $\lim_{s\to\infty}f_{\mu}(s)=\infty$ because of $\mu<2$. Finally, we observe that 
\begin{equation}\label{bound:V}
 (1+|z|^2)^\frac{2-\mu}{4}-1\leq |V_\mu(z)|\leq |z|^\frac{2-\mu}{2}\qquad\text{for all}\; z\in\R^{N\times n}.
\end{equation}
\color{black}
Combining \eqref{eq:Vest} and \eqref{eq:weighted}, we obtain
\begin{align}\label{eq:portugal}
\begin{split}
\int_{\Omega}\eta^{2}|\nabla (V_{\mu}(\nabla u_{j}))|^{2}\dif x & \leq c\,\bigg(\int_{\ball_{R}}(1+|\nabla u_{j}|^{2})^{\frac{q}{2}}|\nabla\eta|^{2}\dif x \bigg. \\ &  \bigg. \qquad + \frac{1}{\CstabG j^{2}} \|\nabla\eta\|_{\lebe^{\infty}(\Omega)}^{2} \int_{\Omega}|\nabla u_{j}|^{2}\dif x + \frac{1}{j}\|\nabla u_{j}\|_{\lebe^{1}(\Omega)} \bigg) 
\end{split}
\end{align}
for any $\eta\in\sobo_{c}^{1,\infty}(\Omega)$ with $0\leq\eta\leq\mathbbm{1}_{\ball_{R}}$, with a constant $c>0$ which is independent of $j\in\mathbb{N}$ and $\eta$. Now, because of $\mu\in[1,2)$, the following exponents are well-defined: 
\begin{equation}\label{eq:choicepQ}
p\coloneqq \frac2{2-\mu},\qquad Q\coloneqq \begin{cases}2\frac{n-1}{n-3}&\mbox{if $n\geq4$},\\Q\in(\frac2{3-\mu-q},\infty)\;\text{arbitrary}&\mbox{if $n=3$},\\\infty&\mbox{if $n=2$.}\end{cases}
\end{equation}
The assumption $\mu\in[1,2)$ ensures that $p\in[2,\infty)$. Next, we observe that $pq< Q$. The case $n=2$ is obvious. If $n=3$, we first note that \eqref{eq:maincondition} gives us $3-\mu-q>0$. Then,  due to $\mu\in[1,2)$, the inequality $pq<Q$ follows from 
\begin{align*}
Q-pq>2\frac{2-\mu-q(3-\mu-q)}{(3-\mu-q)(2-\mu)}=2\frac{(q-1)(\mu+q-2)}{(3-\mu-q)(2-\mu)}\stackrel{\mu\geq 1}{\geq} 2\frac{(q-1)^{2}}{(3-\mu-q)(2-\mu)}\geq 0.
\end{align*}
Lastly, for $n\geq 4$, we first notice $(2-\mu-q)(n-1) > -2$ from \eqref{eq:maincondition}, which then yields
\begin{equation*}
Q-pq= \frac{2(n-1)}{n-3} -\frac{2q}{2-\mu}= 2 \frac{(n-1) (2-\mu-q) + 2q}{(n-3)(2-\mu)} > 2 \frac{-2 + 2q}{(n-3)(2-\mu)} \geq 0.
\end{equation*}
We fix $j\in\mathbb{N}$. We then use $p\leq pq<Q$, \eqref{bound:V} in the form $(1+|z|^2)^\frac{1}{2}\leq (1+|V_\mu(z)|)^p$ and H\"older's inequality to obtain for the first term on the right-hand side of \eqref{eq:portugal}
\begin{multline}\label{est:LqinVinterpolation}
\int_{\ball_{R}}(1+|\nabla u_{j}|^2)^\frac{q}2|\nabla \eta|^2\dx \leq \| (1+|V_{\mu}(\nabla u_{j})|) |\nabla \eta|^\frac2{pq}\|_{\lebe^{pq}(\ball_{R})}^{pq} \\ \leq \big(\| (1+|V_{\mu}(\nabla u_{j})|) |\nabla \eta|^\frac2{pq}\|_{\lebe^{p}(\ball_{R})}^{1-\theta}\| (1+|V_{\mu}(\nabla u_{j})|) |\nabla \eta|^\frac2{pq}\|_{\lebe^{Q}(\ball_{R})}^{\theta}\big)^{pq},
\end{multline}
where
\begin{align}\label{eq:Lyapunov}
\frac{1-\theta}{p}+\frac{\theta}Q=\frac1{pq} \qquad\mbox{and thus}\qquad \theta=\frac{1-\frac1q}{1-\frac{p}Q}.
\end{align}
For future reference, we note that  our assumption \eqref{eq:maincondition}  and the choice \eqref{eq:choicepQ} imply that
\begin{align}\label{eq:thetapqcond1}
\theta q p<2\qquad\text{provided that}\;n\geq 3.
\end{align}
Indeed, for $n\geq4$ we have
\begin{align*}
2-\theta qp=2-\frac{q-1}{\frac1p-\frac1Q}=2-\frac{q-1}{\frac{2-\mu}2-\frac{n-3}{2(n-1)}}=\frac{2-\mu-\frac{n-3}{n-1}-(q-1)}{\frac{2-\mu}2-\frac{n-3}{2(n-1)}}=\frac{2+\frac2{n-1}-q-\mu}{\frac{2-\mu}2-\frac{n-3}{2(n-1)}}\stackrel{\eqref{eq:maincondition}}>0.
\end{align*}
If $n=3$, a similar computation shows that $Q\in(\frac2{3-\mu-q},\infty)$ and $q+\mu<3=2+\frac2{3-1}$ combine to \eqref{eq:thetapqcond1} too. 

\emph{Step 2. Good cut-offs.} Appealing to Lemma~\ref{L:optim}, which fixes the constants $\gamma_{1},\gamma_{2}$ according to \eqref{eq:gammaadjust},  we find for all radii $\frac{R}{2} \leq \rho<\sigma \leq R$ a function $\eta \in \sobo_0^{1,\infty}(\ball_R)$ satisfying \eqref{L:optim:eta} with $\alpha=\frac{2}{pq}$ such that 
\begin{multline}\label{est:optimQ:largen}
\| (1+|V_{\mu}(\nabla u_{j})|) |\nabla \eta|^\frac2{pq}\|_{\lebe^{Q}(\ball_{R})} \stackrel{\eqref{L:optim:claim}}{\leq} c(\sigma-\rho)^{\frac1Q-\frac2{pq}}\Big(R^{1-\gamma_{1}/2}\frac{\|\nabla (V_{\mu}(\nabla u_{j}))\|_{\lebe^2(\ball_\sigma\setminus \ball_\rho)}}{(\sigma-\rho)^\frac12} \Big. \\ \Big. + R^{-\gamma_{2}/p}\frac{\| (1+|V_{\mu}(\nabla u_{j})|)\|_{\lebe^p(\ball_\sigma\setminus \ball_\rho)}}{(\sigma-\rho)^\frac1p}\Big) 
\end{multline}
in the case $n\geq 3$, and 
\begin{multline}\label{est:optimQ:2n}
\| (1+|V_{\mu}(\nabla u_{j})|) |\nabla \eta|^\frac2{pq}\|_{\lebe^{Q}(\ball_R)} \stackrel{\eqref{L:optim:claimn2}}{\leq}  c(\sigma-\rho)^{-\frac2{pq}}\Big(R^{-1/p}\frac{\| (1+|V_{\mu}(\nabla u_{j})|))\|_{\lebe^p(\ball_\sigma\setminus \ball_\rho)}}{(\sigma-\rho)^\frac1p} \Big. \\  \Big.  + \frac{\| (1+|V_{\mu}(\nabla u_{j})|)\|_{\lebe^p(\ball_\sigma\setminus \ball_\rho)}^\frac{p}{p+2}\|\nabla (V_{\mu}(\nabla u_{j}))\|_{\lebe^2(\ball_\sigma\setminus \ball_\rho)}^\frac{2}{p+2}}{(\sigma-\rho)^\frac2{p+2}}\Big)
\end{multline}
in the case $n=2$. In the following, we fix this choice of~$\eta$. We stress that the constant $c>0$ is independent of $j \in \N$ and we proceed to distinguish the cases $n\geq 3$ and $n=2$.  

\emph{Step 3. Proof of assertion~\ref{item:propmain1}.} If $n\geq 3$, we combine \eqref{eq:portugal}, \eqref{est:LqinVinterpolation} and  \eqref{est:optimQ:largen} to find
\begin{align*}
\|\nabla V_{\mu}(\nabla u_{j})\|_{\lebe^2(\ball_\rho)}^2 & \leq c\,R^{(1-\frac{\gamma_{1}}{2})\theta pq}\frac{\|\nabla (V_{\mu}(\nabla u_{j}))\|_{\lebe^2(\ball_\sigma)}^{\theta q p}\|(1+|V_{\mu}(\nabla u_{j})|)\|_{\lebe^p(B_\sigma)}^{(1-\theta) {q}p}}{(\sigma-\rho)^{2+pq\theta(\frac12-\frac1Q)}} \\ & \qquad + c\,R^{-\gamma_{2}q\theta}\frac{\|(1+|V_{\mu}(\nabla u_{j})|)\|_{\lebe^p(\ball_\sigma)}^{{q}p}}{(\sigma-\rho)^{2+pq\theta(\frac{1}{p}-\frac{1}{Q})}} \\
& \qquad + \frac{1}{\CstabG j^{2}(\sigma-\rho)^{2}}\int_{\Omega}|\nabla u_{j}|^{2}\dif x + \frac{1}{j}\|\nabla u_{j}\|_{\lebe^{1}(\Omega)}. 
\end{align*}
Recalling that $\theta pq<2$ from \eqref{eq:thetapqcond1}, we thus obtain by use of  Young's inequality: 
\begin{align*}
\|\nabla (V_{\mu}(\nabla u_{j}))\|_{\lebe^2(\ball_\rho)}^2 & \leq  \frac{1}{2} \|\nabla (V_{\mu}(\nabla u_{j}))\|_{\lebe^2(\ball_\sigma)}^{2}+c\,R^{(1-\frac{\gamma_{1}}2)\frac{2\theta pq}{2-\theta pq}}\frac{\|(1+|V_{\mu}(\nabla u_{j})|)\|_{\lebe^p(\ball_\sigma)}^{\frac{2(1-\theta)qp}{2-\theta pq}}}{(\sigma-\rho)^{(2+pq\theta(\frac12-\frac1Q))\frac{2}{2-\theta pq}}} \\ 
 & \qquad + {c\,R^{-\gamma_{2}q\theta}\frac{\|(1+|V_{\mu}(\nabla u_{j})|))\|_{\lebe^p(\ball_\sigma)}^{{q}p}}{(\sigma-\rho)^{2+pq\theta(\frac{1}{p}-\frac{1}{Q})}}} \\
 & \qquad + \frac{1}{\CstabG j^{2}(\sigma-\rho)^{2}}\int_{\Omega}|\nabla u_{j}|^{2}\dif x + \frac{1}{j}\|\nabla u_{j}\|_{\lebe^{1}(\Omega)}.
\end{align*}
We then apply Lemma~\ref{lem:holefilling} to conclude
\begin{align}
\|\nabla (V_{\mu}(\nabla  u_{j}))\|_{\lebe^2(\ball_{R/2})}^2 &\leq  cR^{n-2}\Big((R^{-\frac{n}{p}}\|(1+|V_{\mu}(\nabla u_{j})|)\|_{\lebe^p(\ball_R)})^{\frac{2(1-\theta) {q}p}{2-\theta pq}} \notag\\ 
& \hspace{2cm} +(R^{-\frac{n}{p}}\|(1+|V_{\mu}(\nabla u_{j})|)\|_{\lebe^p(\ball_R)})^{qp}\Big)\notag\\
& \qquad + c\Big(\frac{1}{\CstabG j^{2}R^{2}}\int_{\Omega}|\nabla u_{j}|^{2}\dif x + \frac{1}{j}\|\nabla u_{j}\|_{\lebe^{1}(\Omega)}\Big). \label{eq:davidbowie} 
\end{align}
Indeed, the correct scaling in $R$ follows by an elementary computation using the definition of $\gamma_1$, $\gamma_2$, see \eqref{eq:gammaadjust}, and \eqref{eq:Lyapunov} in the form
\begin{align*}
\Big(1-\frac{\gamma_1}{2}\Big)\frac{2\theta pq}{2-\theta pq} & -\Big(2+pq\theta\Big(\frac12-\frac1Q\Big)\Big)\frac{2}{2-\theta pq}+\frac{n}{p} \frac{2(1-\theta)pq}{2-\theta pq} \\ & 
\!\!\!\stackrel{\eqref{eq:gammaadjust}}{=}\frac{2}{2-\theta p q}\Big(\theta pq -n\frac{\theta pq}{2}+npq\frac{\theta}{Q}-2+pq n\frac{1-\theta}{p}\Big)\\
& \!\!\!\! \stackrel{\eqref{eq:Lyapunov}}{=} \frac{2}{2-\theta p q}\Big(\theta pq-n\frac{\theta pq}{2}+n-2\Big)=n-2,
\end{align*}
and similarly 
\begin{align*}
&-\Big(\gamma_2 q\theta+2+pq\theta \Big(\frac1p-\frac1Q\Big)\Big)+nq=n-2.
\end{align*}
As an elementary consequence of \eqref{bound:V} and the choice $p=\frac{2}{2-\mu}$, we have 
\begin{equation}\label{est:Vnablaup}
  \|V_{\mu}(\nabla u_j)\|_{\lebe^p(\ball_{R})}^p \leq \|\nabla u_j\|_{\lebe^1(\ball_{R})}.
\end{equation}
At this stage, we employ \eqref{eq:constitution} from Proposition~\ref{prop:summary} to arrive, in combination with \eqref{eq:davidbowie} and \eqref{est:Vnablaup}, at 
\begin{multline}
\limsup_{j\to\infty} \bigg(\dashint_{\ball_{R/2}} |\nabla (V_{\mu}(\nabla u_{j}))|^{2}\dif x\bigg)^{\frac{1}{2}} \label{eq:summary1}\\ \leq  c\, R^{-1}\limsup_{j\to\infty}\bigg(\bigg(\dashint_{\ball_{R}}(1+|\nabla u_j|)\dif x\bigg)^\frac{(1-\theta)q}{2-\theta pq} +\bigg( \dashint_{\ball_{R}} (1+|\nabla u_{j}|)\dif x\bigg)^\frac{q}{2}\bigg) . 
\end{multline}
For $1\leq \mu<1+\frac{2}{n}$, we have $\frac{(2-\mu)n}{(n-2)}>1$. In combination with  Sobolev's inequality (recall that $p\geq 2$) and the lower bound in \eqref{bound:V}, we thus conclude that 
\begin{align}
\limsup_{j\to\infty} &\bigg(\dashint_{\ball_{R/2}}|\nabla u_{j}|^{\frac{2-\mu}{2}\frac{2n}{n-2}}\dif x\bigg)^{\frac{n-2}{2n}}  \stackrel{\eqref{bound:V}}{\leq} 1+\limsup_{j\to\infty}\bigg(\dashint_{\ball_{R/2}}|V_{\mu}(\nabla u_{j})|^{\frac{2n}{n-2}}\dif x\bigg)^{\frac{n-2}{2n}} \notag \\ 
& \leq 1+  c\limsup_{j\to\infty}\,\bigg( \bigg(\dashint_{\ball_{R/2}}|V_\mu(\nabla u_{j})|^{p}\dif x\bigg)^\frac{1}{p} + R\bigg(\dashint_{\ball_{R/2}}|\nabla (V_{\mu}(\nabla u_{j}))|^{2}\dif x \bigg)^{\frac{1}{2}}\bigg) \label{eq:hsv}\\ 
 & \!\!\!\!\!\!\!\!\!\!\!  \stackrel{\eqref{est:Vnablaup},\,\eqref{eq:summary1}}{\leq}c\,\limsup_{j\to\infty} \bigg(1+\bigg(\dashint_{\ball_{R}}|\nabla u_j|\dif x\bigg)^\frac1p+\bigg(\dashint_{\ball_{R}}|\nabla u_j|\dif x\bigg)^{\frac{(1-\theta)q}{2-\theta pq}}  +\bigg(\dashint_{\ball_{R}}|\nabla u_j|\dif x \bigg)^\frac{q}{2}\bigg)\notag\\
 &\!\!\!\!\stackrel{\eqref{eq:unifboundL1}}{<}\infty.\notag 
\end{align}
Next, define a convex function $\Psi_{\mu}\colon\R^{N\times n}\to\R$ by $\Psi_{\mu}(z)\coloneqq  |z|^{\frac{(2-\mu)n}{n-2}}$, so that its recession function is given by 
\begin{align}\label{eq:reccompute1}
\Psi_{\mu}^{\infty}(z)=\begin{cases} 
0&\;\text{if}\;z=0,\\ 
+\infty &\;\text{if}\;z\neq 0. 
\end{cases}
\end{align}
By Proposition~\ref{prop:summary}\ref{item:EkelandSeq1}, we may assume that $u_{j}\stackrel{*}{\rightharpoonup} u$ in $\bv(\Omega;\R^{N})$ and so $\D u_{j}\stackrel{*}{\rightharpoonup}\D u$ in $\mathrm{RM}_{\mathrm{fin}}(\Omega;\R^{N\times n})$. Hence, by Reshetnyak's lower semicontinuity theorem (see Proposition~\ref{prop:resh}) applied in the way as displayed in Remark~\ref{rem:linearperspective}, we conclude that 
\begin{align}
&\dashint_{\ball_{R/2}}\Psi_{\mu}(\nabla u)\dif x   + \dashint_{\ball_{R/2}}\Psi_{\mu}^{\infty}\Big(\frac{\dif \D^{s}u}{\dif|\D^{s}u|}\Big)\dif|\D^{s}u| \stackrel{\text{Rem.~\ref{rem:linearperspective}}}{=} \dashint_{\ball_{R/2}}\Psi_{\mu}^{\#}\Big(\frac{\dif\,(\mathscr{L}^{n},\D u)}{\dif|(\mathscr{L}^{n},\D u)|}\Big)\dif|(\mathscr{L}^{n},\D u)| \notag \\ & \;\;\;\;\;\;\;\;\stackrel{\text{Prop.~\ref{prop:resh}}}{\leq} \liminf_{j\to\infty}\dashint_{\ball_{R/2}}  \Psi_{\mu}^{\#}\Big(\frac{\dif\,(\mathscr{L}^{n},\D u_{j})}{\dif|(\mathscr{L}^{n},\D u_{j})|}\Big)\dif|(\mathscr{L}^{n},\D u_{j})|  \notag \\ & \;\;\;\;\;\;\;\; \stackrel{\text{Rem.~\ref{rem:linearperspective}}}{=} \liminf_{j\to\infty}\int_{\ball_{R/2}}\Psi_{\mu}(\nabla u_{j})\dif x \label{eq:finalbound1}\\ 
& \;\;\;\;\;\;\;\;\;\;\stackrel{\eqref{eq:hsv}}{\leq}  c\,\limsup_{j\to\infty} \bigg(1+\bigg(\dashint_{\ball_{R}}|\nabla u_j|\dif x\bigg)^\frac1p+\bigg(\dashint_{\ball_{R}}|\nabla u_j|\dif x \bigg)^{\frac{(1-\theta)q}{2-\theta pq}}  +\bigg(\dashint_{\ball_{R}}|\nabla u_j|\dif x \bigg)^\frac{q}{2}\bigg)^\frac{2n}{n-2}.\notag
\end{align}
Since the ultimate expression is finite by \eqref{eq:unifboundL1}, \eqref{eq:reccompute1} implies that $\D^{s}u\equiv 0$ on $\ball_{R/2}$. By arbitrariness of $\ball_{R}(x_{0})\Subset\Omega$, we obtain 
\begin{equation*}
u\in\sobo^{1,\frac{(2-\mu)n}{n-2}}_{\rm loc}(\Omega;\R^{N})\subset \sobo^{1,1}_{\rm loc}(\Omega;\R^{N}). 
\end{equation*}
Note that, alternatively, we can use \eqref{eq:hsv} to extract a weakly convergent subsequence in $\sobo^{1,(2-\mu)n/(n-2)}(\ball_{R/2};\R^N)$, identify the limit with the fixed relaxed minimizer~$u$ and then take advantage of the lower semicontinuity of the norm. By definition of $\Psi_{\mu}$, \eqref{eq:finalbound1} yields 
\begin{align}\label{eq:finalboundA1:precise}
 \begin{split}\bigg(\dashint_{\ball_{R/2}}|\nabla u|^{\frac{(2-\mu)n}{n-2}}\dif x\bigg)^\frac{n-2}{(2-\mu)n} \leq c\,\limsup_{j\to\infty}\bigg(1+\sum_{i=0}^2\bigg(\dashint_{{\ball_{R}}}|\nabla u_{j}|\dif x \bigg)^{\mathtt{d}_{i}}\bigg),
 \end{split}
\end{align}
where the exponents $\mathtt{d}_{0}, \mathtt{d}_{1},\mathtt{d}_{2}$ are given by 
\begin{equation}\label{def:d1d2}
 \mathtt{d}_0=1, \qquad \mathtt{d}_1=\frac{(1-\theta)pq}{2-\theta pq}\qquad\mbox{and}\qquad \mathtt{d}_2=\frac{pq}{2}\stackrel{\eqref{eq:choicepQ}}{=}\frac{q}{2-\mu}.
 \end{equation}
 Note that $p\geq2$, $q\geq1$ and $\theta pq<2$ imply $\mathtt{d}_1,\mathtt{d}_2\geq1$. Hence, \eqref{eq:finalboundA1:precise} implies the claimed estimate from~\ref{item:propmain1} with the choice $\mathtt{d}=\max\{\mathtt{d}_{0},\mathtt{d}_{1},\mathtt{d}_{2}\}\geq1$.

\emph{Step 4. Proof of assertion~\ref{item:propmain2}.} For $n=2$ we argue by analogous means as in the case $n\geq 3$, now however combining \eqref{eq:portugal}, \eqref{est:LqinVinterpolation}  and \eqref{est:optimQ:2n}. In the resulting inequality
\begin{align*}
\|\nabla V_{\mu}(\nabla u_{j})\|_{\lebe^2(\ball_\rho)}^2 & \leq c\frac{\|\nabla (V_{\mu}(\nabla u_{j}))\|_{\lebe^2(\ball_\sigma)}^{\frac{2\theta q p}{p+2}}\|(1+|V_{\mu}(\nabla u_{j})|)\|_{\lebe^p(B_\sigma)}^{\frac{qp(p+2-2\theta)}{p+2}}}{(\sigma-\rho)^{2+\frac{2\theta qp}{p+2}}} \\ & \qquad + c\,R^{-\theta q}\frac{\|(1+|V_{\mu}(\nabla u_{j})|)\|_{\lebe^p(\ball_\sigma)}^{{q}p}}{(\sigma-\rho)^{2+q\theta}} \\
& \qquad + \frac{1}{\CstabG j^{2}(\sigma-\rho)^{2}}\int_{\Omega}|\nabla u_{j}|^{2}\dif x + \frac{1}{j}\|\nabla u_{j}\|_{\lebe^{1}(\Omega)}
\end{align*}
the second derivative quantities may be absorbed into the left-hand side provided that $\frac{\theta pq}{p+2}<1$. To see that the latter inequality is true, note that $Q=\infty$ by \eqref{eq:choicepQ} for $n=2$, whereby \eqref{eq:Lyapunov} gives us $q-1=\theta q$, which in turn yields 
\begin{align*}
1-\frac{\theta pq}{p+2} = 1-\frac{q-1}{1+\frac{2}{p}} \stackrel{\eqref{eq:choicepQ}_{1}}{=} 1 - \frac{q-1}{3-\mu} = \frac{4-\mu-q}{3-\mu}>0, 
\end{align*}
since $\mu+q<4$ in the present situation. We may then follow the steps leading to \eqref{eq:summary1}. In this way, we first obtain (using $q-1=\theta q$ and Young inequality with exponents $\frac{p+2}{\theta q p}$ and $\frac{p+2}{p+2-\theta p q}=\frac{p+2}{2+2p-qp}$)
\begin{align*}
\|\nabla (V_{\mu}(\nabla u_{j}))\|_{\lebe^2(\ball_\rho)}^2 & \leq  \frac{1}{2} \|\nabla (V_{\mu}(\nabla u_{j}))\|_{\lebe^2(\ball_\sigma)}^{2}+c\,\frac{\|(1+|V_{\mu}(\nabla u_{j})|)\|_{\lebe^p(\ball_\sigma)}^{\frac{p(2+pq)}{2+2p-pq}}}{(\sigma-\rho)^{\frac{2(2+pq)}{2+2p-pq}}} \\ 
& \qquad + 
{c\,R^{1-q}\frac{\|(1+|V_{\mu}(\nabla u_{j})|))\|_{\lebe^p(\ball_\sigma)}^{{q}p}}{(\sigma-\rho)^{1+q}}} \\
& \qquad + \frac{1}{\CstabG j^{2}(\sigma-\rho)^{2}}\int_{\Omega}|\nabla u_{j}|^{2}\dif x + \frac{1}{j}\|\nabla u_{j}\|_{\lebe^{1}(\Omega)}
\end{align*}
and then via Lemma~\ref{lem:holefilling},~\eqref{eq:constitution} and~\eqref{est:Vnablaup}
\begin{align}\label{eq:summary1n2}
\begin{split}
&\limsup_{j\to\infty} \bigg(\dashint_{\ball_{R/2}} |\nabla V_{\mu}(\nabla u_{j})|^{2}\dif x\bigg)^{\frac{1}{2}} \\
& \stackrel{\eqref{eq:constitution}}{\leq} c\, R^{-1}\limsup_{j\to\infty}\bigg(\bigg(\dashint_{\ball_{R}} (1+|\nabla u_j|)\dif x\bigg)^\frac{2+pq}{2(2+2p-pq)}+\bigg( \dashint_{\ball_{R}} (1+|\nabla u_{j}|)\dif x\bigg)^\frac{q}{2}\bigg) . 
\end{split}
\end{align}
Recalling that $n=2$ and $\mu<2$ by assumption,  we may now argue as for the case $n\geq3$, replacing \eqref{eq:summary1} with \eqref{eq:summary1n2} and using the Moser--Trudinger inequality from~\cite[Thm.~2]{TRUDINGER67} in the form
\begin{align}\label{eq:mosertrudinger}
\begin{split}
    {\;\;-}{\!\!\!\|\,|\nabla u_j|^\frac{2-\mu}{2}\|}_{\exp\lebe^2(\ball_{R/2})}& \stackrel{\eqref{bound:V}}{\leq}    {\;\;-}{\!\!\!\| (1+|V_{\mu}(\nabla u_{j})|)\|}_{\exp\lebe^2(\ball_{R/2})} \\ 
    & \!\!\!\!\!\!\!\!\!\!\!\!\!\!\!\!\!\!\!\!\!\!\!\!\!\!\!\!\!\!\!\!\leq c\bigg(\dashint_{\ball_{R/2}}(1+|V_{\mu}(\nabla u_j)|)^{p}\dif x \bigg)^{\frac{1}{p}} + cR\,\bigg(\dashint_{\ball_{R/2}}|\nabla (V_{\mu}(\nabla u_j))|^{2}\dif x \bigg)^{\frac{1}{2}}, 
    \end{split}
\end{align}
where $\!\!\!\!{\;\;-}{\!\!\!\|\cdot\|}_{\exp\lebe^2(\ball_{R/2})}$ is the scaled version of Orlicz norm associated with the convex function $t\mapsto\exp(t^{2})-1$, see Section~\ref{sec:notation}. To obtain the estimate in the form as given in Proposition~\ref{prop:semimain}\ref{item:propmain2}, we let $\Phi_{\mu}$ be the convex function introduced in 
Section~\ref{sec:expOrliczclasses}; see also Remark~\ref{rem:convex} below for a discussion. We then fix a particular value of $t_{1}>0$, see \eqref{eq:escobar}, such that Lemma~\ref{lem:exponential} is available with this choice of $t_{1}$. Based on 
\begin{align}
 {\;\;-}{\!\!\!\|\,|\nabla u_j|^\frac{2-\mu}{2}\|}{_{\exp\lebe^2(\ball_{R/2})}^{p}} & \stackrel{p=\frac{2}{2-\mu}}{=} \bigg(\inf\Big\{\lambda^{2} >0\colon\; \dashint_{\ball_{R/2}} \Big(\exp\Big(\frac{|\nabla u_{j}|^{2-\mu}}{\lambda^{2}} \Big) -1 \Big) \dif x \leq 1\Big\}\bigg)^{\frac{1}{2-\mu}} \notag\\ 
 & \;\;\; = \bigg(\inf\Big\{\zeta >0\colon\; \dashint_{\ball_{R/2}} \Big( \exp\Big(\frac{|\nabla u_{j}|^{2-\mu}}{\zeta} \Big) -1 \Big) \dif x \leq 1\Big\}\bigg)^{\frac{1}{2-\mu}} \label{eq:rasenmaeher} \\ & \;\;\; = {\;\;-}{\!\!\!\|\,|\nabla u_j |^{2-\mu}\|}{_{\exp\lebe^{1}(\ball_{R/2})}^{\frac{1}{2-\mu}}},\notag
\end{align}
a combination of \eqref{eq:summary1n2}--\eqref{eq:rasenmaeher} with Lemma~\ref{lem:exponential} gives us 
\begin{align}\label{eq:mrbeneett}
\limsup_{j\to\infty}\!\!\!{\;\;-}{\!\!\!\|\nabla u_{j}\|}_{\lebe^{\Phi_{\mu}}(\ball_{R/2})}  & \leq c\limsup_{j\to\infty}\bigg(1 + \sum_{i=0}^{2}\bigg(\dashint_{\ball_{R}}|\nabla u_{j}|\dif x\bigg)^{\mathtt{d}_{i}} \bigg).
\end{align}
Here, the limit on the right-hand side  of \eqref{eq:mrbeneett} is finite by \eqref{eq:unifboundL1}, and the exponents $\mathtt{d}_{0},\mathtt{d}_{1},\mathtt{d}_{2}\geq 1$ are now given by
\begin{align}\label{def:d1d2n2}
\mathtt{d}_{0}=1, \qquad \mathtt{d}_{1}= \frac{p(2+pq)}{2(2+2p-pq)}\;\;\;\text{and}\;\;\;\mathtt{d}_{2}= \frac{pq}{2}\stackrel{\eqref{eq:choicepQ}}{=}\frac{q}{2-\mu}. 
\end{align}
To conclude the proof, we address  the transfer of \eqref{eq:mrbeneett} to $u$ in the requisite form. To this end, we define $\Psi_{\mu}(z)\coloneqq \Phi_{\mu}(|z|)$ with $\Phi_{\mu}$ from~\eqref{eq:escobar}. Based on \eqref{eq:mrbeneett} and by passing to a non-relabelled subsequence if necessary, we may assume that $\liminf_{j\to\infty}\!\!\!\!\!{\;\;-}{\!\!\!\|\nabla u_j\|}{_{\lebe^{\Phi_{\mu}}(\ball_{R/2})}}(<\infty)$ is a limit, and we set $l\coloneqq \lim_{j\to\infty}\!\!\!\!\!{\;\;-}{\!\!\!\|\nabla u_j\|}{_{\lebe^{\Phi_{\mu}}(\ball_{R/2})}}$ for brevity. Let $\varepsilon>0$ be arbitrary. We then find $j_{0}\in\mathbb{N}$ such that $j\geq j_{0}$ implies $\!\!\!\!{\;\;-}{\!\!\!\|\nabla u_j\|}{_{\lebe^{\Phi_{\mu}}(\ball_{R/2})}}<l+\varepsilon$. For such $j$, the definition of the scaled Luxemburg norm gives us
\begin{align}\label{eq:frankenbrunnen}
\dashint_{\ball_{R/2}}{\Psi}_{\mu}\Big(\frac{\nabla u_{j}}{l+\varepsilon}\Big)\dif x\leq 1. 
\end{align}
Clearly, $\frac{1}{l+\varepsilon}\D u_{j} \stackrel{*}{\rightharpoonup}\frac{1}{l+\varepsilon}\D u$ in $\mathrm{RM}_{\mathrm{fin}}(\Omega;\R^{N\times n})$ as a consequence of Proposition~\ref{prop:summary}\ref{item:EkelandSeq1}. Employing Reshetnyak's lower semicontinuity theorem as in \eqref{eq:finalbound1} from above, we find that 
\begin{align}\label{eq:lotus}
\begin{split}
\dashint_{\ball_{R/2}}{\Psi}_{\mu}\Big(\frac{\nabla u}{l+\varepsilon}\Big)\dif x & + \frac{1}{l+\varepsilon}\dashint_{\ball_{R/2}}{\Psi}_{\mu}^{\infty}\Big(\frac{\dif\D^{s}u}{\dif|\D^{s}u|}\Big)\dif|\D^{s}u| \\ & \leq \liminf_{j\to\infty}\dashint_{\ball_{R/2}}{\Psi}_{\mu}\Big(\frac{\nabla u_{j}}{l+\varepsilon}\Big) \dif x \stackrel{\eqref{eq:frankenbrunnen}}{\leq} 1. 
\end{split}
\end{align}
Since ${\Psi}_{\mu}^{\infty}$ equally satisfies \eqref{eq:reccompute1}, we firstly deduce from \eqref{eq:lotus} that $\D^{s}u\equiv 0$ on $\ball_{R/2}$ and that
\begin{align*}
{\;\;-}{\!\!\!\|\nabla u\|}{_{\lebe^{\Phi_{\mu}}(\ball_{R/2})}} \leq l+\varepsilon. 
\end{align*}
Sending $\varepsilon\searrow 0$, the definition of $l$, \eqref{eq:mrbeneett} and the second inequality from \eqref{eq:tenpenny} yield~\ref{item:propmain2} with the choice $\mathtt{d}=\max\{\mathtt{d}_0,\mathtt{d}_1,\mathtt{d}_2\} \geq 1$. 
\end{proof}

\begin{rem}[On the exponents $\mathtt{d}_{1},\mathtt{d}_{2}$]
A direct computation of $\mathtt{d}_1$ defined in \eqref{def:d1d2} and \eqref{def:d1d2n2} yields
\begin{equation*}
\mathtt{d}_1=\frac{1}{2-\mu}\frac{(2-\mu)(n-1)-q(n-3)}{(n-1)(3-\mu-q)-(n-3)}\qquad\mbox{for $n\geq4$ and $n=2$,}
\end{equation*}
with a small correction in dimension $n=3$. In particular, we have that $q>1$ implies that $\mathtt{d}_1>\frac{1}{2-\mu}=\mathtt{d}_2$. In the case $q=1$, we have for all $n\geq2$ that $\mathtt{d}_1=\frac{1}{2-\mu}=\mathtt{d}_2$ (for $n\geq3$ this follows since $\theta=0$ in that case, see \eqref{eq:Lyapunov}), and the scaling of the right-hand side in \eqref{eq:finalboundA1:precise} coincides with previous findings in the literature for functionals with linear growth; see, e.g., \ \cite[Thm 1.1]{GMEINEDER20}.
\end{rem}
\begin{rem}[On the passage to $\Phi_{\mu}$ in \eqref{eq:mosertrudinger}\emph{ff}.]\label{rem:convex}
The passage to the \emph{convex} function $\Phi_{\mu}$ in \eqref{eq:mosertrudinger}\emph{ff}. is primarily motivated by two points. Firstly, if we already knew at \eqref{eq:mosertrudinger} that $\D^{s}u\equiv 0$ in $\ball_{R/2}$ \emph{and} $\nabla u_{j}\to\nabla u$ pointwise $\mathscr{L}^{n}$-a.e. in $\ball_{R/2}$, then the equivalence from \eqref{eq:rasenmaeher} would allow us to conclude that $|\nabla u_{j}|^{2-\mu}\to|\nabla u|^{2-\mu}$ pointwise $\mathscr{L}^{n}$-a.e. in $\ball_{R/2}$, whereby the lower bound in the inequality from~\ref{item:propmain2} could be obtained by use of Fatou's lemma. The requisite pointwise convergence $\mathscr{L}^{n}$-a.e., however, shall only be a consequence of the results from  Section~\ref{sec:higher} below. Hence, to keep the linear structure of the proof and in order to avoid forward references, we prefer to give a self-contained argument. In view of \eqref{eq:mosertrudinger} or \eqref{eq:rasenmaeher}, it is difficult to extract any useful sort of convergence on the sublinear powers $|\nabla u_{j}|^{2-\mu}$ which directly allows us to conclude lower semicontinuity of the corresponding Orlicz norms. Therefore, we are bound to argue via the weak*-convergence of Radon measures. Since the latter is compatible with convex functions by Reshetnyak's theorem, but $t\mapsto \exp(t^{2-\mu})-1$ is \emph{not} convex, this motivates the detour via the convex function $\Phi_{\mu}$ from \eqref{eq:escobar}. 
\end{rem}

The remaining case $\mu=n=2$ from Proposition~\ref{prop:semimain}\ref{item:propmain3}  cannot be accomplished by fully analogous means. Here, a slightly different $V$-function is required, and the use of the standard Sobolev inequality has to be replaced by a variant of the usual Moser--Trudinger inequality:
\begin{proof}[Proof of Proposition~\ref{prop:semimain}\ref{item:propmain3}]
Let $\ball_{R}=\ball_{R}(x_{0})$ be an open ball and denote by $c=c(\gamma,\Gamma,q)>0$ a constant independent of $j$ which might change from place to place. Instead of the function $V_\mu$ defined in~\eqref{eq:Vfunction}, we now work with 
\begin{align}\label{eq:V2}
V(z)\coloneqq  \log((1+|z|^{2})^{\frac{1}{2}})\qquad \text{for all } z\in\R^{N\times 2}. 
\end{align}
By similar computations as in the proof of Theorem~\ref{thm:main}\ref{item:main2}, we now derive local $\lebe^2$-bounds on $\nabla V(\nabla u_j)$. Indeed, a direct computation yields that 
\begin{align}\label{eq:Vestlog}
 |\nabla (V(\nabla w))|^{2} \leq \frac{|\nabla^{2} w|^{2}}{1+|\nabla w|^{2}}\qquad\text{for all}\;w\in\sobo_{\locc}^{2,2}(\Omega;\R^{N}). 
\end{align}
As a substitute of \eqref{est:LqinVinterpolation}, we estimate for an arbitrary $j\in\mathbb{N}$:
\begin{align}\label{eq:eisenimfeuer}
\begin{split}
\int_{\ball_{R}}(1+|\nabla u_{j}|^{2})^{\frac{q}{2}}|\nabla\eta|^{2}\dif x & = \int_{\ball_{R}}\exp(V(\nabla u_{j}))\exp((q-1)V(\nabla u_{j}))|\nabla\eta|^{2}\dif x \\ 
& \leq \bigg(\int_{\ball_{R}}(1+|\nabla u_{j}|^{2})^{\frac{1}{2}}\dif x\bigg) \|\exp(V(\nabla u_{j}))|\nabla\eta|^{\frac{2}{q-1}}\|_{\lebe^{\infty}(\ball_{R})}^{q-1} 
\end{split}
\end{align}
for any $\eta\in\sobo_{c}^{1,\infty}(\Omega)$ with $0\leq\eta\leq\mathbbm{1}_{\ball_{R}}$. Combining the definition~\eqref{eq:V2} of~$V$ and Lemma~\ref{L:optim}, see \eqref{L:optim:claimn2:exp}, we find for every $\frac{R}2\leq \rho<\sigma\leq R$ a function $\eta\in\sobo_{0}^{1,\infty}(\ball_{R})$ satisfying \eqref{L:optim:eta} together with 
\begin{multline}\label{eq:mrtuttle}
 \|\exp(V(\nabla u_{j}))|\nabla\eta|^{\frac{2}{q-1}}\|_{\lebe^{\infty}(\ball_{R})}^{q-1} \\ 
  \leq c(\sigma-\rho)^{-2} \bigg(\frac{\|(1+|\nabla u_{j}|^{2})^{\frac{1}{2}}\|_{\lebe^{1}(\ball_{R})}}{(\sigma-\rho)R}\bigg)^{q-1}\bigg(1+\frac{\|R\,\nabla V(\nabla u_j)\|_{\lebe^2(\ball_\sigma)}^2}{(\sigma-\rho)R}\bigg)^{q-1}.
\end{multline}
By \eqref{eq:maincondition}, we have $1\leq q<2$. Based on  \eqref{L:optim:eta}, the estimate \eqref{eq:weighted} for $\mu=2$, \eqref{eq:Vestlog} and Lemma~\ref{lem:holefilling} yield 
\begin{align*}
\int_{\ball_{R/2}} |\nabla V(\nabla u_{j})|^{2}\dif x &\leq c\,\bigg(\bigg(\dashint_{\ball_R}(1+|\nabla u_j|^2)^\frac12\dif x\bigg)^\frac{q}{2-q}+\bigg(\dashint_{\ball_R}(1+|\nabla u_j|^2)^\frac12\dif x\bigg)^{q}\bigg)\\
& \qquad + \frac{1}{\CstabG j^{2}R^{2}}\int_{\Omega}|\nabla u_{j}|^{2}\dif x + \frac{1}{j}\|\nabla u_{j}\|_{\lebe^{1}(\Omega)}.
\end{align*}
As above, see \eqref{eq:summary1} and~\eqref{eq:summary1n2}, we obtain by virtue of Proposition~\ref{prop:summary}: 
\begin{align}\label{eq:summary1log}
\begin{split}
\limsup_{j\to\infty} \int_{\ball_{R/2}} |\nabla V(\nabla u_{j})|^{2}\dif x & \leq c \limsup_{j\to\infty}\bigg(1+\bigg(\dashint_{\ball_{R}}|\nabla u_j|\dif x \bigg)^{\frac{q}{2-q}} +\bigg(\dashint_{\ball_{R}}|\nabla u_j|\dif x \bigg)^q\bigg).
\end{split}
\end{align}
Moreover, we record that, by elementary properties of the logarithm, we have for all $j\in\N$: 
\begin{equation}\label{bound:Vnablaulog}
\dashint_{\ball_{R/2}}|V(\nabla u_j)|^2\dif x\leq c\dashint_{\ball_{R/2}} (1+|\nabla u_j|)\dif x .
\end{equation}
Our next goal is to control $\|\nabla u_j\|_{\lebe^t(\ball_{R/2})}$ for every $t \in [1,\infty)$ by means of $\|\nabla V(\nabla u_j)\|_{\lebe^2(\ball_{R/2})}^2$. To this end, we use a variation of the classical Moser--Trudinger inequality, namely, the following scaled version: there exist constants $C_{\mathrm{MT}}>0$ and $c>0$ such that 
\begin{align*}
\dashint_{\ball_{r}}\exp\Big( \frac{C_{\mathrm{MT}}|w-(w)_{\ball_r}|^2}{\|\nabla w\|_{\lebe^2(\ball_{r})}^2}\Big)\dif x \leq c\qquad\mbox{for all non-constant functions $w\in \sobo^{1,2}(\ball_r)$}.
\end{align*}
Using the trivial inequality 
$w^2\leq 2(w-(w)_{\ball_r})^2+2(w)_{\ball_r}^2$, we deduce with $c_{\mathrm{MT}}\coloneqq\frac{1}{2}C_{\mathrm{MT}}$ that 
\begin{equation}\label{est:mosernoconstrain}
\dashint_{\ball_{r}}\exp\Big( \frac{c_{\mathrm{MT}}|w|^2}{1+\|\nabla w\|_{\lebe^2(\ball_{r})}^2}\Big)\dif x \leq c\exp(C_{\mathrm{MT}}(w)_{\ball_{r}}^2)\qquad\mbox{for all $w\in \sobo^{1,2}(\ball_r)$}.
\end{equation} 
Next, we obtain by Young's inequality 
\begin{align*}
\log((1+|\nabla u_j|^2)^\frac{t}2)=t V(\nabla u_j)\leq \frac{t^2(1+\|\nabla V(\nabla u_j)\|_{\lebe^2(\ball_{R/2})}^2)}{4c_{\mathrm{MT}}} + \frac{c_{\mathrm{MT}} V(\nabla u_j)^2}{1+\|\nabla V(\nabla u_j)\|_{\lebe^2(\ball_{R/2})}^2}. 
\end{align*}
Hence, we have for every $t\in[1,\infty)$ and $j\in\N$ that
\begin{align*}
\dashint_{\ball_{R/2}}&(1+|\nabla u_j|^2)^\frac{t}2\dif x =\dashint_{\ball_{R/2}}\exp(\log((1+|\nabla u_j|^2)^\frac{t}2))\dif x\\
&\leq  \exp\bigg(\frac{t^2(1+\|\nabla V(\nabla u_j)\|_{\lebe^2(\ball_{R/2})}^2)}{4c_{\mathrm{MT}}}\bigg)\dashint_{\ball_{R/2}}\exp\Big(\frac{c_{\mathrm{MT}} V(\nabla u_j)^2}{1+\|\nabla V(\nabla u_j)\|_{\lebe^2(\ball_{R/2})}^2}\Big)\dif x\\
&
\!\!\!\!\stackrel{\eqref{est:mosernoconstrain}}{\leq} c\,\exp\bigg(\frac{t^2(1+\|\nabla V(\nabla u_j)\|_{\lebe^2(\ball_{R/2})}^2)}{4c_{\mathrm{MT}}} \bigg)\exp(C_{\mathrm{MT}}(V(\nabla u_j))_{\ball_{R/2}}^2).
\end{align*}
Bounding the right-hand side of the previous inequality by use of \eqref{eq:summary1log} and \eqref{bound:Vnablaulog}, we obtain for a suitable constant $c>0$ independent of $t$ that 
\begin{equation}\label{eq:hsv:nmu2}
\limsup_{j\to\infty}\dashint_{\ball_{R/2}}(1+|\nabla u_j|^2)^\frac{t}2\dif x
 \leq c\, \limsup_{j\to\infty}\bigg(\exp\bigg(t^2c\bigg(1+ 
 \sum_{i=0}^{2}\bigg(\dashint_{\ball_{R}}|\nabla u_{j}|\dif x\bigg)^{\mathtt{d}_{i}} \bigg) \bigg)\bigg)
\end{equation}
with $\mathtt{d}_{0}=1$, $\mathtt{d}_{1}=\tfrac{q}{2-q}$ and $\mathtt{d}_{2}=q$, and the ultimate expression is clearly finite by~\eqref{eq:unifboundL1}. 

Clearly, the above estimate in conjunction with an argument analogous to \eqref{eq:finalbound1}\emph{ff}. yields  that $u\in\sobo^{1,t}_{\rm loc}(\Omega;\R^{N})$ for all $t<\infty$; moreover, we have the estimate as displayed in Proposition~\ref{prop:semimain}\ref{item:propmain3}   with an exponent $\mathtt{d}=\max\{\mathtt{d}_{0},\mathtt{d}_{1},\mathtt{d}_{2}\}=\frac{q}{2-q}\geq1$ (recall $n=\mu=2$ and thus $q<2$ in this case by \eqref{eq:maincondition}). This finishes the proof of Proposition~\ref{prop:semimain}\ref{item:propmain3}.
\end{proof}

\begin{rem}[Limitations and comparison with $(p,q)$-growth]\label{rem:unconditional} The reader might notice that the above proof is primarily centered around the wider range \eqref{eq:maincondition} of $q$'s to yield uniform bounds on $\|\nabla V_{\mu}(\nabla u_{j})\|_{\lebe^{2}(\ball_{\rho})}$. This also allows for a wider range of $\mu$ for a given exponent $q$. However, sending $q\searrow 1$, it does \emph{not} improve the range for the classical $\mu$-ellipticity (meaning that $q=1$). Indeed, here the above proof yields uniform bounds of integrals of \emph{superlinear} expressions of $\nabla u_{j}$ only if $1<\mu<1+\frac{2}{n}$. Thus, even though linear growth functionals subject to the classical $\mu$-ellipticity share some resemblance with $(p,q)$-functionals on the level of second derivatives, this yet shows a fundamental difference to $(p,q)$-growth functionals in view of regularity. In particular, aiming for $\mu\geq 1+\frac{2}{n}$ in view of Sobolev regularity seems to require additional tools. 

\end{rem}
\subsection{Strong convergence, higher derivatives and the proof of Theorem~\ref{thm:main}}\label{sec:higher}
We next address further regularity properties of $u$ and the Ekeland-type approximation sequence $(u_j)$; in particular, we prove the strong convergence $u_j \to u$ in $\sobo^{1,1}_{\locc}(\Omega;\R^N)$. We start with a second order regularity result under more restrictive assumptions. To state it conveniently, we put 
\begin{align*}
\bv_{2(,\locc)}(\Omega;\R^{N}) \coloneqq \big\{ u\in\sobo_{(\locc)}^{1,1}(\Omega;\R^{N})\colon\;\nabla u \in\bv_{(\locc)}(\Omega;\R^{N\times n})\big\}. 
\end{align*}

\begin{prop}[Higher derivatives]\label{lem:ae} Let $\Omega\subset\R^{n}$ be open and bounded with Lipschitz boundary. Moreover, let $F\in\hold^{2}(\R^{N\times n})$ be a variational integrand with \eqref{eq:1qgrowth} and \eqref{eq:muell1q}, where $1\leq\mu,q<\infty$ satisfy~\eqref{eq:maincondition} and, in addition, 
\begin{align}\label{eq:mustrongercond}
1\leq\mu\leq\frac{n}{n-1}. 
\end{align}
Then, for any $u_{0}\in\sobo^{1,q}(\Omega;\R^{N})$, the following statements hold for any relaxed minimizer $u\in\bv(\Omega;\R^{N})$ of $\mathscr{F}$ subject to the Dirichlet datum $u_{0}$: 
\begin{enumerate}
\item\label{item:detective1} $u\in \bv_{2,\locc}(\Omega;\R^{N})$ provided that $\mu=\frac{n}{n-1}$ and $n\geq3$,
\item\label{item:detective2} $u\in \sobo_{\locc}^{2,t}(\Omega;\R^{N})$ with $t=\frac{(2-\mu)n}{n-\mu}\in(1,2)$ provided that $\mu<\frac{n}{n-1}$ and $n\geq3$,
\item\label{item:detective3} $u\in \sobo_{\locc}^{2,t}(\Omega;\R^{N})$ for all $t<2$ provided that $n=2$.
\end{enumerate}
Moreover, let $(u_{j})$ be the Ekeland-type vanishing viscosity sequence in $\sobo_{\locc}^{2,2}(\Omega;\R^{N})$ constructed in Section~\ref{sec:Ekelandstart}. Then, for any open set $K\Subset \Omega$, we have for $n=2$ or $n\geq3$ and $\mu<\frac{n}{n-1}$ that
\begin{equation}\label{lim:weaksecond1}
    \nabla u_j\rightharpoonup \nabla u\qquad \mbox{in $\sobo^{1,t}(K;\R^{N\times n})$}
\end{equation}
with $t\in(1,2)$ as above, and, in the case $n\geq3$ with $\mu=\frac{n}{n-1}$,
\begin{equation}\label{lim:weaksecond2}
    \nabla u_j\stackrel{*}{\rightharpoonup} \nabla u\qquad \mbox{in $\BV(K;\R^{N\times n})$}.
\end{equation}
\end{prop}

\begin{proof}
Let $x_{0}\in\Omega$ and $R>0$ be such that $0<R<\frac{1}{2}\mathrm{dist}(x_{0},\partial\Omega)$. Moreover, let $K\Subset\Omega$ be an arbitrary but fixed relatively compact subset with $\ball_{R}\coloneqq \ball_{R}(x_{0})\Subset K$. For the sequence $(u_{j})$ in $\sobo_{\locc}^{2,2}(\Omega;\R^{N})$ constructed in Section~\ref{sec:Ekelandstart}, we have
\begin{align}\label{eq:uniuni}
\|\nabla u\|_{\lebe^{s}(K)}+\sup_{j\in\mathbb{N}}\|\nabla u_{j}\|_{\lebe^{s}(K)}\leq c(K,s)<\infty\qquad\text{for any}\;1\leq s\begin{cases}\leq \frac{(2-\mu)n}{n-2}&\mbox{if $n\geq3$,}\\<\infty&\mbox{if $n=2$,}\end{cases}
\end{align}
as a consequence of Proposition~\ref{prop:semimain} combined with~\eqref{eq:hsv} for $n\geq3$  and~\eqref{eq:hsv:nmu2} for $n=2$, where $\mu=2=\frac{n}{n-1}$ in the latter case. We put 
\begin{align*}
t \coloneqq \begin{cases}
    \frac{(2-\mu)n}{n-\mu} &\;\text{if}\;n\geq 3,\\ 
    \in (1,2)\;\text{arbitrary}&\;\text{if}\;n=2. 
\end{cases}
\end{align*}
For $n\geq3$, an elementary computation shows $1 \leq t<2$ provided that $\mu\leq\frac{n}{n-1}$, with  $t=1$ if and only if $\mu=\frac{n}{n-1}$. Moreover, we have in the case $n\geq3$
\begin{align*}
\frac{\mu t}{2-t}=\frac{\mu(2-\mu)n}{2(n-\mu)-(2-\mu)n}=\frac{(2-\mu)n}{n-2}.
\end{align*}
By Young's inequality with exponents $\frac{2}{t}$ and $\frac{2}{2-t}$, we find for every $\eta\in \hold_c^\infty(\Omega)$ and $\ell\in\{1,\dots,n\}$ that 
\begin{align*}
\int_{\Omega}\eta^2|\partial_{\ell}\nabla u_{j}|^t\dif x \leq \frac{t}{2}\int_{\Omega}\eta^{2}\frac{|\partial_{\ell}\nabla u_{j}|^{2}}{(1+|\nabla u_{j}|^{2})^{\frac{\mu}{2}}}\dif x+\frac{2-t}{2}\int_{\Omega}\eta^{2}(1+|\nabla u_{j}|^{2})^{\frac{\mu t}{2(2-t)}}\dif x. 
\end{align*}
Recalling \eqref{eq:constitution}, we may combine the Caccioppoli inequality \eqref{eq:weighted} and any  $\eta\in\hold_{c}^{\infty}(\Omega)$ with $\mathbbm{1}_{\ball_{R/2}}\leq\eta \leq \mathbbm{1}_{\ball_{R}}$ and $|\nabla\eta|\leq\frac{4}{R}$  to deduce for $\ell\in\{1,\dots,n\}$ that
\begin{align*}
\limsup_{j\to\infty}\int_{\ball_{R/2}}|\partial_{\ell}\nabla u_{j}|^t\dif x \leq c\limsup_{j\to\infty}\Big( R^{-2}\int_{\ball_R}(1+|\nabla u_{j}|^{2})^{\frac{q}{2}}\dif x+\int_{\ball_R}(1+|\nabla u_{j}|^{2})^{\frac{(2-\mu)n}{2(n-2)}}\dif x\Big).
\end{align*}
With the help of \eqref{eq:uniuni} and the inequality 
\begin{equation}\label{ineq:qleq}
q-\frac{(2-\mu)n}{n-2}\stackrel{\eqref{eq:maincondition}}{<}2+\frac{2}{n-1}-\mu-\frac{(2-\mu)n}{n-2}=\frac{2}{n-1}-(2-\mu)\frac{2}{n-2}\stackrel{\mu\leq \frac{n}{n-1}}\leq0,
\end{equation}
we obtain 
\begin{align*}
\limsup_{j\to\infty}\int_{\ball_{R/2}}|\nabla^2 u_{j}|^t\dif x <\infty.
\end{align*}
Therefore, passing to a suitable non-relabelled subsequence, the sequence $(u_{j})$ is uniformly bounded in $\sobo^{2,t}(\ball_{R/2};\R^{N})$. Based on this observation and the arbitrariness of $\ball_{R}\Subset K$, the statements~\ref{item:detective2} and~\ref{item:detective3} follow from standard compactness arguments in reflexive spaces. Concerning~\ref{item:detective1}, we note that the second total variation is lower semicontinuous with respect to $\lebe_{\locc}^{1}$-convergence, whereby $u\in\bv_{2}(\ball_{R/2};\R^{N})$ follows in this case, too. The assertions \eqref{lim:weaksecond1} and \eqref{lim:weaksecond2} follow from elementary compactness arguments, and the proof is complete.
\end{proof}

Next, we establish strong convergence $u_j \to u$ in $\sobo^{1,1}_{\locc}(\Omega;\R^N)$  under the more general assumptions of Theorem~\ref{thm:main}. 

\begin{prop}[Strong convergence of gradients]\label{Lem:stronglimit}
Consider the situation of Theorem~\ref{thm:main}. Let $(u_{j})$ be the Ekeland-type vanishing viscosity sequence in $\sobo_{\locc}^{2,2}(\Omega;\R^{N})$ constructed in Section~\ref{sec:Ekelandstart}. Then, for any open set  $K\Subset \Omega$, there holds 
\begin{align}\label{lim:ae:eq2}
\nabla u_{j}\to \nabla u\qquad\text{in $\lebe^s(K;\R^{N\times n})$ for any } 1\leq s< t^*\coloneqq \begin{cases}\frac{(2-\mu)n}{n-2}&\mbox{if $n\geq3$,}\\{\in(1,\infty)\mbox{ arbitrary}}&\mbox{if $n=2$.}\end{cases}
\end{align}
\end{prop}

\begin{proof}
Let $\ball_{R}=\ball_{R}(x_0)\Subset \Omega$ be an arbitrary ball. We show \eqref{lim:ae:eq2} for $K=\ball_{R/2}$, and the claim follows by the arbitrariness of $\ball_{R}$. Firstly, we consider the case $\mu<2$. In view of \eqref{eq:summary1}\emph{ff}. for $n \geq 3$ and \eqref{eq:summary1n2}\emph{ff}. for $n =2$, we have
$$
\limsup_{j\to\infty}\|V_{\mu}(\nabla u_j)\|_{\sobo^{1,2}(\ball_{R/2})}<\infty.
$$  
By the Rellich--Kondrachov theorem, this ensures that there exists $\overline V\in \sobo^{1,2}(\ball_{R/2};\R^{N\times n})$ and a (non-relabelled) subsequence such that
$$
V_{\mu}(\nabla u_j)\to \overline V\qquad\mathscr{L}^{n}\text{-a.e. in $\ball_{R/2}$}.
$$
Since $V_{\mu}$ is a diffeomorphism, see \eqref{eq:homeo1}, we deduce that
\begin{equation*}
\nabla u_j\to V_{\mu}^{-1}(\overline V)\qquad\mathscr{L}^{n}\text{-a.e. in $\ball_{R/2}$},
\end{equation*}
which in turn implies $\nabla u_{j}\to V_{\mu}^{-1}(\overline{V})$ in $\mathscr{L}^{n}$-measure on $\ball_{R/2}$. Now let $A\subset \ball_{R/2}$ be measurable. Based on \eqref{eq:hsv} and \eqref{eq:mrbeneett}, we estimate for an arbitrary $1\leq s<t^*$:
\begin{align*}
\sup_{j\in\N}\int_{A}|\nabla u_{j}|^{s}\dif x \leq \mathscr{L}^{n}(A)^{\frac{t^*-s}{t^*}}\sup_{j\in\N}\bigg(\int_{\ball_{R/2}}|\nabla u_{j}|^{t^*}\dif x\bigg)^{\frac{s}{t^*}} \leq c\mathscr{L}^{n}(A)^{\frac{t^*-s}{t^*}}. 
\end{align*}
Hence, the sequence $(\nabla u_{j})$ is $s$-equi-integrable in $\ball_{R/2}$, and from here it follows that $(\nabla u_j)$ converges strongly to $V_{\mu}^{-1}(\overline V)$ in $\lebe^{s}(\ball_{R/2};\R^{N\times n})$ by Vitali's convergence theorem. In combination with $\nabla u_{j}\stackrel{*}{\rightharpoonup} \D u = \nabla u\mathscr{L}^{n}$ in $\mathrm{RM}_{\mathrm{fin}}(\ball_{R};\R^{N\times n})$ and the uniqueness of weak*-limits, we have $\nabla u=V_{\mu}^{-1}(\overline V)$. Since these arguments apply to  every subsequence of the initial sequence $(u_j)$, the claim \eqref{lim:ae:eq2} follows. %

In the remaining case $n=\mu=2$,  Proposition~\ref{lem:ae} ensures that $(\nabla u_j)$ is bounded in $\sobo^{1,1}(\ball_{R/2};\R^{N\times n})$. Hence, by the  Rellich--Kondrachov theorem, there exists a subsequence that converges almost everywhere. In combination with \eqref{eq:hsv:nmu2}, the above argument applies also in this case and the claim  \eqref{lim:ae:eq2} follows.  This completes the proof. 
\end{proof}

Based on Proposition~\ref{Lem:stronglimit}, we are now ready to give the: 
\begin{proof}[Proof of Theorem~\ref{thm:main}]
In view of the estimates from Proposition \ref{prop:semimain}\ref{item:propmain1}--\ref{item:propmain3}, the strong convergence $\nabla u_{j}\to\nabla u$ in $\lebe_{\locc}^{1}(\Omega;\R^{N\times n})$ from Proposition~\ref{Lem:stronglimit} allows us to pass to the limit on the corresponding right-hand sides. This yields the assertions of Theorem~\ref{thm:main}.
\end{proof}

\begin{rem}\label{rem:w1qmu}
Note that, in the situation of Theorem~\ref{thm:main}, Proposition~\ref{Lem:stronglimit} ensures the strong convergence $u_j\to u$ in $\sobo_{\locc}^{1,q}(\Omega;\R^{N})$ in $n=2$ and for $n\geq3$ under the additional assumption
\begin{equation}\label{ineq:remw1qmu}
q<\frac{(2-\mu)n}{n-2}.
\end{equation}
In view of  \eqref{ineq:qleq}, condition \eqref{ineq:remw1qmu} is automatically satisfied if $\mu\leq\frac{n}{n-1}$. In the case $q=1$, condition \eqref{ineq:remw1qmu} reduces to the well-known condition $\mu<1+\frac{2}{n}$,  which is also satisfied in the situation of Theorem~\ref{thm:main}.
\end{rem}

\subsection{Euler--Lagrange systems and dimension reduction}\label{sec:ELdimreduc}

A rather direct consequence of Theorem~\ref{thm:main} and Proposition~\ref{lem:ae} is the validity of the Euler--Lagrange system for every relaxed minimizer; a variant thereof, which uses the machinery of Young measures, has been stated in the unpublished note \cite{KOCKRI22} by Koch \& Kristensen. 
 
\begin{corollary}[Local minimality and Euler--Lagrange system]\label{cor:locmin}
Consider the situation of Theorem~\ref{thm:main} where we  additionally  assume  \eqref{ineq:remw1qmu} provided that $n\geq 3$. Then, for any $u_{0}\in\sobo^{1,q}(\Omega;\R^{N})$, every  relaxed minimizer $u\in\bv(\Omega;\R^{N})$ of $\mathscr{F}$ subject to the Dirichlet datum $u_{0}$ is a \emph{local minimizer} of $\mathscr F$ in the sense that 
\begin{equation}\label{eq:ulocalmin}
    \int_{K} F(\nabla u)\dif x\leq \int_{K} F(\nabla u+\nabla \varphi)\dif x\qquad \mbox{for all open $K\Subset\Omega$ and all $\varphi\in \sobo_c^{1,q}(K;\R^{N})$}.
    \end{equation}
Moreover, $u$ satisfies the \emph{Euler--Lagrange system}
    \begin{equation}\label{eq:ulocalminEL}
     \int_\Omega \langle \nabla F(\nabla u),\nabla \varphi\rangle \dif x=0\qquad \text{for all}\; \varphi\in \sobo^{1,q}_c(\Omega;\R^{N}).
    \end{equation}
\end{corollary}

\begin{proof}
We begin with the proof of \eqref{eq:ulocalmin}. From \eqref{eq:uptown} and $F\leq G_j$, we deduce for the sequence $(u_{j})$ in $\sobo_{\locc}^{2,2}(\Omega;\R^{N})$ constructed in Section~\ref{sec:Ekelandstart} that 
 $$
 \overline {\mathscr F}_{u_0}^{*}[u;\Omega]=\lim_{j\to\infty} \mathscr F[u_{j};\Omega].
 $$
 Let $\varphi\in \sobo_c^{1,q}(K;\R^{N})$ be given. By the very definition of the relaxation $\overline{\mathscr{F}}_{u_{0}}^{*}[-;\Omega]$, we have 
 $$
 \overline{\mathscr{F}}_{u_{0}}^{*}[u+\varphi;\Omega]\leq \liminf_{j\to\infty} {\mathscr{F}}[u_j+\varphi;\Omega].
 $$
 Clearly, the growth condition in \eqref{eq:1qgrowth} ensures that $v\mapsto \mathscr{F}[v,K]$ is continuous with respect to strong convergence in $\sobo^{1,q}(K;\R^{N})$.  Hence, Proposition~\ref{lem:ae} together with Remark~\ref{rem:w1qmu} yield 
$$
 \lim_{j\to\infty} \mathscr{F}[u_j;K] = \mathscr{F}[u;K]  \qquad\mbox{and}\qquad\lim_{j\to\infty} \mathscr{F}[u_j+\varphi;K] = \mathscr{F}[u + \varphi ;K]  .
 $$
 Combining the previous three displayed limits with the fact that $\spt(\varphi)\subset K$, we obtain from the fact that $u$ is a relaxed minimizer
 \begin{align*}
 \int_K F(\nabla u)-F(\nabla u+\nabla \varphi)\dif x & =\lim_{j\to\infty} (\mathscr{F}[u_j;\Omega]-\mathscr{F}[u_j+\varphi;\Omega]) \\
 & \leq\overline{\mathscr{F}}_{u_{0}}^{*}[u;\Omega]-\overline{\mathscr{F}}_{u_{0}}^{*}[u+\varphi;\Omega] \leq 0, 
 \end{align*}
 and thus~\eqref{eq:ulocalmin}. Finally, \eqref{eq:ulocalminEL} follows from \eqref{eq:ulocalmin} as in the proof of Proposition \ref{prop:summary}\ref{item:EkelandSeq2}. 
\end{proof}

With the higher differentiability of relaxed minimizers according to Proposition~\ref{prop:semimain}, we can further differentiate the Euler--Lagrange equation \eqref{eq:ulocalminEL}. This will be exploited in the proof of Theorem~\ref{thm:main2} in Section~\ref{sec:proofmain2} below.

\begin{corollary}[Differentiated Euler--Lagrange system]\label{cor:EL2}
Consider the situation of Theorem~\ref{thm:main} where we assume in addition that $\mu<\frac{n}{n-1}$ provided that $n\geq 3$. Then, for any $u_{0}\in\sobo^{1,q}(\Omega;\R^{N})$, every relaxed minimizer $u\in\bv(\Omega;\R^{N})$ of $\mathscr{F}$ subject to the Dirichlet datum $u_{0}$ satisfies the \emph{differentiated Euler--Lagrange equation}
\begin{equation}\label{eq:ulocalminlin}
     \int_\Omega \langle \nabla^2 F(\nabla u)\nabla \partial_\ell u,\nabla \varphi\rangle \dif x=0\qquad\text{for all}\;\varphi\in \hold_c^{\infty}(\Omega;\R^{N})\;\text{and all}\;\ell\in\{1,\ldots ,n\}.
\end{equation}
\end{corollary}
 
 \begin{proof}
 Let $(u_{j})$ be the Ekeland-type approximation sequence in $\sobo_{\locc}^{2,2}(\Omega;\R^{N})$ constructed in Section~\ref{sec:Ekelandstart}. We deduce from \eqref{eq:hbftab1} that 
 \begin{equation}\label{eq:ulocalminlinprep}
 \lim_{j\to\infty}\int_\Omega \langle \nabla^2 F(\nabla u_j)\partial_\ell \nabla u_j,\nabla \varphi\rangle\dif x=0\qquad \mbox{for all } \varphi\in \hold_c^\infty(\Omega;\R^N).
 \end{equation}
 Here, we have used that, as a consequence of \eqref{eq:constitution} and after an application of the integration by parts formula, the two terms which stem from the stabilization vanish in the limit.
 In order to pass to the limit in the integral on the left-hand side, we first observe that $\nabla^2 F(\nabla u_j)\to  \nabla^2 F(\nabla u)$ in $\lebe^s(K;\R^{N\times n})$ for all $s<\infty$ and $K\Subset \Omega$. Indeed, \eqref{eq:muell1q} ensures $|\nabla^2 F(z)|\leq c(1+|z|)^{q-2}$. In the case $q\leq2$, the claim follows directly from \eqref{lim:ae:eq2}. Moreover, we observe that $q>2$ and assumption \eqref{eq:maincondition} imply that $n=2$, and thus the desired convergences follows again from \eqref{lim:ae:eq2}. Clearly, the strong convergence of $F(\nabla u_j)$ in combination with the weak convergence~\eqref{lim:weaksecond1} implies that \eqref{eq:ulocalminlin} follows from \eqref{eq:ulocalminlinprep}.
 \end{proof}

Lastly, we record a dimension reduction result for the singular set of relaxed minimizers. To this end, we denote the \emph{regular set} of a minimizer $u\in\bv(\Omega;\R^{N})$ of $\overline{\mathscr{F}}_{u_{0}}^{*}[-;\Omega]$ by 
\begin{align*}
\mathrm{Reg}(u) \coloneqq  \{x_{0}\in\Omega \colon u \text{ is of class $\hold^{1,\alpha}$ in an open neighbourhood of $x_{0}$ for all $0<\alpha<1$}\}, 
\end{align*}
and define the \emph{singular set} by $\Sigma_{u}\coloneqq \Omega\setminus\mathrm{Reg}(u)$. We then have:
\begin{corollary}[Dimension reduction]\label{cor:dimred}
Let $F\in\hold^{2}(\R^{N\times n})$ satisfy \eqref{eq:1qgrowth} and \eqref{eq:muell1q} with 
\begin{align}\label{eq:muqbd}
1<\mu < \frac{n}{n-1}\;\;\;\text{and}\;\;\;1\leq q < \frac{n}{n-1}.
\end{align}
Then, for any $u_{0}\in\sobo^{1,q}(\Omega;\R^{N})$, the singular set of every  relaxed minimizer $u\in\bv(\Omega;\R^{N})$ of $\mathscr{F}$ subject to the Dirichlet datum $u_{0}$ satisfies the \emph{Hausdorff dimension bound}
\begin{align*}
\dim_{\mathscr{H}}(\Sigma_{u})\leq n-1. 
\end{align*}
\end{corollary}

\begin{proof}
We set 
\begin{align}\label{eq:Vpartial}
V(z)\coloneqq  \int_{0}^{|z|}\int_{0}^{t}\frac{\dif s \dif t}{(1+s^{2})^{\frac{\mu}{2}}},\qquad \text{for all } z\in\R^{N\times n}, 
\end{align}
whereby $V$ is of class $\hold^{2}$, strictly convex and satisfies, for some constants $c_{1},...,c_{4}>0$
\begin{align}\label{eq:Vprops}
c_{1}|z|^{2}\leq V(z) \leq c_{2}|z|^{2}\;\;\;\text{if}\;|z|\leq 1\;\;\;\text{and}\;\;\;c_{3}|z|\leq V(z) \leq c_{4}(1+|z|). 
\end{align}
In the present situation, \eqref{eq:muell1q} implies that there exists a constant $\ell>0$ such that $F-\ell V$ is convex (and hence, in particular, quasiconvex). Based on \eqref{eq:muqbd}, $F$ thus satisfies all assumptions of \cite[Thm. 2.1]{GMEKRI24} up to two points: In \cite[Thm. 2.1]{GMEKRI24}, it is required that $F-\ell\langle\cdot\rangle_{1}$ is quasiconvex and that $F\in\hold^{\infty}(\R^{N\times n})$. From the perspective of partial regularity and because of \eqref{eq:Vprops}, it is immaterial if we work with $\langle\cdot\rangle_{1}$ instead of $V$ given by \eqref{eq:Vpartial}. Secondly, following the discussion after \cite[Thm. 2.1]{GMEKRI24}, we obtain the $\hold^{1,\alpha}$-partial regularity of relaxed minimizers if $F$ is only of class $\hold^{2}$. Moreover, because of \eqref{eq:muqbd}, Theorem~\ref{thm:main} implies that $\D u^{s}\equiv 0$, and so \cite[Thm. 2.1]{GMEKRI24} gives us the following characterisation of $\Sigma_{u}$: 
\begin{align}\label{eq:singsetchar}
\begin{split}
\Sigma_{u} = \Sigma_{u}^{1}\cup\Sigma_{u}^{2} & \coloneqq  \Big\{x_{0}\in\Omega\colon\;\liminf_{r\searrow 0}\dashint_{\ball_{r}(x_{0})}|\nabla u - (\nabla u)_{\ball_{r}(x_{0})}|\dif x >0\Big\} \\ 
& \qquad \cup \Big\{x_{0}\in\Omega\colon\;\limsup_{r\searrow 0}|(\nabla u)_{\ball_{r}(x_{0})}| = + \infty\Big\}. 
\end{split}
\end{align}
In view of Proposition~\ref{lem:ae}, we have that $u\in\bv_{2,\locc}(\Omega;\R^{N})$. The usual Poincar\'{e} inequality on $\bv$ then implies for $\ball_{r}(x_0)\Subset\Omega$ that 
\begin{align*}
\dashint_{\ball_{r}(x_{0})}|\nabla u-(\nabla u)_{\ball_{r}(x_{0})}|\dif x \leq c\,r\,\frac{|\D^{2}u|(\ball_{r}(x_{0}))}{r^{n}}, 
\end{align*}
and so 
Giusti's measure density lemma (see e.g. \cite[Prop. 1.76]{BECK16}, which directly extends to functions in $\bv(\Omega;\R^{N\times n})$), yields that $\dim_{\mathscr{H}}(\Sigma_{u})\leq n-1$. This completes the proof. 
\end{proof}

\begin{rem}
The condition \eqref{eq:muqbd} is required both for the application of \cite[Thm. 2.1]{GMEKRI24} \emph{and} the higher Sobolev regularity from Proposition~\ref{lem:ae}. If $q=1$, then partial regularity in itself holds without any restriction on $\mu$ (see, e.g., \cite{ANZGIA88,GMEKRI19b,GMEINEDER21}); however, in view of Remark~\ref{rem:unconditional}, it is then required for higher Sobolev regularity, leading to the estimate for the Hausdorff dimension of the singular set. For partial regularity in the setting of Corollary~\ref{cor:dimred}, the exponent $\mu<\frac{n}{n-1}$ is primarily required to reduce to the more general quasiconvex setting as assumed in \cite{GMEKRI24}. We expect that, for partial regularity for \emph{convex}  problems with $(1,q)$-growth, a substantially larger range of $\mu$ and $q$ will work. 
\end{rem}
\section{Proof of Theorem~\ref{thm:main2}:  $\hold^{1,\alpha}$-regularity in two dimensions}\label{sec:proofmain2}

We now come to the proof of Theorem~\ref{thm:main2} and base our arguments on the notation and results of Section~\ref{sec:proofmain}; recall that now $n=2$. In what follows, let $u\in\bv(\Omega;\R^{N})$ be an arbitrary but fixed  relaxed minimizer, and denote by $(u_{j})$ the Ekeland-type vanishing viscosity sequence from Section~\ref{sec:Ekelandstart}. Moreover, we recall from \eqref{eq:2dstrongercond} that $1\leq \mu <2$ and $q\geq 1$ are now such that  
\begin{align}\label{eq:qmupropaganda}
\max\{2,q\}+3\mu<6, 
\end{align}
and that $F\in\hold^{2}(\R^{N\times n})$ satisfies \eqref{eq:muell1q} with these choices of $\mu$ and $q$. As an elementary conclusion of \eqref{eq:qmupropaganda}, we record that
\begin{align*}
q<6-3\mu \stackrel{\mu\geq 1}{\leq} 4-\mu \qquad \Longrightarrow \qquad q+\mu<4=2+\frac{2}{n-1}\qquad\mbox{and}\qquad \mu<2=\frac{n}{n-1}.
\end{align*}
Hence, both Theorem~\ref{thm:main}\ref{item:main2b} and Proposition~\ref{lem:ae}\ref{item:detective3} apply in this setting. The proof of Theorem~\ref{thm:main2} is strongly inspired by Bildhauer \& Fuchs \cite{BILFUC03,BILDHAUER03}, where similar results have been established in the context of $(p,q)$-growth conditions: 
\begin{proof}[Proof of Theorem~\ref{thm:main2}]
Let $x_{0}\in\Omega$,  $0<R<\frac{1}{2}\mathrm{dist}(x_{0},\partial\Omega)$ and $\xi_1,\xi_2 \in\R^{N}$. Moreover, let $\ell\in\{1,2\}$. We pick a smooth cut-off function $\eta\in\hold_{c}^{\infty}(\Omega;[0,1])$ such that $\mathbbm{1}_{\ball_{R}(x_{0})}\leq\eta\leq\mathbbm{1}_{\ball_{2R}(x_{0})}$ and $|\nabla\eta|\leq \frac{4}{R}$. The starting point for the present proof is the Caccioppoli-type inequality \eqref{eq:weighted0} from Proposition~\ref{prop:cacc}\ref{item:Cacc1}. To simplify the following computations, we note that Proposition~\ref{prop:summary}\ref{item:EkelandSeq0} and~\ref{item:EkelandSeq1} imply that there exists a sequence $(\delta_{j})$ in $(0,\infty)$ with $\delta_{j}\searrow 0$  and independent of $\xi_\ell$ such that 
\begin{align}\label{eq:klasse}
\begin{split}
\int_{\ball_{R}(x_{0})} & \langle \nabla^2F(\nabla  u_{j})\partial_\ell \nabla u_{j},\partial_\ell \nabla u_{j}\rangle \,\mathrm dx \\ & \leq 2\bigg\vert\int_{\mathcal{A}_{R}(x_{0})} \langle \nabla^{2}F(\nabla u_{j})\partial_\ell \nabla u_{j}, (\partial_\ell u_{j}-\xi_{\ell})\otimes\nabla\eta\rangle\mathrm dx\bigg\vert  
+ \delta_{j}\,(1+|\xi_{\ell}|^{2}),
\end{split}
\end{align}
where $\mathcal{A}_{R}(x_{0})\coloneqq \ball_{2R}(x_{0})\setminus\overline{\ball}_{R}(x_{0})$ as usual. 
For brevity, we put 
\begin{align*}
H_{j,\ell}\coloneqq \sqrt{\langle \nabla^2F(\nabla  u_{j})\partial_\ell \nabla u_{j},\partial_\ell \nabla u_{j}\rangle}.
\end{align*}
Applying the Cauchy--Schwarz inequality pointwisely to the integrands on the right-hand side of \eqref{eq:klasse}, we hence obtain by a subsequent use of H\"{o}lder's inequality and the estimate $|\nabla \eta|\leq \frac{4}{R}$: 
\begin{align}\label{eq:bruderduell}
\begin{split}
& \int_{\ball_{R}(x_{0})} |H_{j,\ell}|^2 \,\mathrm dx \\ & \leq 2\bigg\vert\int_{\mathcal{A}_{R}(x_{0})} H_{j,\ell} \big(\langle \nabla^{2}F(\nabla u_{j})((\partial_{\ell}u_{j}-\xi_\ell)\otimes\nabla\eta), (\partial_\ell u_{j}-\xi_\ell)\otimes\nabla\eta\rangle \big)^{\frac{1}{2}}\mathrm dx\bigg\vert \\ 
& \qquad + \delta_{j}\,(1+|\xi|^{2}) \\ 
& \leq \frac{8}{R}\bigg(\int_{\mathcal{A}_{R}(x_{0})}|H_{j,\ell}|^{2}\dif x\bigg)^{\frac{1}{2}} \, \mathrm{I} + \delta_{j}(1+|\xi_\ell|^{2})
\end{split}
\end{align}
with 
\begin{equation*}
\mathrm{I} \coloneqq \bigg(\int_{\mathcal{A}_{R}(x_{0})\cap\{|\nabla\eta|\neq 0\}}\langle\nabla^{2}F(\nabla u_{j})((\partial_{\ell}u_{j}-\xi_\ell)\otimes\tfrac{\nabla\eta}{|\nabla \eta|}),((\partial_{\ell}u_{j}-\xi_\ell)\otimes\tfrac{\nabla\eta}{|\nabla \eta|})\rangle\dif x\bigg)^{\frac{1}{2}}.
\end{equation*}
We now establish the continuity of $\nabla u$; to this end, we distinguish the cases $1\leq q\leq 2$ and $q>2$. 

\emph{Case $1\leq q\leq 2$.} We recall the auxiliary function $V_{\mu}$ defined by \eqref{eq:Vfunction}. Employing the lower bound from  \eqref{eq:muell1q}, we note that the estimates 
\begin{align}\label{eq:malente}
\begin{split}
|\partial_\ell \nabla u_{j}| & =  \big((1+|\nabla u_{j}|^2)^{-\mu/2}| \partial_\ell \nabla u_{j}|^{2}\big)^{\frac{1}{2}}(1+|\nabla u_{j}|^2)^{\frac{\mu}{4}}\\
& \leq \lambda^{-\frac{1}{2}} |H_{j,\ell}|(1+|\nabla u_{j}|^2)^{\frac{\mu}{4}}\stackrel{\eqref{bound:V}}{\leq} \lambda^{-\frac{1}{2}}|H_{j,\ell}|(|V_{\mu}(\nabla u_{j})|+1)^{\frac{\mu}{2-\mu}} 
\end{split}
 \end{align}
hold $\mathscr{L}^{n}$-a.e. in $\Omega$ for $\ell\in\{1,2\}$. Due to $q\leq 2$, the upper bound from \eqref{eq:muell1q} implies that $|\nabla^{2}F|$ is bounded by $\Lambda$. Hence, the scaled Sobolev inequality in $n=2$ dimensions with the particular choice of $\xi_\ell \coloneqq (\partial_{\ell}u_{j})_{\mathcal{A}_{R}(x_{0})}$ and Poincar\'{e}'s inequality yield 
\begin{align}\label{eq:Munich1972}
\begin{split}
\mathrm{I} & \leq c\,\Lambda^{\frac{1}{2}}\bigg(\int_{\mathcal{A}_{R}(x_{0})}|\partial_{\ell}u_{j}-(\partial_{\ell}u_{j})_{\mathcal{A}_{R}(x_{0})}|^{2}\dif x\bigg)^{\frac{1}{2}} \\
 & \leq c\,\Lambda^{\frac{1}{2}}\int_{\mathcal{A}_{R}(x_{0})}|\partial_{\ell}\nabla u_{j}|\dif x\\
 & \!\!\! \stackrel{\eqref{eq:malente}}{\leq} c\,\Big(\frac{\Lambda}{\lambda}\Big)^{\frac{1}{2}}\int_{\mathcal{A}_{R}(x_{0})}|H_{j,\ell}|\,{(|V_{\mu}(\nabla u_{j})|+1)^{\frac{\mu}{2-\mu}}}\dif x.
\end{split}
\end{align}
In combination with \eqref{eq:bruderduell}, \eqref{eq:Munich1972} implies that 
\begin{align}\label{eq:Alexanderplatz}
\begin{split}
\int_{\ball_{R}(x_{0})} |H_{j,\ell}|^{2} \,\mathrm dx & \leq \frac{c}{R}\bigg(\int_{\mathcal{A}_{R}(x_{0})}|H_{j,\ell}|^{2}\dif x\bigg)^{\frac{1}{2}}\int_{\mathcal{A}_{R}(x_{0})}|H_{j,\ell}|\,(|V_{\mu}(\nabla u_{j})|+1)^{\frac{\mu}{2-\mu}}\dif x \\ & \qquad + \delta_{j}(1+|(\partial_{\ell}u_{j})_{\mathcal{A}_{R}(x_{0})}|^{2}), 
\end{split}
\end{align}
where now $c=c(N,\lambda,\Lambda)>0$. Now let $0<r<R$. By \eqref{eq:summary1n2}, the sequence $(V_{\mu}(\nabla u_{j}))$ is uniformly bounded in $\sobo^{1,2}(\ball_{2r}(x_0);\R^{N\times 2})$. Thus, by the chain rule, $(|V_{\mu}(\nabla u_{j})|+1)$ is uniformly bounded in $\sobo^{1,2}(\ball_{2r}(x_0))$ too. At this stage, we note that the assumption $\max\{2,q\}+3\mu<6$ ensures $\frac{\mu}{2-\mu}\in [1,2)$ in the case $1\leq q\leq 2$. Hence, Lemma~\ref{lem:FRESER98}  gives for all $2<p<\infty$ 
\begin{align}\label{eq:stevienicks}
\int_{\ball_r(x_0)}H_{j,\ell}^2\,\mathrm dx\leq \frac{c}{(\log_2(\frac{2R}r))^p} + \delta_{j}(1+|(\partial_{\ell}u_{j})_{\mathcal{A}_{R}(x_{0})}|^{2}), 
\end{align}
where $c>0$ only depends on $\mu,\lambda,\Lambda,N$, $R^{-1}\|V_\mu(\nabla u_{j})\|_{\lebe^2(\ball_{2R}(x_0))}+\|\nabla V_\mu(\nabla u_{j}))\|_{\lebe^2(\ball_{2R}(x_0)}$ $\|H_{j,\ell}\|_{\lebe^2(\ball_{2R}(x_0)}$ and is thus independent of $0<r<R$ and uniformly bounded in $j\in\mathbb N$. We recall that
\begin{equation*}
    |\nabla (V_\mu(\nabla u_j))|^2\stackrel{\eqref{eq:Vest}}{\leq} c(1+|\nabla u_j|^2)^{-\frac{\mu}{2}}|\nabla^2 u_j|^2\leq c \sum_{\ell=1}^2 H_{j,\ell}^2.
\end{equation*}
In view of Poincar\'{e}'s inequality and \eqref{eq:stevienicks}, we deduce that 
\begin{align}\label{eq:domagic}
\begin{split}
\frac{1}{r^{2}}\int_{\ball_{r}(x_{0})} \big|V_{\mu}(\nabla u_{j})&-(V_{\mu}(\nabla u_{j}))_{\ball_{r}(x_{0})}\big|^{2}\dif x \leq c\int_{\ball_{r}(x_{0})}|\nabla V_{\mu}(\nabla u_{j})|^{2}\dif x \\ 
& \leq \frac{c}{(\log_2(\frac{2R}r))^p} + c\,\delta_{j}\sum_{\ell=1}^2 (1+|(\partial_{\ell}u_{j})_{\mathcal{A}_{R}(x_{0})}|^{2}).
\end{split}
\end{align}
Combining $|V_\mu(z)|\leq  (1+|z|^2)^\frac{2-\mu}4$ and the strong convergence~\eqref{lim:ae:eq2}, we obtain 
\begin{align*}
V_{\mu}(\nabla u_{j})-(V_{\mu}(\nabla u_{j}))_{\ball_{r}(x_{0})}\to V_{\mu}(\nabla u)-(V_{\mu}(\nabla u))_{\ball_{r}(x_{0})}\qquad\text{strongly in $\lebe^{2}(\ball_{r}(x_{0}))$.}
\end{align*}
Based on the elementary estimate 
\begin{align*}
|(\partial_{\ell}u_{j})_{\mathcal{A}_{R}(x_{0})}|^{2} \leq \frac{c}{R^{4}}\bigg(\int_{\Omega}|\nabla u_{j}|\dif x\bigg)^{2}\stackrel{\eqref{eq:unifboundL1}}{\leq} \frac{c}{R^{4}}\frac{1}{\gamma^2}\Big(4 + \inf_{\sobo_{u_{0}}^{1,q}(\Omega;\R^{N})}\mathscr{F}[-;\Omega]\Big)^2 ,
\end{align*}
inequality \eqref{eq:domagic} then yields in the limit $j \to \infty$: 
\begin{align*}
\frac{1}{r^{2}}\int_{\ball_{r}(x_{0})} \big|V_{\mu}(\nabla u)-(V_{\mu}(\nabla u))_{\ball_{r}(x_{0})}\big|^{2}\dif x \leq \frac{c}{(\log_2(\frac{2R}r))^p}. 
\end{align*}
In consequence, Lemma~\ref{lem:campanato} implies that $V_{\mu}(\nabla u)$ is locally bounded and continuous in $\Omega$. Since $V_\mu \colon\R^{N\times 2}\to\R^{N\times 2}$ is a homeomorphism (see Step~1 of the proof of Proposition~\ref{prop:semimain}), the continuity of $\nabla u$ follows. 

\emph{Case $q>2$.} In the case $q>2$, we use the upper bound from \eqref{eq:muell1q} to find for $\ell\in\{1,2\}$
\begin{equation*}
\mathrm{I} \leq \sqrt{\Lambda}\biggl(\int_{\mathcal{A}_R(x_0)} (1+|\nabla u_{j}|^2)^\frac{q-2}2|\nabla u_{j}-\zeta|^2 \,\mathrm dx\biggr)^{\frac12},
\end{equation*}
where $\zeta=\begin{pmatrix}\xi_1 &\xi_{2}\end{pmatrix}\in \R^{N\times 2}$. In order to estimate the right-hand side from above, we recall that there exists a constant $c=c(N,q)>0$ such that
\begin{align}\label{eq:Wbound}
(1+|z_1|^2)^\frac{q-2}2|z_1-z_2|^2\leq c \,|V_{2-q}(z_1)-V_{2-q}(z_2)|^2
\qquad\text{for all}\;z_{1},z_{2}\in\R^{N\times 2},
\end{align}
where $V_{2-q}(z)\coloneqq (1+|z|^{2})^{-\frac{2-q}{4}}z$, see \eqref{eq:Vfunction}. Indeed, estimate \eqref{eq:Wbound} is a direct consequence of \cite[Eq. (2.4)]{GiaquintaModica86}. Now, by the Sobolev inequality in $n=2$ dimensions, we find with the choice $\zeta=V_{2-q}^{-1}((V_{2-q}(\nabla u_{j}))_{\mathcal A_{R}(x_0)})$:
\begin{align*}
\mathrm{I}\,\, & \!\! \stackrel{\eqref{eq:Wbound}}{\leq}   c\,\biggl(\int_{\mathcal{A}_R(x_0)} |V_{2-q}(\nabla u_{j})-V_{2-q}(\zeta)|^2 \,\mathrm dx\biggr)^{\frac12} \\ & 
= c\, \biggl(\int_{\mathcal{A}_{R}(x_{0})}|V_{2-q}(\nabla u_{j})-(V_{2-q}(\nabla u_{j}))_{\mathcal{A}_{R}(x_{0})}|^{2}\dif x \biggr)^\frac{1}{2}\\ 
& \leq c\,\int_{\mathcal{A}_R(x_0)} |\nabla (V_{2-q}(\nabla u_{j}))| \,\mathrm dx.
\end{align*}
To conclude the proof, we use that there exists a constant $c=c(q,\lambda,\Lambda)>0$ such that 
\begin{align*}
|\nabla (V_{2-q}(\nabla u_{j}))| & \leq c\,(1+|\nabla u_{j}|^2)^\frac{q-2}4 |\nabla^2 u_{j}| \\ & \leq c\, \sum_{\ell=1}^{2}|H_{j,\ell}| (1+|\nabla u_j|^2)^\frac{q+\mu-2}4\stackrel{\eqref{bound:V}}{\leq} c\,\sum_{\ell=1}^{2} |H_{j,\ell}| (|V_{\mu}(\nabla u_j)|+1)^\frac{q+\mu-2}{2-\mu}.
\end{align*}
Our condition  \eqref{eq:qmupropaganda} with $q\geq2$ then yields $\frac{q+\mu-2}{2-\mu}<2$. With a similar argument as in the case $q\leq 2$, we then infer the continuity of $V_{2-q}(\nabla u)$ together with the corresponding estimates. As in the case $q\leq 2$, this ensures the continuity of $\nabla u$. 

It remains to upgrade the continuity of $\nabla u$ to H\"older continuity. To this end, let $\ell\in\{1,2\}$. We recall that $\partial_\ell u\in \sobo_{\locc}^{1,t}(\Omega;\R^{N})$ for all $t<2$, see Proposition~\ref{lem:ae}, and thus we deduce from the differentiated Euler--Lagrange system~\eqref{eq:ulocalminlin} by approximation that 
\begin{equation}\label{eq:brezis0}
\int_{\Omega}\langle A\nabla \partial_\ell u,\nabla \varphi\rangle\,\mathrm dx=0\qquad\mbox{for all $\varphi\in \sobo_{c}^{1,p}(\Omega;\R^{N})$ with $p>2$},
\end{equation}
where we have used the shorthand notation $A \coloneqq \nabla^2 F(\nabla u)$.
Since, for every $K\Subset \Omega$,  $\nabla u$ is bounded and continuous in $K$,  $A$ is uniformly elliptic and continuous in $K$. As shall be addressed in detail below, we are now in a position to apply \cite[Theorem A1.1]{ANCONA09} to obtain that $\partial_\ell u\in \sobo^{1,p}(K';\R^{N})$ for all $p<\infty$ and all open subsets  $K'\Subset K$. Thus, by Morrey's embedding, it follows that $u\in \hold_{\locc}^{1,\alpha}(\Omega;\R^{N})$ for all $0<\alpha<1$ as claimed.  

Since \cite{ANCONA09} is formulated for elliptic equations (that is, $N=1$), we briefly give the argument underlying~\cite[Theorem A1.1]{ANCONA09} for the convenience of the reader. In view of standard interior $\lebe^p$-estimates for linear elliptic systems (see, e.g.,\ \cite[Theorem 7.2]{Giaquinta12}), it suffices to show that $\partial_{\ell}u\in \sobo_{\locc}^{1,2}(\Omega;\R^{N})$. For an arbitrary ball $\ball_R=\ball_{R}(x_0)\Subset \Omega$, we shall now establish that  $\partial_\ell u\in \sobo^{1,2}(\ball_{R/2};\R^{N})$. 

To this end, let $\Phi\in \hold_c^\infty(\ball_{R};\R^{N\times 2})$ be arbitrary with $\|\Phi\|_{\lebe^2(\ball_{R})}=1$, and consider the unique solution $v\in \sobo^{1,2}_0(\ball_{R};\R^{N})$ of
\begin{equation}\label{eq:brezis1}
\int_{\ball_{R}}\langle A\nabla v,\nabla \varphi\rangle \,\mathrm dx = \int_{\ball_{R}}\langle \Phi,\nabla \varphi\rangle \,\mathrm dx\qquad \text{for all } \varphi\in \sobo_0^{1,2}(\ball_{R};\R^{N}).
\end{equation}
Using the uniform ellipticity of $A$ and $\|\Phi\|_{\lebe^2(\ball_{R})}=1$, we find that  $\|\nabla v\|_{\lebe^2(\ball_{R})}\leq C<\infty$. Here and in the rest of the proof, $C>0$ denotes a finite positive constant depending only on~$N$ and the uniform ellipticity ratio of~$A$. Moreover, standard $\lebe^{p}$-theory, see \cite[Theorem 7.2]{Giaquinta12}, in combination with the continuity of $A$ and the assumption $\Phi\in \hold_c^\infty(\ball_{R};\R^{N\times 2})$ yields $\nabla v\in \lebe_{\rm loc}^p(\ball_{R};\R^{N\times 2})$ for all $p<\infty$. The higher integrability of $\nabla v$ ensures that \eqref{eq:brezis1} is also valid for $\varphi\in \sobo_{c}^{1,s}(\ball_{R};\R^{N})$ with $s>1$. In particular, we may use $\varphi=\eta (\partial_\ell u-(\partial_\ell u)_{\ball_{R}})$ as a test function in \eqref{eq:brezis1}, where $\eta\in \hold_{c}^{\infty}(\ball_{R};[0,1])$ is a smooth cut-off function satisfying $\eta=1$ in $\ball_{R/2}$ and $|\nabla\eta|\leq 8/R$. This yields
\begin{align*}
\int_{\ball_{R}}\langle \Phi,\eta \nabla \partial_{\ell}u\rangle \,\mathrm dx &=\int_{\ball_{R}}\langle A\nabla v,\eta \nabla \partial_{\ell} u\rangle\,\mathrm dx\\
&+\int_{\ball_{R}} \big( \langle A\nabla v, (\partial_{\ell} u-(\partial_{\ell}u)_{\ball_{R}})\otimes \nabla\eta\rangle-\langle \Phi, (\partial_{\ell} u -(\partial_{\ell} u)_{\ball_{R}})\otimes\nabla\eta\rangle \big) \,\mathrm dx. 
\end{align*}
The second integral on the right-hand side can be estimated from above by H\"older's and Sobolev's inequalities by
$$
(\|A\nabla v\|_{\lebe^2(\ball_{R})}+\|\Phi\|_{\lebe^2(\ball_{R})})\|\nabla \eta\|_{\lebe^\infty(\ball_{R})}\| \partial_{\ell} u -(\partial_{\ell} u)_{\ball_{R}}\|_{\lebe^2(\ball_{R})}\leq \frac{C}{R}\|\nabla \partial_{\ell} u\|_{\lebe^1(\ball_{R})}.
$$
For the remaining term, we use the symmetry and boundedness of $A$, equation \eqref{eq:brezis0} with $\varphi=\eta (v-(v)_{\ball_{R}})\in\lebe^{\infty}(\ball_{R};\R^{N})$ and Sobolev's inequality in the form
\begin{align*}
\int_{\ball_{R}}\langle A\nabla v,\eta \nabla \partial_{\ell} u\rangle\,\mathrm dx& =\int_{\ball_{R}}\langle A\nabla \partial_{\ell} u,\eta \nabla v\rangle\,\mathrm dx=-\int_{\ball_{R}}\langle A\nabla \partial_{\ell} u, (v-(v)_{\ball_{R}})\otimes \nabla \eta\rangle \,\mathrm dx\\
& \leq\|A\nabla \partial_{\ell}u \|_{\lebe^{\frac{3}{2}}(\ball_{R})}\|\nabla \eta\|_{\lebe^\infty(\ball_{R})}\|v-(v)_{\ball_{R}}\|_{\lebe^3(\ball_{R})}\leq \frac{C}{R^{\frac{1}{3}}}\|\nabla \partial_{\ell}u \|_{\lebe^{\frac{3}{2}}(\ball_{R})}.
\end{align*}
The previous three displayed formulas in combination with H\"older inequality $\|\cdot\|_{\lebe^1(\ball_{R})}\leq \mathscr{L}^{2}(\ball_{R})^\frac{1}{3}\|\cdot\|_{\lebe^\frac{3}{2}(\ball_{R})}$ imply that 
\begin{equation}
\sup \bigg\{ \int_{\ball_{R}}\langle \Phi,\eta \nabla \partial_{\ell}u\rangle \,\mathrm dx \colon \Phi\in \hold_c^\infty(\ball_{R};\R^{N\times n}), \, \|\Phi\|_{\lebe^2(\ball_{R})}=1 \bigg\} \leq \frac{C}{R^{\frac{1}{3}}}\|\nabla \partial_{\ell} u\|_{\lebe^\frac{3}{2}(\ball_{R})},
\end{equation}
which ensures the claim $\nabla \partial_\ell u\in \lebe^{2}(\ball_{R/2};\R^{N\times 2})$. This completes the proof. 
\end{proof}

\begin{rem}[On the order of limit passages]\label{rem:orderproof} As to the specific set-up of the above proof, note that the classical Frehse--Seregin lemma (see \cite[Lem. 4.2]{FRESER98}) could only be applied if the additional term $\delta_{j}(1+|(\partial_{\ell}u_{j})_{\mathcal{A}_{R}(x_{0})}|^{2})$ were absent in \eqref{eq:Alexanderplatz}. This, in turn, would force us to directly perform the limit passage $j\to\infty$ in \eqref{eq:Alexanderplatz}. At this stage, however, there is no convergence result for $(H_{j,\ell})$ at our disposal. In the above proof, we first pass from second to first order quantities in \eqref{eq:domagic} (for which convergence results are available at this point), subsequently use the modified Frehse--Seregin Lemma~\ref{lem:FRESER98} and finally  send $j\to\infty$. 
\end{rem}

\section{Appendix}\label{sec:appendix}
\subsection{Relaxations}
As a key point of the main part and different from previous contributions, the  regularity assertions of Theorems~\ref{thm:main} and~\ref{thm:main2} do not make use of integral representations of the relaxed functional $\overline{\mathscr{F}}_{u_{0}}^{*}[-;\Omega]$. In view of the ubiquity of such representations for \emph{purely linear growth} functionals (i.e., $q=1$, see, e.g., \cite{BECSCH13,BILDHAUER03}), we briefly discuss here the underlying difficulties in establishing such representations in the general $(1,q)$-growth case. As an upshot, a detour via an integral representation -- which  might a priori only be available for a strictly smaller range of $q$ than displayed in Theorems~\ref{thm:main} and~\ref{thm:main2} -- would complicate the line of argument \emph{while not being necessary} based on  the approach developed in  Sections \ref{sec:relax}--\ref{sec:proofmain2}. 

To this end, it is instructive to firstly consider the case without prescribed Dirichlet boundary values. We begin with:  
\begin{lem}[Koch and Kristensen, {\cite[Thm. 2]{KOCKRI22}}]
Let $\Omega\subset\R^{n}$ be open and bounded with Lipschitz boundary. Moreover, let $1\leq q<\infty$ and let the convex integrand  $F\in\hold(\R^{N\times n})$ satisfy the growth bound \eqref{eq:1qgrowth}. For $u\in\bv(\Omega;\R^{N})$, we put 
\begin{align*}
\overline{\mathscr{F}}^{*}[u;\Omega] \coloneqq  \inf\bigg\{\liminf_{j\to\infty}\int_{\Omega}F(\nabla u_{j})\dif x\colon\;(u_{j})\text{ in } \sobo^{1,1}(\Omega;\R^{N}),\;u_{j}\stackrel{*}{\rightharpoonup}u\;\text{in}\;\bv(\Omega;\R^{N}) \bigg\}. 
\end{align*}
Then we have the \emph{integral representation}
\begin{align}\label{eq:intrep}
\overline{\mathscr{F}}^{*}[u;\Omega] = \overline{\mathscr{F}}_{\mathrm{int}}^{*}[u;\Omega] \coloneqq  \int_{\Omega}F(\nabla u)\dif x + \int_{\Omega}F^{\infty}\Big(\frac{\dif \D^{s}u}{\dif|\D^{s}u|}\Big)\dif|\D^{s}u|, 
\end{align}
where we have used the Lebesgue--Radon--Nikod\'{y}m  decomposition \eqref{eq:gradientdecomp} of $\D u$. 
\end{lem}
Whereas \cite{KOCKRI22} uses the machinery of Young measures to arrive at '$\geq$' in \eqref{eq:intrep}, we note that this step can be accomplished as in the proof of Proposition~\ref{prop:finiteness} by use of the Reshetnyak lower semicontinuity theorem. Once no boundary values are fixed (as in the definition of $\overline{\mathscr{F}}^{*}[-;\Omega]$), a routine mollification of $u$ in conjunction with Jensen's inequality yields '$\leq$' in \eqref{eq:intrep} too; see the proof of \cite[Thm. 2]{KOCKRI22}. If we work with \emph{fixed} boundary values as in \eqref{eq:LSM}, an integral representation for $\overline{\mathscr{F}}_{u_{0}}^{*}[u;\Omega]$ does not follow by similar means. More precisely, if we denote 
\begin{align*}
\overline{\mathscr{F}}_{\mathrm{int},u_{0}}^{*}[u;\Omega] \coloneqq \overline{\mathscr{F}}_{\mathrm{int}}^{*}[u;\Omega] + \int_{\partial\Omega}F^{\infty}(\mathrm{tr}_{\partial\Omega}(u_{0}-u)\otimes\nu_{\partial\Omega})\dif\mathscr{H}^{n-1}, 
\end{align*}
then an integral representation of $\overline{\mathscr{F}}_{u_{0}}^{*}[u;\Omega]$ would mean  that $\overline{\mathscr{F}}_{u_{0}}^{*}[u;\Omega]=\overline{\mathscr{F}}_{\mathrm{int},u_{0}}^{*}[u;\Omega]$. As in the proof of Proposition~\ref{prop:finiteness}, '$\geq$' then follows from the Reshetnyak lower semicontinuity theorem. Hence, an integral representation requires a recovery sequence $(u_{j})$ in $u_{0}+\sobo_{0}^{1,q}(\Omega;\R^{N})$; in the situation considered here, this amounts to 
\begin{align}\label{eq:freeasabird}
\liminf_{j\to\infty}\int_{\Omega}F(\nabla u_{j})\dif x \leq \overline{\mathscr{F}}_{\mathrm{int},u_{0}}^{*}[u;\Omega].  
\end{align}
If $q=1$ in \eqref{eq:1qgrowth} and so $F$ has linear growth from above \emph{and} below, it follows from \cite[Appendix B]{BILDHAUER03} and Reshetnyak's continuity theorem that there exists a sequence $(u_{j})$ in $ u_{0}+\hold_{c}^{\infty}(\Omega;\R^{N})$ such that \eqref{eq:freeasabird} holds with equality. As explained in Example~\ref{ex:reshnon}, Reshetnyak's continuity theorem does not extend to the case where $q>1$. This, in turn, implies that \eqref{eq:freeasabird} must be established by independent means. For suitable cut-off functions $\eta_{j}\in\hold_{c}^{\infty}(\Omega;[0,1])$ with $\eta_{j}\to 1$ pointwise everywhere in $\Omega$, the natural candidate for such a recovery sequence is $u_{j}\coloneqq u_{0} + \rho_{\varepsilon_{j}}*(\eta_{j}(u-u_{0}))$, where $(\varepsilon_{j})$ in $(0,1)$ tends to zero sufficiently fast. It is then readily discovered that \eqref{eq:freeasabird} essentially reduces to showing that 
\begin{align}\label{eq:michaeljackson}
\liminf_{j\to\infty}\int_{\Omega}F((u-u_{0})\otimes\nabla \eta_{j})\dif x \leq \int_{\partial\Omega}F^{\infty}(\mathrm{tr}_{\partial\Omega}(u_{0}-u)\otimes\nu_{\partial\Omega})\dif\mathscr{H}^{n-1}. 
\end{align}
In the purely linear growth case (where $q=1$), this inequality can be established, e.g., by use of the coarea formula and the fundamental theorem of calculus. If $q>1$, an integral representability and  \eqref{eq:michaeljackson} indicate the necessity of the exponent restriction $q<\frac{n}{n-1}$ in view of Proposition~\ref{prop:finiteness}. Even in this case, however, a similar approach as for $q=1$ seems difficult. This is due to the fact that the bulk integrals of $F((u-u_{0})\otimes\nabla\eta_{j})$ are only expected to converge to the corresponding boundary terms close to boundary points where $F^{\infty}(\mathrm{tr}_{\partial\Omega}(u_{0}-u)\otimes\nu_{\partial\Omega})$ is finite; if the latter is not fulfilled, Proposition~\ref{prop:finiteness} in turn indicates that the bulk approximations must decay to zero sufficiently fast as $j\to\infty$. This requires a delicate case distinction and an additional argument that these two cases do not interfere too much. While we believe that, in principle, this might be possible with substantially refined methods, the above reasoning indicates a strictly smaller range of $q$ for such an integral representation than the one given in Theorems \ref{thm:main}, \ref{thm:main2}. 

Most importantly, the underlying chief obstruction is not the potential presence of the singular parts in the interior, but the boundary behaviour of recovery sequences. As implicitly noted in \cite{KOCKRI22}, issues of this type do not arise when imposing the Dirichlet data  only asymptotically; for such versions of relaxed functionals, it is however not clear how to establish key properties such as, e.g., Theorem~\ref{thm:consistency}. 

\subsection{Proof of Lemma \ref{lem:exponential}}\label{sec:AppendixExponential} In order to prove  the lower bound in \eqref{eq:tenpenny}, we may assume that $\!\!\!{\;\;-}{\!\!\!\||v|^{2-\mu}\|}_{\exp\lebe^{1}(\ball_{r}(x_{0}))}<\infty$, as otherwise there is nothing to prove. Moreover, it is customary to put $\widetilde{\Phi}(z)\coloneqq \exp(|z|)-1$ for $z\in\R$. By continuity and monotonicity of $\Phi_{\mu}$, there exists a unique number $t_{2}=t_{2}(t_{1},\mu)>0$ such that 
\begin{align}\label{eq:alexjames}
\Phi_{\mu}(t_{2}) = \frac{1}{2}. 
\end{align}
{We may assume that $t_{2}<t_{1}$.} Next, we choose 
\begin{align*}
\Theta \coloneqq \Big(\Big(\frac{t_{1}}{t_{2}}\Big)^{2-\mu}-1\Big){\;\;-}{\!\!\!\||v|^{2-\mu}\|}_{\exp\lebe^{1}(\ball_{r}(x_{0}))}. 
\end{align*}
In consequence, we have 
\begin{align*}
&{\;\;-}{\!\!\!\||v|^{2-\mu}\|}_{\exp\lebe^{1}(\ball_{r}(x_{0}))} + \Theta =  \Big(\frac{t_{1}}{t_{2}}\Big)^{2-\mu}{\;\;-}{\!\!\!\||v|^{2-\mu}\|}_{\exp\lebe^{1}(\ball_{r}(x_{0}))}  \\ & \Longrightarrow t_{2}^{2-\mu}(\!\!\!{\;\;-}{\!\!\!\||v|^{2-\mu}\|}_{\exp\lebe^{1}(\ball_{r}(x_{0}))} +\Theta) = t_{1}^{2-\mu}{\;\;-}{\!\!\!\||v|^{2-\mu}\|}_{\exp\lebe^{1}(\ball_{r}(x_{0}))}  \\ 
& \Longrightarrow t_{2}(2(\!\!\!{\;\;-}{\!\!\!\||v|^{2-\mu}\|}_{\exp\lebe^{1}(\ball_{r}(x_{0}))} +\Theta))^{\frac{1}{2-\mu}} =  t_{1}(2\!\!\!{\;\;-}{\!\!\!\||v|^{2-\mu}\|}_{\exp\lebe^{1}(\ball_{r}(x_{0}))} )^{\frac{1}{2-\mu}}.
\end{align*}
Hence, defining
\begin{equation*}
\lambda_1 \coloneqq (2(\!\!\!{\;\;-}{\!\!\!\||v|^{2-\mu}\|}_{\exp\lebe^{1}(\ball_{r}(x_{0}))}))^{\frac{1}{2-\mu}}\qquad\text{and}\qquad\lambda_2 \coloneqq (2(\!\!\!{\;\;-}{\!\!\!\||v|^{2-\mu}\|}_{\exp\lebe^{1}(\ball_{r}(x_{0}))}+\Theta))^{\frac{1}{2-\mu}},
\end{equation*}
we conclude for $x\in\ball_{r}(x_{0})$ that  $|v(x)| \geq t_2 \lambda_2$ holds if and only if $|v(x)| \geq t_1 \lambda_1$. Therefore, introducing 
\begin{equation}
\label{eq:A_exp}
\mathfrak{A} \coloneqq \big\{x\in\ball_{r}(x_{0})\colon |v(x)| < t_2 \lambda_2 \big\}
\end{equation}
we have 
\begin{equation}
\label{eq:A_exp_complement}
\ball_{r}(x_{0}) \setminus \mathfrak{A} =  \big\{x\in\ball_{r}(x_{0})\colon |v(x)| \geq t_2 \lambda_2 \big\} = \big\{x\in\ball_{r}(x_{0})\colon |v(x)| \geq t_1 \lambda_1 \big\} \eqqcolon \mathfrak{B}.
\end{equation}
This allows us to estimate via $\lambda_1 < \lambda_2$
\begin{align*}
\dashint_{\ball_{r}(x_{0})}\Phi_{\mu}\Big(\frac{|v|}{\lambda_2}\Big)\dif x
 & = \frac{1}{\mathscr{L}^{n}(\ball_{r}(x_{0}))}\int_{ \mathfrak{A}}\Phi_{\mu}\Big(\frac{|v|}{\lambda_2}\Big)\dif x + \frac{1}{\mathscr{L}^{n}(\ball_{r}(x_{0}))}\int_{ \mathfrak{B}}\Phi_{\mu}\Big(\frac{|v|}{\lambda_2}\Big)\dif x \\
 & \leq  \frac{1}{\mathscr{L}^{n}(\ball_{r}(x_{0}))}\int_{ \mathfrak{A}}\Phi_{\mu}(t_{2})\dif x + \frac{1}{\mathscr{L}^{n}(\ball_{r}(x_{0}))}\int_{ \mathfrak{B}}\Phi_{\mu}\Big(\frac{|v|}{\lambda_1}\Big)\dif x \\
 & \!\!\! \stackrel{\eqref{eq:alexjames}}{\leq}  \frac{1}{2} + \frac{1}{\mathscr{L}^{n}(\ball_{r}(x_{0}))}\int_{ \mathfrak{B}}{\Phi}_{\mu}\Big(\frac{|v|}{\lambda_1}\Big)\dif x \eqqcolon \mathrm{I}.
\end{align*}
On $\mathfrak{B}$, we have $|v|\geq t_{1}\lambda_{1}$  
and therefore, {recalling our definition of $t_{1}$, see \eqref{eq:escobar}, and that $\widetilde{\Phi}(z)\coloneqq \exp(|z|)-1$}, 
\begin{align*}
\Phi_{\mu}\Big(\frac{|v|}{\lambda_{1}} \Big) = \Phi_{\mu}\Big(\frac{|v|}{(2\!\!\!{\;\;-}{\!\!\!\||v|^{2-\mu}\|}_{\exp\lebe^{1}(\ball_{r}(x_{0}))})^{\frac{1}{2-\mu}}}\Big) & =\widetilde{\Phi}\Big(\frac{|v|^{2-\mu}}{2\!\!\!{\;\;-}{\!\!\!\||v|^{2-\mu}\|}_{\exp\lebe^{1}(\ball_{r}(x_{0}))}}\Big) \\ & \leq \frac{1}{2}\widetilde{\Phi}\Big(\frac{|v|^{2-\mu}}{\!\!\!{\;\;-}{\!\!\!\||v|^{2-\mu}\|}_{\exp\lebe^{1}(\ball_{r}(x_{0}))}}\Big), 
\end{align*}
where the ultimate inequality follows from the convexity of $\widetilde{\Phi}$ and $\widetilde{\Phi}(0)=0$. In particular, our definition of $t_{1}$ gives us 
\begin{align*}
\mathrm{I} & \leq \frac{1}{2} + \frac{1}{2\mathscr{L}^{n}(\ball_{r}(x_{0}))}\int_{\mathfrak{B}}\widetilde{\Phi}\Big(\frac{|v|^{2-\mu}}{(\!\!\!{\;\;-}{\!\!\!\||v|^{2-\mu}\|}_{\exp\lebe^{1}(\ball_{r}(x_{0}))})}\Big)\dif x \leq 1 
\end{align*}
by the very definition of $\!\!\!{\;\;-}{\!\!\!\||v|^{2-\mu}\|}_{\exp\lebe^{1}(\ball_{r}(x_{0}))}$. Hence, by the definition of the Luxemburg norm and recalling that $t_{2}=t_{2}(t_{1},\mu)>0$, we arrive at 
\begin{align*}
\!\!\!{\;\;-}{\!\!\!\|v\|}_{\lebe^{\Phi_{\mu}}(\ball_{r}(x_{0}))}& \leq \lambda_{2}  \leq c(t_{1},\mu)\, \!\!\!{\;\;-}{\!\!\!\||v|^{2-\mu}\|}_{\exp\lebe^{1}(\ball_{r}(x_{0}))}^{\frac{1}{2-\mu}}.
\end{align*}
This establishes the lower bound in \eqref{eq:tenpenny}; the upper bound follows by analogous means. \hfill $\square$

\subsection{Proof of Lemma~\ref{lem:FRESER98}}\label{sec:proofofFS}
The proof follows the lines of \cite[Lemma 4.1]{FRESER98},  where the claim is established in the case $\alpha=1$ and $\theta=0$. For the convenience of the reader and since we require the enlarged range of~$\alpha$ and~$\theta$, we briefly sketch the argument. We recall the following decay estimate: For every $Q\in[1,\infty)$ there exists $C=C(Q)<\infty$ such that
\begin{equation}\label{est:decayh1n2}
\biggl(\fint_{\mathcal{A}_r (x_{0})}\big| h-(h)_{\mathcal{A}_R (x_{0})}\big|^Q\,\mathrm dx\biggr)^\frac1Q\leq C\sqrt{\log_2(2R/r)}\|\nabla h\|_{\lebe^2(\ball_{2R}(x_0))}
\end{equation}
holds for every function $h\in \sobo^{1,2}(\ball_{2R}(x_0)))$ and all $r \in (0,R)$. For the sake of completeness, we provide the argument for \eqref{est:decayh1n2} at the end of the proof. Set $f(\rho)\coloneqq \int_{\ball_\rho(x_0)}H^2\,\mathrm dx$. Combining the triangle inequality in the form
\begin{align*}
\fint_{\mathcal{A}_\rho(x_0)}|h|^\alpha |H|\,\mathrm dx\leq2^{\alpha-1}\fint_{\mathcal{A}_\rho(x_0)}\bigl| h-(h)_{\mathcal{A}_R (x_{0})}\bigr|^\alpha |H|\,\mathrm dx+2^{\alpha-1}\bigl|(h)_{\mathcal{A}_R (x_{0})} \bigr|^\alpha \fint_{\mathcal{A}_\rho (x_{0})}|H|\,\mathrm dx
\end{align*}
with the assumption \eqref{eq1:FreSe99variant} and the decay estimate~\eqref{est:decayh1n2}, we obtain for all $0<\rho<R$
\begin{align}
f(\rho) & \leq c2^{\alpha-1}L\biggl(\int_{\mathcal{A}_\rho (x_{0})} |H|^2\,{\mathrm d}x\biggr)  \biggl(\biggl(\fint_{\mathcal{A}_\rho (x_{0})} \bigl|h-(h)_{\mathcal{A}_R (x_{0})}\bigr|^{2\alpha}\,\mathrm dx\biggr)^\frac12+\bigl|(h)_{\mathcal{A}_R (x_{0})} \bigr|^\alpha\biggr) + \theta \notag\\ 
& \leq c2^{\alpha-1}L\biggl(\int_{\mathcal{A}_\rho (x_{0})}|H|^2\,{\mathrm d}x\biggr)  \biggl(C\sqrt{\log_{2}(2R/\rho)}^{\alpha}\|\nabla h\|_{\lebe^{2}(\ball_{2R}(x_{0}))}^{\alpha}+\bigl|(h)_{\mathcal{A}_R (x_{0})} \bigr|^\alpha\biggr) + \theta \label{eq:besuchszeit}
\\ 
& \leq c_{1}\sqrt{\log_{2}(2R/\rho)}^{\alpha}\int_{\mathcal{A}_{\rho}(x_{0})} |H|^{2}\dif x + \theta, \notag  
\end{align}
where $c_{1}=c_{1}(\alpha,\|\nabla h\|_{\lebe^{2}(\ball_{2R}(x_{0}))}+R^{-1}\|h\|_{\lebe^{2}(\ball_{2R}(x_{0}))},L)$. The above estimate, in combination with routine hole-filling, implies with $\beta \coloneqq \alpha/2\in (0,1)$ that 
\begin{equation}\label{L:appendix:estfrho}
f(\rho)\leq \frac{c_1\log_2(2R/\rho)^\beta}{c_1\log_{2}(2R/\rho)^\beta+1}f(2\rho) + \frac{\theta}{c_1\log_{2}(2R/\rho)^\beta+1} \qquad \text{for all }  \rho\in(0,R].
\end{equation}
It remains to show
$$
f(r)\leq c \log_2(2R/r)^{-p} f(2R)+\theta
$$
for all $p \in [1,\infty)$, where the constant $c>0$ is as in the claim of the lemma. For simplicity, we suppose in what follows that $M \coloneqq \log_2(R/r)\in \mathbb N$, the general case can be deduced by a simple post-processing argument. From \eqref{L:appendix:estfrho}, we deduce that
\begin{align*}
f(2^j r)\leq \frac{c_1 (M+1-j)^{\beta}}{c_1 (M+1-j)^{\beta}+1}f(2^{j+1}r) + \frac{\theta}{c_1(M+1-j)^\beta+1}\qquad \text{for all } j\in\{0,\dots,M-1\}.  
\end{align*}
Iterating this inequality gives us 
\begin{align}\label{L:appendix:estfr}
\begin{split}
f(r) & \leq \bigg(\prod_{j=0}^{M-1}\frac{c_1(M+1-j)^\beta}{c_1(M+1-j)^\beta+1}\bigg)f(2^Mr) + \bigg(\sum_{j=0}^{M-1}\frac{\prod_{i=0}^{j-1}\frac{c_1(M+1-i)^\beta}{c_1(M+1-i)^\beta+1}}{c_1(M+1-j)^\beta+1}\bigg)\theta\\ 
& = \bigg(\prod_{k=1}^{M} \frac{c_{1}(k+1)^{\beta}}{c_{1}(k+1)^{\beta}+1}\bigg)f(2^{M}r) + \bigg(\sum_{k=1}^{M}\frac{\prod_{i=0}^{M-k-1}\frac{c_1(M+1-i)^\beta}{c_1(M+1-i)^\beta+1}}{c_1(k+1)^\beta+1}\bigg)\theta \\ 
& = \bigg(\prod_{k=1}^{M} \frac{c_{1}(k+1)^{\beta}}{c_{1}(k+1)^{\beta}+1}\bigg)f(2^{M}r) + \bigg(\sum_{k=1}^{M}\frac{\prod_{\ell=k+1}^{M}\frac{c_1(\ell+1)^\beta}{c_1(\ell+1)^\beta+1}}{c_1(k+1)^\beta+1}\bigg)\theta.
\end{split} 
\end{align}
Here, we employed the index shift $k=M-j$ and $\ell=M-i$ in the ultimate two  equations and adopted the convention of the empty product to equal~$1$. Now, we estimate the two terms on the above right-hand side separately. Setting
$$
z_{k,M} \coloneqq \prod_{\ell=k}^{M}\frac{c_1(\ell+1)^\beta}{c_1(\ell+1)^\beta+1},
$$
we observe that
$$
z_{k+1,M}-z_{k,M}=\Big(1-\frac{c_1(k+1)^\beta}{c_1(k+1)^\beta +1}\Big)z_{k+1,M}=\frac{z_{k+1,M}}{c_1(k+1)^\beta+1}
$$
and thus, by a telescope sum argument for the second term on the right-hand side of \eqref{L:appendix:estfr}, we find 
\begin{align}\label{L:appendix:estthetaterm}
\begin{split}
 \bigg(\sum_{k=1}^{M}\frac{\prod_{\ell=k+1}^{M}\frac{c_1(\ell+1)^\beta}{c_1(\ell+1)^\beta+1}}{c_1(k+1)^\beta+1}\bigg)\theta=&\sum_{k=1}^{M}\frac{z_{k+1,M}}{c_1(k+1)^\beta+1}\theta=\sum_{k=1}^{M}(z_{k+1,M}-z_{k,M})\theta\\
 =&(z_{M+1,M}-z_{1,M})\theta\leq\theta
\end{split} 
\end{align}
for every $M$, where we have used in the last inequality $z_{M+1,M}=1$ and $z_{1,M}\geq0$. Next, we estimate the first term on the right-hand side of \eqref{L:appendix:estfr}. Using the inequality $1-x \leq \mathrm{e}^{-x}$ for all $x \in \R$, we find
$$
\prod_{k=1}^M\frac{c_1(k+1)^\beta}{c_1(k+1)^\beta+1}=\prod_{k=1}^M\Big(1-\frac{1}{c_1(k+1)^\beta+1}\Big)\leq \mathrm{e}^{-\sum_{k=1}^M\frac{1}{c_1(k+1)^\beta+1}}.
$$
Taking into account
\begin{align*}
\sum_{k=1}^M\frac{1}{c_1(k+1)^\beta+1}\geq \frac1{c_1+1}\sum_{k=2}^{M+1}\frac1{k^\beta}\geq \frac1{c_1+1} \int_2^{M+2} x^{-\beta} \dif x = \frac{(M+2)^{1-\beta}-2^{1-\beta}}{(c_1+1)(1-\beta)}
\end{align*}
and~\eqref{L:appendix:estthetaterm}, we conclude by ~\eqref{L:appendix:estfr} that
\begin{equation*}
f(r)\leq \mathrm{e}^{-\frac{(M+2)^{1-\beta}}{(c_1+1)(1-\beta)}}\mathrm{e}^{\frac{2^{1-\beta}}{(c_1+1)(1-\beta)}}f(2^Mr)+\theta.
\end{equation*}
Clearly, this estimate in combination with $\mathrm{e}^{-t}\leq c(p)t^{-p}$ for all $p > 0$ and $M=\log_2(R/r)$ implies the claim. 

Finally, we recall the argument for \eqref{est:decayh1n2}. For simplicity, we again assume $R=2^Mr$ with $M\in\mathbb N$. Writing
$$
(h)_{\mathcal{A}_R (x_{0})} =
(h)_{\mathcal{A}_r (x_{0})}+\sum_{j=0}^{M-1}\big[ (h)_{\mathcal{A}_{2^{j+1}r}(x_{0})} -  (h)_{\mathcal{A}_{2^{j}r}(x_{0})}\big] 
$$
as a telescope sum, we have
\begin{align*}
    \biggl(\fint_{\mathcal{A}_r (x_{0})}\big| h-(h)_{\mathcal{A}_R (x_{0})}\big|^Q \,\mathrm dx\biggr)^\frac1Q\leq&\biggl(\fint_{\mathcal{A}_r (x_{0})}\big| h-(h)_{\mathcal{A}_r (x_{0})} \big|^Q\,\mathrm dx\biggr)^\frac1Q\\
    &+\sum_{j=0}^{M-1}\big| (h)_{\mathcal{A}_{2^{j+1}r}(x_{0})} -  (h)_{\mathcal{A}_{2^{j}r}(x_{0})}\big|.
\end{align*}
We next note that each term in the previous sum can be estimated by a multiple  of the mean deviation of~$h$ on the set $\ball_{2^{j+2}r}(x_0)\setminus \ball_{2^{j}r}(x_0)$. The Sobolev--Poincar\'e inequality in dimension $n=2$ then implies that for every $Q\in[1,\infty)$ there exists a constant $c=c(Q)>0$ such that 
\begin{align*}
    \biggl(\fint_{\mathcal{A}_r (x_{0})}\big| h-(h)_{\mathcal{A}_R (x_{0})}\big|^Q \,\mathrm dx\biggr)^\frac1Q 
    & \leq c \sum_{j=0}^{M-1}\bigg(\int_{\ball_{2^{j+2}r}(x_0)\setminus \ball_{2^{j}r}(x_0)}|\nabla h|^2 \mathrm dx\bigg)^\frac12 \\
    & \leq c\sqrt{M}\bigg(2\int_{\ball_{2^{M+1}r}(x_0)}|\nabla h|^2\mathrm dx\bigg)^\frac{1}{2},
\end{align*}
where we have also used $\mathcal{A}_r (x_{0}) \subset \ball_{4r}(x_0) \setminus \ball_r(x_0)$ and the Cauchy--Schwarz inequality. Since $M=\log_2(R/r)$ by assumption, we have proved the estimate \eqref{est:decayh1n2} in the particular case $R=2^Mr$, from which the general case can then be deduced by easy means.\hfill $\square$

\section*{Acknowledgments}
The authors are thankful to Cristiana De Filippis for insightful discussions on the theme of the paper. F.G. moreover acknowledges funding by the Hector foundation, FP 626/21. M.S. wishes to thank the University of Konstanz for the funding of a research visit in February 2025, during which this project came to fruition.

\bibliographystyle{alpha}		
\bibliography{BGS}

\end{document}